\documentclass[11pt]{amsart}
\usepackage{graphicx}  
\usepackage{subfig}  
\usepackage{amssymb, amsmath, amsthm}
\usepackage[colorlinks=true,linkcolor=blue,citecolor=red]{hyperref}
\hypersetup{hypertexnames=false}
\usepackage{color}
\usepackage{float}
\usepackage{enumerate}
\usepackage{graphics}
\usepackage{epstopdf}
\usepackage{tikz}
\usetikzlibrary{arrows,decorations.markings}
\usepackage{cite}
\usepackage{epsfig}
\usepackage{amsmath}
\usepackage{amssymb}
\usepackage{amscd}
\usepackage{graphicx}
\usepackage{mathrsfs}
\usepackage{needspace}
\usepackage[T1]{fontenc}
\hypersetup{hypertexnames=false}

\makeatletter
\newcommand{\equationlabels}[1]{%
  \begingroup
  \edef\@currentlabel{\theequation}%
  \edef\@currentHref{equation.\theHequation}%
  #1%
  \endgroup}
\makeatother
\numberwithin{equation}{section}
\newtheorem{theorem}{Theorem}[section]
\newtheorem{corollary}[theorem]{Corollary}
\newtheorem{lemma}[theorem]{Lemma}
\newtheorem{proposition}[theorem]{Proposition}
\newtheorem{assumption}[theorem]{Assumption}
\newtheorem{Dbarproblem}{$\bar{\partial}$-Problem}[section]
\newtheorem{dbar-RHP}[theorem]{$\bar{\partial}$-RH problem}
\theoremstyle{definition}
\newtheorem{definition}[theorem]{Definition}
\newtheorem{remark}[theorem]{Remark}
\newtheorem{RHP}[theorem]{RH Problem}


\begin{document}
	
	\title[Long-time asymptotics for the Tzitz\'{e}ica equation]{On the Cauchy problem for the Tzitz\'{e}ica equation: Soliton resolution conjecture and asymptotic analysis}
	
	
		\author[S. F. Tian]{Shou-Fu Tian$^{*}$}

    \author[J. F. Tong]{Jia-Fu Tong$^{*,\dag}$}

	\address{Shou-Fu Tian  (Corresponding author) \newline
		School of Mathematics, China University of Mining and Technology, Xuzhou 221116, China}
\email{sftian@cumt.edu.cn}
\address{Jia-Fu Tong  (Corresponding author) \newline
		School of Mathematics, China University of Mining and Technology, Xuzhou 221116, China}
\email{jftong@cumt.edu.cn}	
	
	\thanks{$^{*}$Corresponding authors(sftian@cumt.edu.cn (S.F. Tian) and jftong@cumt.edu.cn (J.F. Tong)).\\
\hspace*{3ex}$^\dag$This author is contributed equally as the first author.}

\begin{abstract}
We prove the soliton resolution conjecture for the Tzitz\'eica equation, which is an important integrable model arising in affine differential geometry and characterizing proper affine spheres. Recently, Huang, Wang and Zhu (Math. Ann. \textbf{395}, 18 (2026)) reported the long-time asymptotic result for the Tzitz\'eica equation with purely continuous spectrum. Inspired by their work, we carry out a full asymptotic analysis of the Tzitz\'eica equation. Compared with the previous work, by developing the $\bar{\partial}$-nonlinear steepest descent method, we reveal that the solutions of the Tzitz\'eica equation exhibits three different asymptotic regions depending on $\xi:=x/t$. For the region $\xi\in(-\infty,-1]\cup[1,\infty)$, where there are no stationary phase points, we rigorously prove that the solution satisfies $u(x,t)=O_\xi(t^{-1})$ for $\xi\ge1$ and $u(x,t)=2\pi i+O_\xi(t^{-1})$ for $\xi\le-1$. In the region $\xi\in(-1,1)$, the phase function $\theta(z)$ has two stationary phase points. The corresponding asymptotic approximations can be characterized with an $N$-soliton solution as well as an interaction term between soliton solutions and the dispersion term with a residual error of order $O(t^{-3/4+b_K})$ with $0\le b_K<1/4$. In the transition regions $\{(x,t):1-\varepsilon<|\xi|<1\}$, the corresponding asymptotic approximation can be characterized by the dispersion term with a residual error of order $O(t^{-3/4+b_\varepsilon})$ where $0\le b_\varepsilon<1/4$. Our results provide a rigorous analysis of the long-time asymptotic behavior of solutions of the Tzitz\'eica equation in different regions and confirm the soliton resolution conjecture for the Tzitz\'eica equation.
\end{abstract}

	\maketitle
	\tableofcontents
	\section{Introduction}
In this work, we consider the Cauchy problem for the
Tzitz\'{e}ica equation:
\begin{equation}\label{TZI}
\begin{cases}
u_{tt}-u_{xx}=e^{-2u}-e^u,\\
u(x,0)=u_0(x),\qquad u_t(x,0)=u_1(x),
\end{cases}
\end{equation}
with finite density complex-valued initial data
\[
u_0-2\pi i\chi_-\in\mathcal S(\mathbb R;\mathbb C),
\qquad
u_1\in\mathcal S(\mathbb R;\mathbb C),
\]
where $\chi_-\in C^\infty(\mathbb R;\mathbb R)$ is a fixed
function equal to one on a left half-line and zero on a
right half-line, and
$
\mathcal S(\mathbb R;\mathbb C)
=
\left\{
f\in C^\infty(\mathbb R;\mathbb C):
\sup_{x\in\mathbb R}
|x^m\partial_x^n f(x)|<\infty
\text{ for all }m,n\in\mathbb N_0
\right\},
$
with $\mathbb N_0=\{0,1,2,\ldots\}$.
Thus $u_0(-\infty)=2\pi i$, $u_0(+\infty)=0$ and $u_1(\pm\infty)=0$. The Tzitz\'eica equation, admits both real-valued
and complex-valued formulations. In the real-valued
setting, it arises in affine differential geometry
in connection with proper affine spheres
\cite{BobenkoSchief1999,LoftinMcIntosh2016}
and provides an integrable scalar field model.
Its complex-valued extension has also attracted
attention in integrable field theory, particularly
in the study of complex solitons and parity-time
($\mathcal{PT}$) symmetry. Certain complex soliton
solutions have been shown to possess real energies
\cite{CorreaFringTaira2022a,CorreaFringTaira2022b}.

Equation \eqref{TZI} has its origins in the classical theory of surfaces in affine differential geometry. In a series of works published between 1907 and 1910 \cite{Tzitzica1907,Tzitzica1908,Tzitzica1910}, Tzitz\'eica investigated a distinguished class of hyperbolic surfaces in \(\mathbb{R}^3\) whose second fundamental form is non-degenerate. For such a surface, let \(K\) denote its Gauss curvature and let \(d\) denote the Euclidean distance from a fixed point to the tangent plane. Tzitz\'eica observed that the ratio
\[
I=\frac{K}{d^4}
\]
is invariant under equiaffine transformations of \(\mathbb{R}^3\). Surfaces for which \(I\) is constant are now called Tzitz\'eica surfaces \cite{Krivoshapko2015}. In the language of affine differential geometry, they are also known as proper affine spheres, because the affine distance from the origin is a nonzero constant \cite{Inoguchi2018}.

For an indefinite proper affine sphere with negative affine mean curvature, the Gauss--Codazzi equations with Blaschke metric $h=2e^udXdT$ in isothermal coordinates reduce to
\begin{equation}\label{TZI-original}
u_{X,T}=e^{u}-e^{-2u}.
\end{equation}
This is the Tzitz\'eica equation in its original form. Introducing the light-cone coordinates
\[
x=X+T,\qquad t=T-X,
\]
one obtains the equivalent form \eqref{TZI}.

The Tzitz\'{e}ica equation is not merely a geometric curiosity. In the theory of integrable systems, Mikhailov \cite{Mikhailov1981} showed that the Tzitz\'{e}ica \eqref{TZI} arises as a reduction of the periodic two-dimensional Toda lattice of type \(A_2^{(2)}\),
\[
u_{n,X,T}=e^{u_n-u_{n-1}}-e^{u_{n+1}-u_n},
\]
with period three under the constraint \(u_1=u\), \(u_2=-u\), and \(u_3=0\). This reduction is associated with the affine root system \(A_2^{(2)}\). Independently, as shown in \cite{Dunajski2009}, the same equation arises as a reduction of the anti-self-dual Yang--Mills equations on \(\mathbb{R}^{2,2}\) with gauge group \(SL(3,\mathbb{R})\).

Beyond its geometric origin, the Tzitz\'eica equation has appeared in several physical contexts. In one-dimensional gas dynamics, it was shown by Gaffet that for a particular class of gas laws, a \((1+1)\)-dimensional anisotropic gas dynamics system may be reduced to the Tzitz\'eica equation \cite{Schief1994}. More explicitly, when the equation of state takes the form \(P=\rho^3/M^4\), where \(P\), \(\rho\), and \(M\) denote the pressure, density, and Lagrangian mass, respectively, the system of Euler equations is reducible to the Tzitz\'{e}ica \eqref{TZI} \cite{ElKalaawy2002}. In nonlinear optics, the Tzitz\'eica--Dodd--Bullough (TDB) equation, a close relative of \eqref{TZI}, serves as a paradigmatic model for wave propagation with exponential nonlinearities in optical media and electromagnetic waves \cite{Khater2025}. In quantum field theory, the TDB model arises as an integrable affine Toda field theory and is connected to the Bethe ansatz description of massive quantum field theories on a cylinder \cite{Dorey2013}. Additional applications include the construction of (pseudo) hyper-complex metrics in four dimensions via solutions of \eqref{TZI} \cite{Dunajski2009}, and the interpretation of the elliptic Tzitz\'eica equation as a Taubes-type equation for vortices on constant mean curvature surfaces embedded in \(\mathbb{R}^{2,1}\) \cite{Dunajski2012}. These diverse realizations underscore the significance of \eqref{TZI} beyond its geometric origin.

The Tzitz\'eica equation also has a natural
interpretation in integrable field theory as the complex
Bullough--Dodd model, associated with the twisted affine algebra
$A_2^{(2)}$. Allowing the field to be complex gives rise to
nontrivial topological soliton sectors
\cite{AssisFerreira2008}.
Complex shifts can also regularize otherwise singular soliton
profiles. In particular, Correa, Fring, and Taira
\cite{CorreaFringTaira2022a} investigated complex soliton families
with real total energies and analysed their linear stability,
including stable $\mathcal{PT}$-symmetric solutions.
Their subsequent work \cite{CorreaFringTaira2022b}
studied the collective dynamics of complex solitons, including
spatial displacements and time delays in soliton scattering.
These developments provide a physical motivation for studying
complex-valued solutions of \eqref{TZI}, particularly their
long-time evolution and the interaction between solitons
and dispersive radiation.

The Tzitz\'eica equation \eqref{TZI} and the sine-Gordon equation are two typical examples of integrable systems arising from classical differential geometry \cite{Udriste2011}. The sine-Gordon equation appears as the Gauss--Codazzi equation for surfaces with constant Gauss curvature \(-1\) in \(\mathbb{R}^3\), and hence describes pseudo-spherical surfaces. It possesses a \(2\times 2\) Lax pair and has been extensively studied by the Riemann--Hilbert (RH) method \cite{Cheng1999,HuangLenells2018}. Lu and Miller \cite{LuMiller2022} studied the Dubrovin universality near the gradient catastrophe point in the semiclassical sine-Gordon equation. By contrast, the Tzitz\'eica equation \eqref{TZI} admits a \(3\times 3\) Lax pair \cite{Mikhailov1981}. This \(3\times 3\) spectral problem introduces substantial difficulties in the development of the inverse spectral theory, including two reflection coefficients, a spectral plane divided into six sectors, and the possible presence of discrete eigenvalues and solitons. Significant progress has been made for integrable systems with third-order Lax pairs by Charlier, Lenells, and coauthors \cite{CharlierLenells2021,CharlierLenells2022,CharlierLenellsWang2023}. Related RH analyses include the Degasperis--Procesi equation \cite{Boutet2013,Boutet2019}, the ``good'' and ``bad'' Boussinesq equations \cite{Charlier2024,CharlierLenells2024,CharlierLenells2025,DeiftTomeiTrubowitz1982}, and the Lenells equation \cite{CharlierLenells2021}.

The integrable structure of \eqref{TZI} has been extensively investigated since the late 1970s. Dodd and Bullough \cite{Dodd1977} discovered nontrivial conservation laws, and Zhiber and Shabat \cite{Zhiber1979} showed that the equation admits an infinite Lie--B\"acklund symmetry group; for this reason it is sometimes also referred to as the Bullough--Dodd--Zhiber--Shabat equation. Mikhailov \cite{Mikhailov1981} discovered a Lax pair for \eqref{TZI}, which opened the way to the construction of various exact solutions, including soliton solutions, finite-gap solutions, and algebro-geometric solutions \cite{Brezhnev1996,Conte1999,Kaptsov1999,Wu2015}. Various solution-construction techniques have been developed, including B\"acklund transformations \cite{Conte1999,Safin1993}, Darboux transformations \cite{Zhu2006,Nimmo1997}, the dressing method \cite{Babalic2015}, and the Hirota method \cite{Abazari2010}. Babalic, Constantinescu, and Gerdjikov \cite{Babalic2015} studied the spectral properties of the Lax operator associated with \eqref{TZI} and proved that the continuous spectrum of the Lax operator is rotated with respect to the contour of the RH problem by an angle of \(\pi/6\), and that the poles of the dressing factors correspond to discrete eigenvalues.

The soliton resolution conjecture is one of the central open problems in the theory of nonlinear dispersive equations. It asserts that for a large class of integrable equations, the solution of the Cauchy problem with generic initial data decomposes asymptotically into a finite number of solitons plus a radiation term as time tends to infinity \cite{Tao2009,Jendrej2025}. In the integrable setting, where tools such as the inverse scattering transform and the RH method are available, the conjecture is well understood for several models, including the Korteweg--de Vries equation \cite{Eckhaus1983}, the nonlinear Schr\"odinger equation \cite{Borghese2018}, and the derivative nonlinear Schr\"odinger equation \cite{Jenkins2019}. For the sine-Gordon equation, Cheng, Venakides, and Zhou \cite{Cheng1999} investigated the long-time asymptotics of pure-radiation solutions. More recently, Charlier and Lenells \cite{CharlierLenells2024} studied the soliton resolution conjecture for the Boussinesq equation and obtained a complete asymptotic description. Moreover, this method has been widely applied to other integrable systems \cite{boo-36,boo-61,boo-64,boo-65,boo-66,boo-68,boo-47,boo-477,boo-Tian,Novikov2023,DP2025}. However, for the Tzitz\'eica equation \eqref{TZI}, the soliton resolution conjecture has not yet been investigated.

Recently, Huang, Wang, and Zhu \cite{HWW2026} studied the long-time asymptotic result for real-valued solutions for the Tzitz\'eica equation \eqref{TZI} with purely continuous spectrum. By employing the RH method and the nonlinear steepest descent method of Deift and Zhou \cite{DeiftZhou1993}, they derived asymptotic formulas in various regions of the \(x\)-\(t\) half-plane and verified their predictions numerically. This work provided the long-time asymptotic result of real-valued pure-radiation solutions of \eqref{TZI}. In the complex-valued problem, the two reflection coefficients
are generally independent, and the local spectral exponents
need not be real. Some natural problems deserve further
 consideration for the Tzitz\'eica equation \eqref{TZI}: how to study long-time asymptotic of the Tzitz\'eica
 equation and how to prove the soliton resolution conjecture of the Tzitz\'eica equation?

In this work, through the $\bar{\partial}$ nonlinear steepest descent method \cite{boo-52,boo-53}, we conduct an in-depth study of the soliton resolution conjecture and the asymptotic analysis of the Tzitz\'eica equation \eqref{TZI}. Our analysis begins with the direct and inverse scattering problems associated with the \(3\times 3\) spectral problem \cite{BealsCoifman1984,BealsDeiftTomei1988}, leading to an inverse scattering transform framework for the initial-value problem. Within this framework, the solution of \eqref{TZI} is formulated in terms of a \(3\times 3\) matrix RH problem. To establish these results, we reconstruct the solution $u(x,t)$ from the solution of an associated RH problem and analyze the latter by means of the $\bar{\partial}$ nonlinear steepest descent method. This approach has been successfully applied to the long-time asymptotic analysis of the initial value problem for the nonlinear Schr{\"o}dinger equation \cite{boo-80}. It has also provided an effective framework for proving the soliton resolution conjecture for a broad class of nonlinear integrable systems \cite{Borghese2018,boo-Tian,boo-54,boo-55,boo-57}.

\subsection{Plan of the proofs}

The remainder of the work is organized as follows.
Section~\ref{s:2} develops the spectral analysis of the
complex-valued problem and formulates the meromorphic RH problem,
including the two independent reflection functions, the two seed
residue families, the surviving reduction symmetries, and the
reconstruction formula at $\lambda=0$.
Section~\ref{s:3} analyzes the phase functions and discrete
velocities and introduces the diagonal conjugation based on
a continuous logarithm of $1-\alpha\beta$, together with the
pole-removing transformations.
Section~\ref{s:4} constructs the lens extensions and decomposes
the resulting mixed problem into a pure RH problem and a
$\bar\partial$ problem, accounting for the endpoint powers
associated with the complex exponent $\nu$.

The next three sections evaluate the contributions of the
discrete spectrum, the local jumps, and the nonanalytic remainder.
Section~\ref{s:discrete-contribution} defines the retained
reflectionless model with modified norming data, studies its
finite residue system, provides a regular single-soliton example,
and estimates the error from removing the remaining complete orbits.
Section~\ref{s:jump-contribution} constructs the two-parameter
complex parabolic-cylinder parametrices at the six stationary
points and analyzes the associated small-norm RH problem to
determine the leading radiation matrices.
Section~\ref{s:dbar-contribution} treats the $\bar\partial$
factor through its Cauchy--Green integral equation, derives
reconstruction estimates in the exterior and fixed interior
regions, and states the uniform endpoint estimates near
the light cone.
Finally, Section~\ref{s:reconstruction-final} reverses the
transformations at $\lambda=0$ and tracks the continuous
logarithmic branches in the asymptotic formulas of
Theorems~\ref{thm:soliton-resolution}.

\section{Inverse scattering transform and RH problem}
\label{s:2}

We retain the six-ray spectral construction and the general finite
discrete-spectrum framework, but impose no reality reduction.

\subsection{The Lax pair and spectral analysis}
\label{spectralanalysis}

Throughout this section the field is complex. At the reference time zero its Cauchy
data belong to the charge-one topological Schwartz class
\begin{equation}\label{sec2:complex-cauchy-class}
\mathcal X_1^{\mathcal S}
=\{(u_0,u_1):u_0-2\pi i\chi_-\in\mathcal S(\mathbb R;\mathbb C),\quad
u_1\in\mathcal S(\mathbb R;\mathbb C)\},
\end{equation}
where $\chi_-$ is a fixed smooth function equal to one on a left
half-line and zero on a right half-line. Thus $u_0(-\infty)=2\pi i$,
$u_0(+\infty)=0$, whereas $q_0=e^{u_0}$ tends to one at both ends.
The direct-scattering conclusions below do not assert global
regularity for every datum in \eqref{sec2:complex-cauchy-class}.

The Tzitz\'eica equation \eqref{TZI} admits the $3\times3$ Lax pair
\begin{equation}
\phi_x(x,t,\lambda)
=
L(x,t,\lambda)\phi(x,t,\lambda),
\qquad
\phi_t(x,t,\lambda)
=
Z(x,t,\lambda)\phi(x,t,\lambda),
\nonumber
\end{equation}
where
\begin{equation}
L(x,t,\lambda)
=
\frac{\lambda}{2}J+U_0+\frac{1}{\lambda}U_1,\qquad
Z(x,t,\lambda)
=
\frac{\lambda}{2}J+U_0-\frac{1}{\lambda}U_1.
\nonumber
\end{equation}
Here
\begin{equation*}
J
=
\begin{pmatrix}
\omega&0&0\\
0&\omega^2&0\\
0&0&1
\end{pmatrix},
\qquad
U_0
=
\frac{i\sqrt3(u_x+u_t)}{6}
\begin{pmatrix}
0&1&-1\\
-1&0&1\\
1&-1&0
\end{pmatrix},
\qquad
\omega=e^{2\pi i/3},
\end{equation*}
\begin{equation*}
U_1
=
\frac16
\begin{pmatrix}
\omega^2(2e^u+e^{-2u})
&
e^{-2u}-e^u
&
\omega(e^{-2u}-e^u)
\\
e^{-2u}-e^u
&
\omega(2e^u+e^{-2u})
&
\omega^2(e^{-2u}-e^u)
\\
\omega(e^{-2u}-e^u)
&
\omega^2(e^{-2u}-e^u)
&
2e^u+e^{-2u}
\end{pmatrix}.
\end{equation*}
The eigenvalues associated with the Lax pair are
\begin{equation}
l_j(\lambda)
=
\frac{\omega^j\lambda+(\omega^j\lambda)^{-1}}{2},
\qquad
z_j(\lambda)
=
\frac{\omega^j\lambda-(\omega^j\lambda)^{-1}}{2},
\qquad
j=1,2,3.
\nonumber
\end{equation}
The corresponding phase functions are
$
\vartheta_{ij}(x,t,\lambda)
=
\bigl(l_i(\lambda)-l_j(\lambda)\bigr)x
+
\bigl(z_i(\lambda)-z_j(\lambda)\bigr)t.
$
The RH jump contour is
\begin{equation}\label{Sigmadef}
\Sigma
=\mathbb R\cup\omega\mathbb R\cup\omega^2\mathbb R.
\end{equation}
Here $\Sigma$ is determined by
the equal-real-part conditions
$\Re(l_i-l_j)=0$, $i\ne j$.
This convention concerns the sectorial RH formulation
and is not an identification with the $L^2$
continuous spectrum of the spatial operator pencil.
The contour $\Sigma$ divides the spectral plane
into the six open sectors
\begin{equation}
D_j
=
\left\{
\lambda\in\mathbb C:
\frac{(j-1)\pi}{3}
<
\arg\lambda
<
\frac{j\pi}{3}
\right\},
\qquad
j=1,\ldots,6.
\nonumber
\end{equation}

We use the cyclic permutation matrix
$
\mathcal A=\begin{pmatrix}0&1&0\\0&0&1\\1&0&0\end{pmatrix}.
$
We also retain
$\mathcal B=\begin{pmatrix}0&1&0\\1&0&0\\0&0&1\end{pmatrix}$. The cofactor matrix of a determinant-one matrix $X$ is
$X^A=(X^{-1})^T$. Write $\mathcal L=\operatorname{diag}(l_1,l_2,l_3)$,
$\mathcal Z=\operatorname{diag}(z_1,z_2,z_3)$ and $W=L-\mathcal L$.
At time zero, $W(\cdot,\lambda)$ is Schwartz for every $\lambda\ne0$.
The actual normalized right column and left adjoint column satisfy
\begin{align}
v_{+,j}(x,\lambda)
&=e_j-\int_x^\infty e^{(x-y)(\mathcal L-l_jI)}
 W(y,\lambda)v_{+,j}(y,\lambda)\,dy,
\label{sec2:actual-right-volterra}\\
v_{-,i}^{A}(x,\lambda)
&=e_i-\int_{-\infty}^x e^{(x-y)(-\mathcal L+l_iI)}
 W(y,\lambda)^T v_{-,i}^{A}(y,\lambda)\,dy.
\label{sec2:actual-left-adjoint}
\end{align}

On the positive ray, indices one and two tie for the minimal
real part, so all four columns used in
\begin{equation}\label{sec2:actual-scattering-pairings}
s_{ij}(\lambda)
=v_{-,i}^{A}(0,\lambda)^T v_{+,j}(0,\lambda),
\qquad i,j\in\{1,2\},\quad \lambda>0,
\end{equation}
are legal. Cofactor scattering entries are constructed from the dual
columns on their own domains. Whenever a full scattering matrix is
available, this convention agrees with $s^A=(s^{-1})^T$.

We give an explicit column construction of the mixed eigenfunction.
Initially set $t=0$. In $D_1$ one has
$\Re l_1<\Re l_2<\Re l_3$. Define the additional legal columns
\begin{align}
v_{-,3}(x,\lambda)
&=e_3+\int_{-\infty}^{x}
 e^{(x-y)(\mathcal L-l_3I)}
 W(y,\lambda)v_{-,3}(y,\lambda)\,dy,
\label{sec2:left-third-column}\\
v_{+,3}^{A}(x,\lambda)
&=e_3+\int_x^{\infty}
 e^{(x-y)(-\mathcal L+l_3I)}
 W(y,\lambda)^T v_{+,3}^{A}(y,\lambda)\,dy.
\label{sec2:right-adjoint-third-column}
\end{align}
All four kernels are bounded in their respective integration
directions. Volterra iteration gives columns analytic in $D_1$. For the normalized Jost and adjoint columns, we use the notation
$
(\Phi_\pm)_{\cdot j}=v_{\pm,j},
(\Phi_\pm^A)_{\cdot j}=v_{\pm,j}^A
$
on their respective legal column domains.
These symbols are understood columnwise; the existence of a
full Jost matrix on a common sector is not required. In $D_1$, the ordering
$\Re l_1<\Re l_2<\Re l_3$
makes the columns
$(\Phi_+)_{\cdot1}$, $(\Phi_-^A)_{\cdot1}$,
$(\Phi_+^A)_{\cdot3}$ and $(\Phi_-)_{\cdot3}$
well-defined and analytic.
The analytic transmission functions are
\begin{equation}\label{sec2:sector-transmissions}
s_{11}(\lambda)
=(\Phi_-^A)_{\cdot1}^{T}
  (\Phi_+)_{\cdot1},
s_{33}^A(\lambda)
=(\Phi_+^A)_{\cdot3}^{T}
  (\Phi_-)_{\cdot3}.
\end{equation}
Both pairings are independent of $x$. For $\lambda\in D_1$,
the mixed sectorial eigenfunction is defined by
\begin{equation}\label{sec2:basic-sector-columns}
M_1(x,0,\lambda)=
\begin{pmatrix}
(\Phi_+)_{11}
&
\dfrac{
(\Phi_+^A)_{23}(\Phi_-^A)_{31}
-(\Phi_+^A)_{33}(\Phi_-^A)_{21}
}{s_{11}}
&
\dfrac{(\Phi_-)_{13}}{s_{33}^A}
\\[4mm]
(\Phi_+)_{21}
&
\dfrac{
(\Phi_+^A)_{33}(\Phi_-^A)_{11}
-(\Phi_+^A)_{13}(\Phi_-^A)_{31}
}{s_{11}}
&
\dfrac{(\Phi_-)_{23}}{s_{33}^A}
\\[4mm]
(\Phi_+)_{31}
&
\dfrac{
(\Phi_+^A)_{13}(\Phi_-^A)_{21}
-(\Phi_+^A)_{23}(\Phi_-^A)_{11}
}{s_{11}}
&
\dfrac{(\Phi_-)_{33}}{s_{33}^A}
\end{pmatrix}.
\end{equation}

The remaining sectors are generated by the cyclic and
inverse-transpose reductions:
\begin{equation}\label{sec2:all-sector-generation}
M(x,0,\lambda)=
\begin{cases}
\mathcal A^j
M_1(x,0,\omega^{-j}\lambda)
\mathcal A^{-j},
&\lambda\in\omega^jD_1,
\\[2mm]
\mathcal A^j
M_1(x,0,-\omega^{-j}\lambda)^{-T}
\mathcal A^{-j},
&\lambda\in-\omega^jD_1,
\end{cases}
\qquad j=0,1,2.
\end{equation}
The same column formulas extend to any smooth existence
interval on which the Cauchy solution retains the stated
spatial decay, using the time part of the Lax pair.

The algebraic symmetries of the Lax matrices that survive for complex
$u$ are the cyclic symmetry and
\begin{equation}\label{Laxcofactor}
L(-\lambda)=-L(\lambda)^T,\qquad
Z(-\lambda)=-Z(\lambda)^T.
\end{equation}
Here $J^T=J$, $U_0^T=-U_0$ and $U_1^T=U_1$. Direct substitution also
gives $L(\omega\lambda)=\mathcal A L(\lambda)\mathcal A^{-1}$ and
the corresponding identity for $Z$.

\begin{lemma}\label{lem:cofactor}
On their corresponding legal column domains the Jost solutions and
scattering data obey
\begin{subequations}
\begin{align}
\Phi_\pm(x,t,-\lambda)&=\Phi_\pm(x,t,\lambda)^A,
\label{Jostcofactor}\\
s(-\lambda)&=s^A(\lambda).
\label{scofactor}
\end{align}
\end{subequations}
The mixed sectorial eigenfunctions satisfy
$
M(x,t,\omega\lambda)=\mathcal A M(x,t,\lambda)\mathcal A^{-1},
 M(x,t,-\lambda)=M(x,t,\lambda)^{-T}.
$
The first identity above for $\Phi_\pm$ is understood columnwise
through ordinary and adjoint solutions when a full normalized Jost
matrix has no common domain.
\end{lemma}

We therefore use two independent reflection coefficients
\begin{equation}\label{sec2:two-reflections}
\alpha(\lambda)=\frac{s_{12}(\lambda)}{s_{11}(\lambda)},\qquad
\beta(\lambda)=\frac{s_{21}(\lambda)}{s_{22}(\lambda)},
\qquad \lambda>0.
\end{equation}
On the negative ray the corresponding cofactor quotients
are generated by \eqref{scofactor}:
\begin{equation}\label{reflectioncofactorsymmetry}
\frac{s_{12}^{A}(-\lambda)}{s_{11}^{A}(-\lambda)}=\alpha(\lambda),
\qquad
\frac{s_{21}^{A}(-\lambda)}{s_{22}^{A}(-\lambda)}=\beta(\lambda),
\qquad \lambda>0.
\end{equation}

Define the flat Schwartz space
\begin{equation}\label{sec2:flat-spectral-space}
\mathcal F_{\mathrm{flat}}=\left\{f\in C^\infty(\mathbb R_+):
\sup_{s>0}(s^m+s^{-m})|\partial_s^j f(s)|<\infty
\text{ for every }m,j\ge0\right\}.
\end{equation}
\begin{proposition}
\label{prop:reflection-positivity}
For data in \eqref{sec2:complex-cauchy-class}, the four pairings
\eqref{sec2:actual-scattering-pairings} are smooth on $\mathbb R_+$.
For every $m,j\ge0$,
\[
|\partial_\lambda^j s_{12}|+|\partial_\lambda^j s_{21}|
\le C_{m,j}\begin{cases}\lambda^m,&0<\lambda\le1,\\
\lambda^{-m},&\lambda\ge1.\end{cases}
\]
Moreover $s_{11},s_{22}=1+O(\lambda)$ at zero and
$1+O(\lambda^{-1})$ at infinity, with differentiated symbol bounds.
If $s_{11}(\lambda)s_{22}(\lambda)\ne0$ for every $\lambda>0$, then
$
\alpha,\beta\in\mathcal F_{\mathrm{flat}}\subset H^{1,1}(\mathbb R_+).
$
\end{proposition}

\begin{proof}
Set
$
H_0=\frac{i\sqrt3}{3}(\mathcal A-\mathcal A^2),
\mathcal G_0=e^{u_0H_0}.
$
Since $H_0^T=-H_0$ and $e^{2\pi iH_0}=I$,
the prescribed spatial tails imply
\[
\mathcal G_0^{\pm1}-I,\quad
u_0'\pm u_1\in\mathcal S(\mathbb R),
\qquad
\mathcal G_0(\pm\infty)=I.
\]
With $\mu=\lambda^{-1}$ and
$\widehat Y=\mathcal G_0^{-1}Y$, the spatial
Lax equation becomes
\begin{equation}\label{sec2:small-endpoint-gauge}
\widehat Y_x=
\left(
\frac{\mu}{2}J^{-1}
+\frac{u_1-u_0'}{2}H_0
+\frac{\mathcal G_0^{-1}J\mathcal G_0}{2\mu}
\right)\widehat Y.
\end{equation}
Thus both endpoints reduce to systems of the form
\[
Y_x=\mathscr A Y,\qquad
\mathscr A=zD+z^{-1}D_-+V_0+z^{-1}V_{-1},
\]
where $D,D_-$ are diagonal, $D$ has distinct
eigenvalues, and $V_0,V_{-1}$ are Schwartz,
with $\operatorname{diag}V_0=0$.

Successive diagonalization, using the invertibility
of $\operatorname{ad}_D$ on off-diagonal matrices,
gives
\[
P_N=I+\sum_{n=1}^{N}z^{-n}P_n,\qquad
B_N=\sum_{n=1}^{N-1}z^{-n}B_n,
\]
with off-diagonal Schwartz $P_n$ and diagonal
Schwartz $B_n$, such that
\[
P_N^{-1}\mathscr A P_N-P_N^{-1}(P_N)_x
=zD+z^{-1}D_-+B_N+E_N.
\]
For any prescribed $K,m,k$, sufficiently large $N$
ensures
\begin{equation}\label{sec2:true-normal-residual}
\sum_{0\le a,b\le k}
\|\langle x\rangle^m
\partial_x^a\partial_z^bE_N\|_{L^1}
\le C_{K,m,k}|z|^{-K},
\end{equation}
while
\[
\|\partial_z^bB_N\|_{L^1\cap L^\infty}
=O(|z|^{-1-b}).
\]
Moreover, $P_N^{\pm1}$ are uniformly bounded
for large $|z|$.

For the legal columns in
\eqref{sec2:actual-scattering-pairings},
the diagonal Volterra kernels are bounded;
their spectral derivatives introduce only
polynomial factors in the integration variables.
The weighted residual estimates therefore give
convergent Neumann series in polynomially weighted
spaces. At $x=0$, the actual ordinary and adjoint
columns are approximated, with any prescribed
number of spectral derivatives, by the
$P_N$- and $P_N^{-T}$-transforms of the normalized
diagonal columns, up to $O(|z|^{-K})$.
These normal-form factors cancel in the bilinear
pairings. Consequently, for $i,j\in\{1,2\}$,
\[
\partial_z^r s_{ij}(z)=O(|z|^{-K}),
\qquad i\ne j,
\]
and
\[
s_{ii}(z)
=\exp\left(
-\int_{\mathbb R}(B_N)_{ii}(x,z)\,dx
\right)+O(|z|^{-K})
=1+O(z^{-1}),
\]
with differentiated symbol bounds.

Apply this argument with $z=\lambda$, $D=J/2$,
and with $z=\mu=\lambda^{-1}$, $D=J^{-1}/2$,
using \eqref{sec2:small-endpoint-gauge}.
The gauge tends to $I$ at both spatial ends,
so it preserves the Jost normalizations and
ordinary--adjoint pairings.
The chain rule gives arbitrary-order flatness
at $\lambda=0$.
Smoothness on compact positive spectral intervals
follows from the differentiated Volterra equations.
Since $s_{11}s_{22}$ is nonvanishing, quotient
differentiation yields
\[
\alpha,\beta\in\mathcal F_{\mathrm{flat}}.
\]
Their zero extensions are Schwartz, and hence
$\alpha,\beta\in H^{1,1}(\mathbb R_+)$.
\end{proof}

\begin{assumption}\label{ass:discrete}
The scattering data satisfy the following conditions.

\begin{enumerate}
\item
On the positive real axis,
\begin{equation}\label{sec2:continuous-nondegeneracy}
s_{11}(s)s_{22}(s)\ne0,
\qquad
\delta(s):=1-\alpha(s)\beta(s)\ne0,
\qquad
\|\alpha\beta\|_{L^\infty}\le\rho_*<1,
s>0.
\end{equation}

\item
The discrete spectrum consists of simple zeros of
the analytic denominators $s_{11}$ and $s_{33}^A$
in $D_1$, located at
\[
\zeta_n=\rho_ne^{i\theta_n},
\qquad
\eta_n=e^{\pi i/3}\overline{\zeta_n},
\qquad n=1,\ldots,N,
\]
where $\rho_n>0$, $0<\theta_n<\pi/3$,
$\theta_n\ne\pi/6$, and
\[
s_{11}(\zeta_n)=0,\quad
s_{11}'(\zeta_n)\ne0,
\qquad
s_{33}^A(\eta_n)=0,\quad
(s_{33}^A)'(\eta_n)\ne0.
\]
\end{enumerate}
\end{assumption}

Proposition~\ref{prop:sec5:regular-one-orbit-example} states that the above assumption exists and is reasonable.

\begin{definition}\label{def:ray-admissible}
For a fixed observation velocity $\xi\in(-1,1)$, let $\Lambda(\xi)$
denote the retained orbit indices in the velocity classification.
Write their modified paired data as
\[
\widetilde{\mathcal D}_\Lambda(\xi)
=\{(\zeta_n,\eta_n,
\widetilde c_n^\Lambda(\xi),\widetilde d_n^\Lambda(\xi)):
n\in\Lambda\}.
\]
\end{definition}

Set
\begin{equation}\label{betandefnew}
\beta_n(x,t)=-c_n e^{\vartheta_{12}(x,t,\zeta_n)},\qquad
\beta_n^{\sharp}(x,t)=-d_n e^{\vartheta_{23}(x,t,\eta_n)}.
\end{equation}
The discrete amplitude $\beta_n$ is distinct from the reflection
function $\beta(s)$. The superscript $\sharp$ labels the second
independent seed and imposes no conjugation relation. To preserve the orbit indexing, use
\begingroup
\begin{equation}\label{sec2:paired-orbit-indexing}
\begin{aligned}
\zeta_{n+N}&=-\omega^2\eta_n,&
\zeta_{n+2N}&=\omega\zeta_n,&
\zeta_{n+3N}&=-\eta_n,\\
\zeta_{n+4N}&=\omega^2\zeta_n,&
\zeta_{n+5N}&=-\omega\eta_n,&
\zeta_{n+(j+6)N}&=-\zeta_{n+jN}\quad(0\le j\le5).
\end{aligned}
\end{equation}
In particular $\zeta_{n+9N}=\eta_n=:\zeta_n^\sharp$. The complete
pole set is
\[
\mathcal O_n=\{\pm\omega^j\zeta_n,\ \pm\omega^j\eta_n:j=0,1,2\},
\qquad \mathcal P_N=\bigcup_{n=1}^N\mathcal O_n
=\{\zeta_j\}_{j=1}^{12N},\qquad \mathcal P:=\mathcal P_N.
\]
\begin{figure}[htbp]
\centering
\begin{tikzpicture}[line width=0.7pt, font=\small]
  \foreach \angle in {0,60,120} {
    \draw (\angle:3.55) -- (\angle+180:3.55);
  }
  \node[anchor=west] at (0:3.65) {$\mathbb R$};
  \node[anchor=south east] at (120:3.65) {$\omega\mathbb R$};
  \node[anchor=south west] at (60:3.65) {$\omega^2\mathbb R$};
  \node[anchor=north, inner sep=2pt] at (0,0) {$0$};

  \foreach \j/\angle in {1/30,2/90,3/150,4/210,5/270,6/330} {
    \node at (\angle:1.25) {$D_{\j}$};
  }
  \draw[blue!70!black, dashed, line width=0.75pt] (0,0) circle (1.65);
  \node[blue!70!black, anchor=north west, inner sep=2pt]
    at (1.67,-0.03) {$|\lambda|=1$};

  \foreach \angle in {20,40,80,100,140,160,200,220,260,280,320,340} {
    \fill (\angle:2.45) circle (1.6pt);
  }
    \node[anchor=west]       at (20:2.58)  {$\zeta_n$};
  \node[anchor=south west] at (40:2.58)  {$\eta_n$};

  \node[anchor=south west] at (80:2.58)  {$-\omega^2\zeta_n$};
  \node[anchor=south east] at (100:2.58) {$-\omega^2\eta_n$};

  \node[anchor=south east] at (140:2.58) {$\omega\zeta_n$};
  \node[anchor=east]       at (160:2.58) {$\omega\eta_n$};

  \node[anchor=east]       at (200:2.58) {$-\zeta_n$};
  \node[anchor=north east] at (220:2.58) {$-\eta_n$};

  \node[anchor=north east] at (260:2.58) {$\omega^2\zeta_n$};
  \node[anchor=north west] at (280:2.58) {$\omega^2\eta_n$};

  \node[anchor=north west] at (320:2.58) {$-\omega\zeta_n$};
  \node[anchor=west]       at (340:2.58) {$-\omega\eta_n$};
\end{tikzpicture}
\caption{The spectral contour $\Sigma$, the sectors $D_j$, and a generic
twelve-point orbit $\mathcal O_n$ in the $\lambda$-plane, with
$\zeta_n^\sharp=-\omega^2\bar\zeta_n$. The dashed circle denotes
$|\lambda|=1$; the radial placement of the orbit is schematic.}
\label{fig:discrete-spectrum-orbit}
\end{figure}
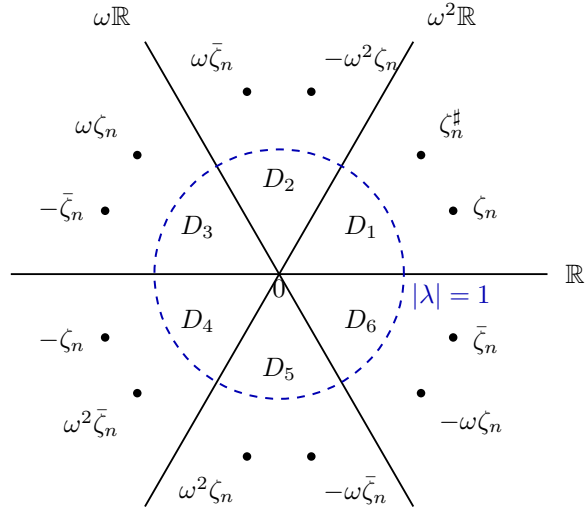
Within $D_1$, $\zeta_n$ and $\eta_n$ have arguments $\theta_n$ and
$\pi/3-\theta_n$. They coincide only at $\theta_n=\pi/6$;
the remaining cyclic and negative images lie in the other sectors.
The angular exclusion therefore gives twelve distinct points.
The two $D_1$ types have index sets
\[
\widetilde{\mathcal N}=\{1,\ldots,N\},\qquad
\widetilde{\mathcal N}^{A}=\{n+9N:1\le n\le N\}.
\]
Their positions are paired by assumption; their norming data remain
independent. Figure~\ref{fig:discrete-spectrum-orbit} illustrates this
position geometry, without imposing a reality reduction on residues.



For $E_{ij}=e_i e_j^T$, the two seed residues are
$B_{\zeta_n}=\beta_nE_{12}$ and $B_{\eta_n}=\beta_n^{\sharp}E_{23}$.
All other residues are generated only by
\begin{equation}\label{residueZ3new}
B_{\omega p}=\omega\mathcal A B_p\mathcal A^{-1},\qquad
B_{-p}=B_p^T.
\end{equation}
The complete first half of each indexed orbit is
\begin{subequations}\label{firstsixresiduematricesnew}
\begin{align}
B_n&=\beta_nE_{12},&
B_{n+N}&=\omega^2\beta_n^{\sharp}E_{13},&
B_{n+2N}&=\omega\beta_nE_{31},\\
B_{n+3N}&=\beta_n^{\sharp}E_{32},&
B_{n+4N}&=\omega^2\beta_nE_{23},&
B_{n+5N}&=\omega\beta_n^{\sharp}E_{21},\\
B_{n+6N}&=\beta_nE_{21},&
B_{n+7N}&=\omega^2\beta_n^{\sharp}E_{31},&
B_{n+8N}&=\omega\beta_nE_{13},\\
B_{n+9N}&=\beta_n^{\sharp}E_{23},&
B_{n+10N}&=\omega^2\beta_nE_{32},&
B_{n+11N}&=\omega\beta_n^{\sharp}E_{12},\qquad 1\le n\le N.
\end{align}
\end{subequations}

\subsection{Set up of RH problem}
\label{setuprhp}

We now formulate the meromorphic RH problem associated with the scattering data introduced above. The contour is $\Sigma=\mathbb R\cup\omega\mathbb R\cup\omega^2\mathbb R$, with every half-ray oriented
from zero to infinity.

The behavior of the RH solution at the origin is determined by the matrix
\begin{equation}\label{G}
G(q)=\frac13\left\{(1+q+q^{-1})I
 +(1+\omega q+\omega^2q^{-1})\mathcal A
 +(1+\omega^2q+\omega q^{-1})\mathcal A^2\right\},\quad q\ne0.
\end{equation}
\begin{RHP}\label{RHP:basic-meromorphic}
Find a matrix valued function $M(x,t,\lambda)$ admits:
\begin{enumerate}[($i$)]
\item Analyticity:~$M$ is meromorphic in $\mathbb C\setminus\Sigma$ with only simple poles at the points of $\mathcal P_N$.
\item Symmetry:~$M(\omega\lambda)=\mathcal A M(\lambda)\mathcal A^{-1},\qquad
M(-\lambda)=M(\lambda)^{-T}$.
\item The boundary values satisfy $M_+=M_-V$, where
\begin{equation}\label{jump-matrix}
V(\lambda)=\begin{cases}
V^0(\lambda),&\lambda\in\mathbb R,\\
\mathcal A V^0(\omega^2\lambda)\mathcal A^{-1},
 &\lambda\in\omega\mathbb R,\\
\mathcal A^2 V^0(\omega\lambda)\mathcal A^{-2},
 &\lambda\in\omega^2\mathbb R.
\end{cases}
\end{equation}
On the positive outward ray,
\begin{equation}\label{V0}
V^0(\lambda)=\begin{pmatrix}
1&-\alpha(\lambda)e^{\vartheta_{12}(\lambda)}&0\\
\beta(\lambda)e^{-\vartheta_{12}(\lambda)}&1-\alpha(\lambda)\beta(\lambda)&0\\
0&0&1
\end{pmatrix},\qquad \lambda>0.
\end{equation}
On the negative outward ray,
\begin{equation}\label{V0-negation}
V^0(-\lambda)=V^0(\lambda)^{-T}
=\begin{pmatrix}
1-\alpha(\lambda)\beta(\lambda)&-\beta(\lambda)e^{-\vartheta_{12}(\lambda)}&0\\
\alpha(\lambda)e^{\vartheta_{12}(\lambda)}&1&0\\0&0&1
\end{pmatrix}.
\end{equation}
\item Asymptotic behavior:
\begin{equation}\nonumber
M(x,t,\lambda)=I+O(\lambda^{-1}),\quad\lambda\to\infty.
\end{equation}
\begin{equation}\label{basicoriginbehavior}
M(x,t,\lambda)=G(q(x,t))+O(\lambda),\qquad \lambda\to0.
\end{equation}

\item For each $n$ the seed and cyclic residue conditions are
\begin{subequations}\label{basicresiduecondition}
\begin{equation}
\begin{aligned}
\operatorname*{Res}_{\lambda=\omega^j\zeta_n}M(\lambda)
 &=\lim_{\lambda\to\omega^j\zeta_n}
 M(\lambda)\,\omega^j\mathcal A^j
       (\beta_nE_{12})\mathcal A^{-j},\\
\operatorname*{Res}_{\lambda=\omega^j\eta_n}M(\lambda)
 &=\lim_{\lambda\to\omega^j\eta_n}
 M(\lambda)\,\omega^j\mathcal A^j
       (\beta_n^{\sharp}E_{23})\mathcal A^{-j},\quad j=0,1,2.
\end{aligned}
\end{equation}
Their negative partners satisfy
\begin{equation}\label{basicresiduenegative}
\operatorname*{Res}_{\lambda=-p}M(\lambda)
=\lim_{\lambda\to-p}M(\lambda)B_p^T,
\quad p\in\{\omega^j\zeta_n,\omega^j\eta_n:j=0,1,2\}.
\end{equation}
\end{subequations}
Equivalently, all $12N$ conditions use the indexed matrices
\eqref{firstsixresiduematricesnew}.
\end{enumerate}
\end{RHP}
\subsection{RH characterization of the solution for the Tzitz\'eica equation}
\label{uissolution}

Put $\boldsymbol\ell=(\omega,\omega^2,1)$. The endpoint matrix in
\eqref{G} satisfies
\[
\boldsymbol\ell G(q)=q\boldsymbol\ell,\qquad
\boldsymbol\ell G(q)^{-1}=q^{-1}\boldsymbol\ell.
\]
Consequently the reconstruction has two distinct steps:
\begin{equation}\label{TzitzeicaReconstructionFormula}
q(x,t)=\lim_{\lambda\to0}\left[(\omega,\omega^2,1)M(x,t,\lambda)\right]_{13},
\qquad u(x,t)=\log_{\rm cont}q(x,t).
\end{equation}
The continuous logarithm is normalized at the right spatial end.
It is not the principal logarithm, which cannot distinguish the
two charge-one limits even though $q$ tends to one at both ends.

\section{Normalization of the RH problem}
\label{s:3}

\subsection{Phase points and signature table}
\label{s:3.1}

To deal with the oscillatory components
$e^{\pm it\theta_{ij}(\lambda)}$ in RH problem
\ref{RHP:basic-meromorphic}, we consider the imaginary part of
$\theta_{12}(\lambda)$. Recall that
$
\vartheta_{ij}(\lambda)
=
it\theta_{ij}(\lambda),
$
and, with $\xi=\frac{x}{t}$, we have
\begin{equation*}
\theta_{12}(\lambda)
=
\frac{\sqrt3}{2}
\left[
(1+\xi)\lambda
+
(1-\xi)\lambda^{-1}
\right].
\end{equation*}

Hence
\begin{equation}
\label{Imtheta12}
\Im\theta_{12}(\lambda)
=
\frac{\sqrt3}{2}
\Im\lambda
\left[
(1+\xi)
-
(1-\xi)|\lambda|^{-2}
\right]
=
\frac{\sqrt3}{2}
\left(1+|\lambda|^{-2}\right)
\Im\lambda
\left[
\xi-
\frac{1-|\lambda|^2}{1+|\lambda|^2}
\right].
\end{equation}

The signature table of $\Im\theta_{12}$ follows directly from
\eqref{Imtheta12}. In particular, for $\xi\ge1$,
\begin{equation*}
\operatorname{sgn}\Im\theta_{12}(\lambda)
=
\operatorname{sgn}\Im\lambda,
\end{equation*}
whereas for $\xi\le-1$ the signs are reversed.

\begin{figure}[H]
\centering
\subfloat[$\xi\leq-1$]{\includegraphics[width=0.3\linewidth]{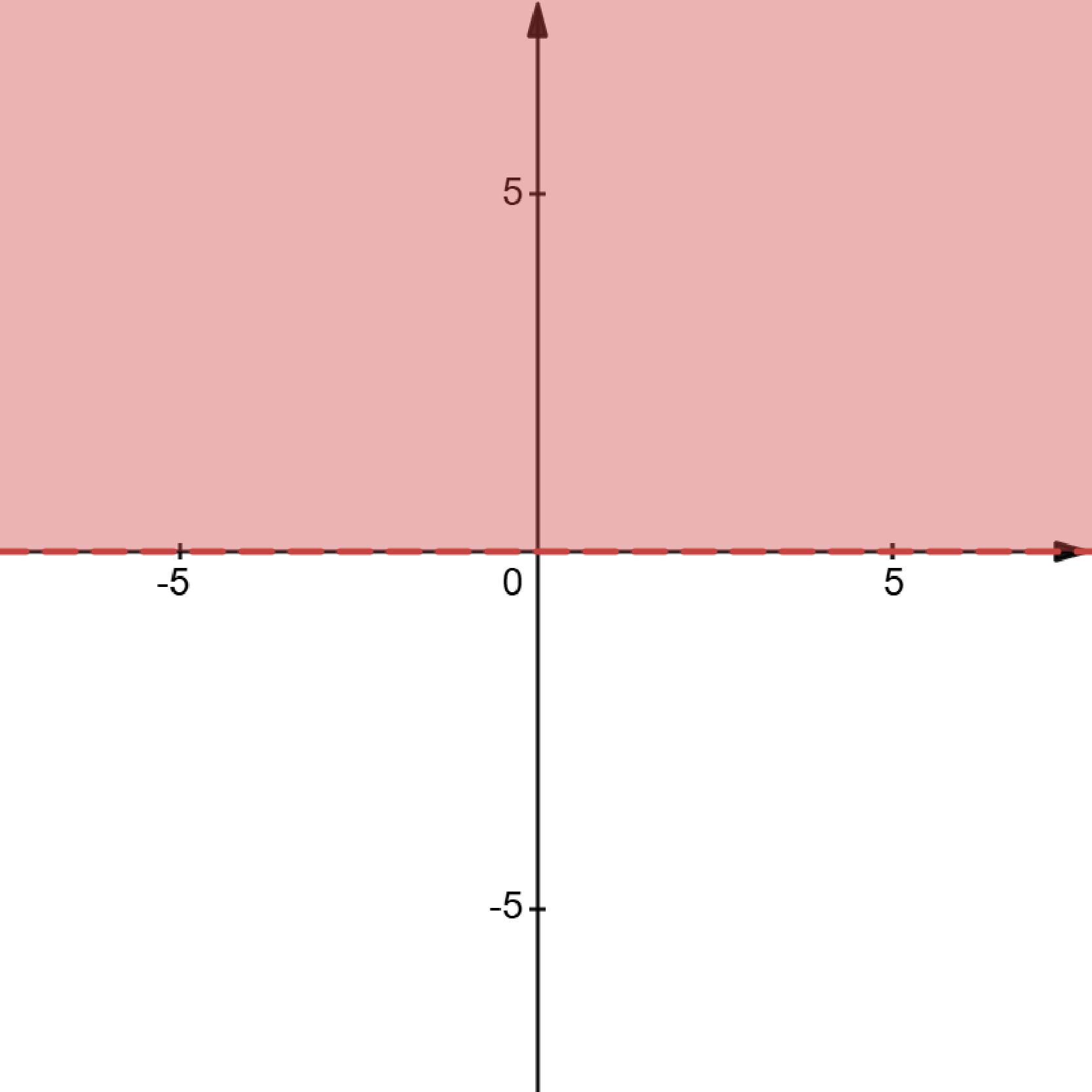}
\label{fig:desmos-graph}}
\subfloat[$-1<\xi<1$]{\includegraphics[width=0.3\linewidth]{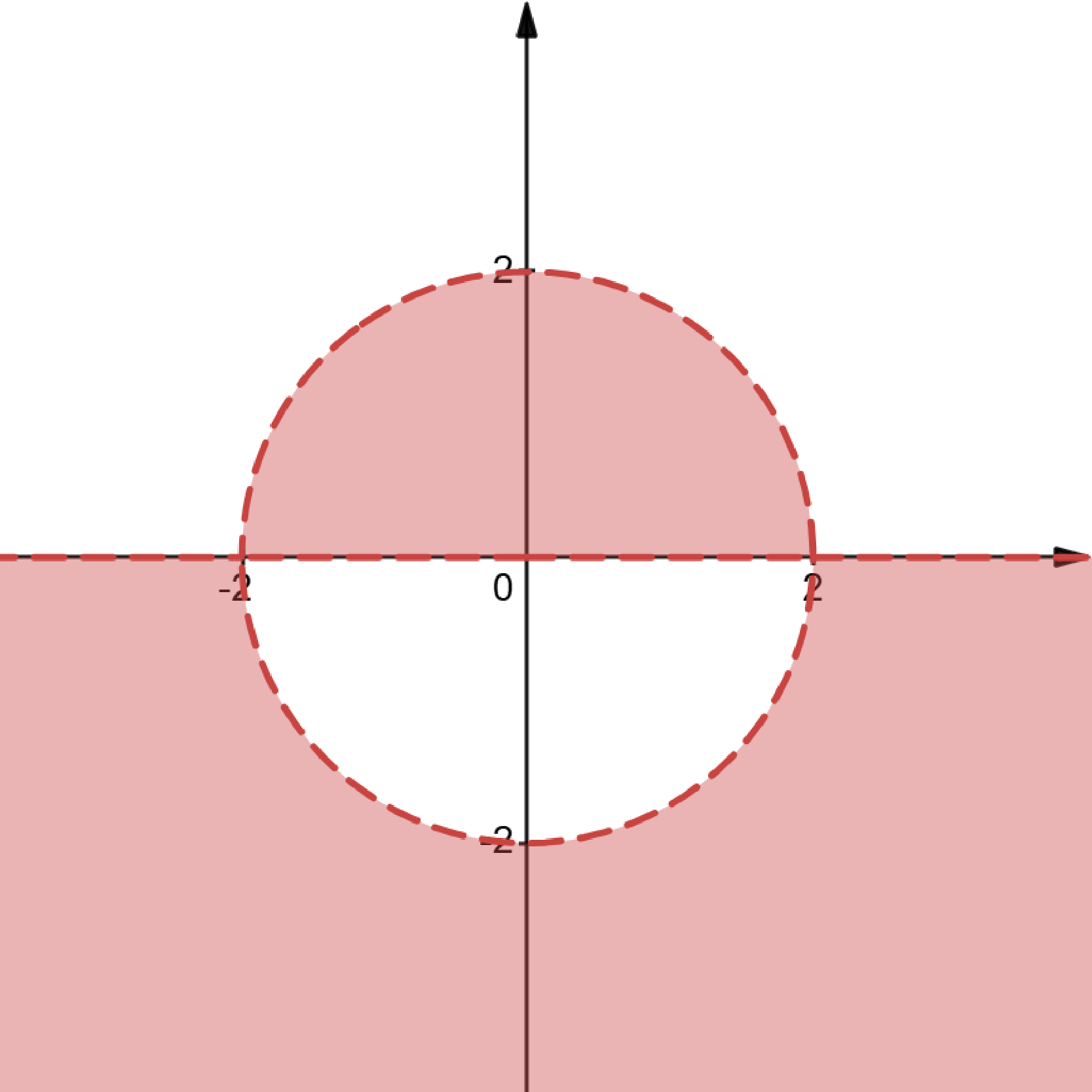}
\label{fig:desmos-graph-1}}
\subfloat[$\xi\geq1$]{\includegraphics[width=0.3\linewidth]{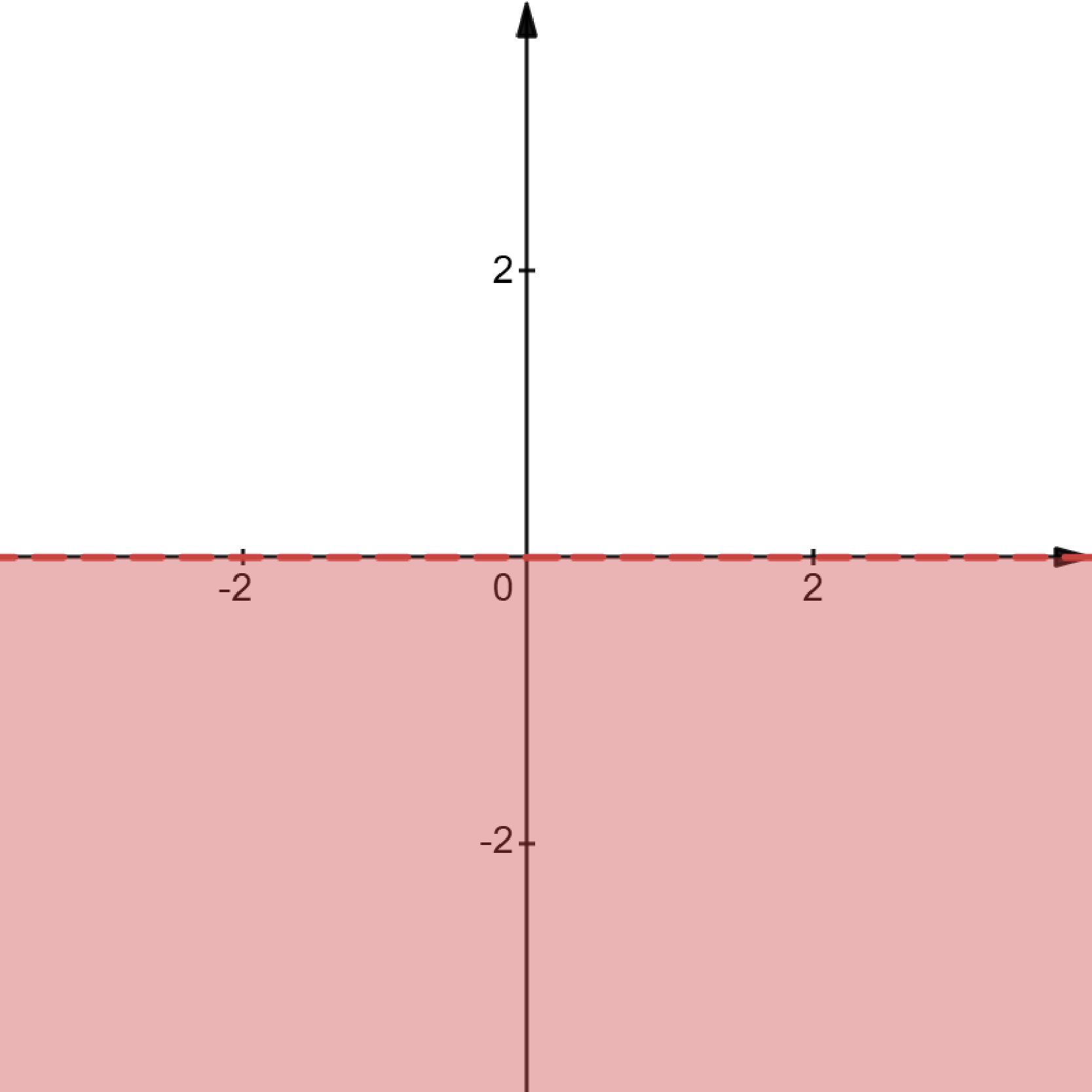}
\label{fig:desmos-graph-5}}
\caption{Signature of $\Re(2it\theta_{12}(\lambda))$.
The shaded region has positive sign, where $e^{-2it\theta_{12}(\lambda)}$
decays; the unshaded region has negative sign, where
$e^{2it\theta_{12}(\lambda)}$ decays. Their boundary is
$\Im\theta_{12}(\lambda)=0$, where both exponential factors have modulus one.}
\label{fig3}
\end{figure}

For $|\xi|<1$, introduce $\lambda_0=\sqrt{\frac{1-\xi}{1+\xi}}$. Then
\begin{equation*}
\Im\theta_{12}(\lambda)
=
\frac{\sqrt3(1+\xi)}
{2|\lambda|^2}
\Im\lambda
\left(
|\lambda|^2-\lambda_0^2
\right).
\end{equation*}

Thus the zero level set of $\Im\theta_{12}$ consists of the real
axis and the circle
$
|\lambda|=\lambda_0.
$
The stationary phase points of $\theta_{12}$ are determined by
\begin{equation*}
\frac{\partial\theta_{12}}
{\partial\lambda}
=
\frac{\sqrt3}{2}
\left[
(1+\xi)
-
(1-\xi)\lambda^{-2}
\right]
=
0.
\end{equation*}
Therefore, there are no real stationary phase points for $|\xi|\ge1$,
whereas for $|\xi|<1$ there are exactly two,
$
-\lambda_0,
\qquad
\lambda_0.
$
Consequently, the number of stationary phase points on the real axis
is
\begin{equation}
p(\xi)
=
\begin{cases}
0,
&
\xi\in(-\infty,-1]\cup[1,+\infty),
\\[1mm]
2,
&
\xi\in(-1,1).
\end{cases}
\nonumber
\end{equation}

The remaining stationary phase points follow from the cyclic phase
relations
$
\theta_{31}(\omega\lambda)
=
\theta_{12}(\lambda),
\theta_{23}(\omega^2\lambda)
=
\theta_{12}(\lambda).
$
Thus, for $|\xi|<1$, the complete six-ray RH problem has the six
stationary phase points
$
\pm\lambda_0,
\pm\omega\lambda_0,
\pm\omega^2\lambda_0.
$
\subsection{Conjugation}
\label{s:3.2}

For $t>0$, the signature table determines the triangular factorization of each
jump. We retain the two independent reflection functions
$\alpha,\beta$ of Section~\ref{s:2}, conjugate the diagonal factors,
and replace the remote pole conditions by small jumps on circles.
Throughout this section the scattering data satisfy
Assumption~\ref{ass:discrete} and the continuous nonvanishing condition
$\delta(s)=1-\alpha(s)\beta(s)\ne0$ for $s>0$.

For $\lambda>0$, the jump matrix $V(\lambda)$
admits the following two factorizations:
\begin{equation}\label{factorization-positive}
V=b_1^{\rm l}b_1^{\rm u}
=\widehat b_1^{\rm u}T_{0,+}
\widehat b_1^{\rm l},
\qquad \lambda>0,
\end{equation}
where
\[
b_1^{\rm l}
=
\begin{pmatrix}
1&0&0\\
\beta(\lambda)e^{-it\theta_{12}(\lambda)}&1&0\\
0&0&1
\end{pmatrix},
\qquad
b_1^{\rm u}
=
\begin{pmatrix}
1&-\alpha(\lambda)e^{it\theta_{12}(\lambda)}&0\\
0&1&0\\
0&0&1
\end{pmatrix},
\]
and
\[
\widehat b_1^{\rm u}
=
\begin{pmatrix}
1&
-\dfrac{\alpha(\lambda)}{\delta(\lambda)}
e^{it\theta_{12}(\lambda)}&0\\
0&1&0\\
0&0&1
\end{pmatrix},
T_{0,+}(\lambda)
=
\begin{pmatrix}
\delta(\lambda)^{-1}&0&0\\
0&\delta(\lambda)&0\\
0&0&1
\end{pmatrix},
\widehat b_1^{\rm l}
=
\begin{pmatrix}
1&0&0\\
\dfrac{\beta(\lambda)}{\delta(\lambda)}
e^{-it\theta_{12}(\lambda)}&1&0\\
0&0&1
\end{pmatrix}.
\]

Because the negative real ray is oriented from
the origin toward $-\infty$ in the six-ray
formulation, its jump has the complementary
factorization
\begin{equation}\label{factorization-negative}
V(-\lambda)
=b_4^{\rm u}b_4^{\rm l}
=\widehat b_4^{\rm l}T_{0,-}
\widehat b_4^{\rm u},
\qquad \lambda>0,
\end{equation}
where
\[
b_4^{\rm u}
=
\begin{pmatrix}
1&-\beta(\lambda)e^{-it\theta_{12}(\lambda)}&0\\
0&1&0\\
0&0&1
\end{pmatrix},
\qquad
b_4^{\rm l}
=
\begin{pmatrix}
1&0&0\\
\alpha(\lambda)e^{it\theta_{12}(\lambda)}&1&0\\
0&0&1
\end{pmatrix},
\]
and
\[
\widehat b_4^{\rm l}
=
\begin{pmatrix}
1&0&0\\
\dfrac{\alpha(\lambda)}{\delta(\lambda)}
e^{it\theta_{12}(\lambda)}&1&0\\
0&0&1
\end{pmatrix},
T_{0,-}(-\lambda)
=
\begin{pmatrix}
\delta(\lambda)&0&0\\
0&\delta(\lambda)^{-1}&0\\
0&0&1
\end{pmatrix},
\widehat b_4^{\rm u}
=
\begin{pmatrix}
1&
-\dfrac{\beta(\lambda)}{\delta(\lambda)}
e^{-it\theta_{12}(\lambda)}&0\\
0&1&0\\
0&0&1
\end{pmatrix}.
\]

We next classify the discrete spectrum. Use the polar variables of
Section~\ref{s:2},
\[
 \zeta_n=\rho_ne^{i\theta_n},\qquad
 \eta_n=e^{\pi i/3}\overline{\zeta_n},\qquad
 \rho_n>0,\quad 0<\theta_n<\pi/3,\quad \theta_n\ne\pi/6.
\]
By \eqref{Imtheta12},
\begin{equation}\label{discretephasevelocity}
 \Im\theta_{12}(\zeta_n)=\gamma_n(\xi-v_n),\qquad
 \gamma_n=\frac{\sqrt3(1+\rho_n^2)\sin\theta_n}{2\rho_n}>0,\qquad
 v_n=\frac{1-\rho_n^2}{1+\rho_n^2}.
\end{equation}

For a fixed $\xi$ with $|\xi|\ne1$, choose $\delta_0>0$ so that
$
 \delta_0<\min_{\{n:\,\xi\ne v_n\}}\gamma_n|\xi-v_n|,
$
omitting this restriction when the set is empty. Set
\begin{equation}\label{discretesplitting}
 \Delta=\{n:\Im\theta_{12}(\zeta_n)<-\delta_0\},
 \nabla=\{n:\Im\theta_{12}(\zeta_n)>\delta_0\},
 \Lambda=\{n:|\Im\theta_{12}(\zeta_n)|\le\delta_0\}.
\end{equation}
These sets partition $\{1,\ldots,N\}$. At a fixed ray with this
choice of $\delta_0$, $\Lambda=\{n:v_n=\xi\}$. For a velocity
neighborhood one may instead fix a smaller uniform gap for the
removed orbits and retain every orbit within the neighborhood.
Each index refers to the complete paired twelve-point set
$\mathcal O_n$, consisting of the two six-point algebraic orbits
generated from $\zeta_n$ and $\eta_n$. Neither the classification
nor distinct velocities establish general-$N$ regularity.

The full orbit index sets and corresponding pole sets are
\[
 \widetilde\Delta=\{n+kN:n\in\Delta,\ 0\le k\le11\},
 \widetilde\nabla=\{n+kN:n\in\nabla,\ 0\le k\le11\},
 \widetilde\Lambda=\{n+kN:n\in\Lambda,\ 0\le k\le11\}.
\]
\[
 \Lambda^c=\{1,\ldots,N\}\setminus\Lambda,\qquad
 \widetilde{\Lambda^c}=\widetilde\Delta\cup\widetilde\nabla,\qquad
 \mathcal P_\Lambda=\{\zeta_j:j\in\widetilde\Lambda\},\qquad
 \mathcal P_{\Lambda^c}=\mathcal P_N\setminus\mathcal P_\Lambda.
\]

Let
\begin{equation}\label{sec3:continuous-log}
f(s)=\operatorname{Log}\delta(s),
\qquad s>0,\qquad f(\infty)=0,
\end{equation}
be the continuous logarithmic lift, which exists
uniquely since $\delta$ is nonvanishing and
$\delta(\infty)=1$. By
Proposition~\ref{prop:reflection-positivity},
for every $m>0$,
\[
f(s)=2\pi i k+O(s^m)\quad(s\to0^+),
\qquad
f(s)=O(s^{-m})\quad(s\to\infty),
\]
with corresponding derivative estimates,
where $k\in\mathbb Z$. Thus $f\in L^1(\mathbb R_+)$.
If $\|\alpha\beta\|_\infty<1$, then $k=0$ and
$f(s)/s\in L^1(\mathbb R_+)$. This stronger
condition is not required for the present
conjugation.

Define
\begin{equation}\label{Ixi}
I(\xi)=
\begin{cases}
\mathbb R,&\xi\le-1,\\
(-\lambda_0,\lambda_0),&-1<\xi<1,\\
\varnothing,&\xi\ge1,
\end{cases}
\qquad
E(\xi)=I(\xi)\cap\mathbb R_+.
\end{equation}
The Cauchy integral below is along the real axis from left to
right, independently of the outward orientation of the RH rays.
Put
\begin{equation}\label{sec3:scalar-splitting}
\begin{aligned}
 H_d(\lambda)
   &=\prod_{n\in\Delta}
       \frac{(\lambda-\zeta_n)(\lambda-\omega\eta_n)}
            {(\lambda+\zeta_n)(\lambda+\omega\eta_n)},\\
 F_{E(\xi)}(\lambda)
   &=\frac{1}{2\pi i}\int_{E(\xi)} f(s)
       \left(\frac1{s-\lambda}-\frac1{s+\lambda}\right)\,ds,\\
 H_c(\lambda)&=\exp\{-F_{E(\xi)}(\lambda)\},\qquad
 H(\lambda)=H_d(\lambda)H_c(\lambda).
\end{aligned}
\end{equation}

The diagonal conjugating matrix is
\begin{equation}\label{THdef}
 T=\operatorname{diag}(T_1,T_2,T_3)=T_dT_c,\qquad
 T_i(\lambda)=\frac{H(\omega^{i+1}\lambda)}
                   {H(\omega^{i+2}\lambda)},\qquad i=1,2,3,
\end{equation}
where $(T_d)_i$ and $(T_c)_i$ are defined by the same quotient
with $H_d$ and $H_c$, respectively. Powers of $\omega$ are taken
modulo three. We also use
$
 T_{ij}=T_i/T_j,\qquad (T_{ij})_\pm=(T_i)_\pm/(T_j)_\pm.
$
In particular
\[
 T_c=\operatorname{diag}
 \bigl(e^{F_E(\lambda)-F_E(\omega^2\lambda)},
       e^{-F_E(\lambda)+F_E(\omega\lambda)},
       e^{-F_E(\omega\lambda)+F_E(\omega^2\lambda)}\bigr).
\]

On the positive ray the plus trace is the upper trace; on the
negative ray it is the lower trace. The Plemelj formula therefore
gives
$
 \frac{H_+(s)}{H_-(s)}=\delta(s)^{-1},
 \frac{H_+(-s)}{H_-(-s)}=\delta(s), s\in E(\xi).
$
The two rotated scalar arguments are regular across the real
ray, and consequently
\begin{equation}\label{Tjump}
 T_+(s)=T_-(s)T_{0,+}(s)^{-1},\qquad
 T_+(-s)=T_-(-s)T_{0,-}(-s)^{-1},\qquad s\in E(\xi).
\end{equation}
Outside these intervals $T$ has no real-axis jump. All remaining
ray jumps follow by cyclic rotation. The odd Cauchy kernel and
the rational divisor give $H(-\lambda)=H(\lambda)^{-1}$.
Thus the surviving reductions are exactly
\begin{equation}\label{Treductions}
 T(\omega\lambda)=\mathcal A T(\lambda)\mathcal A^{-1},\qquad
 T(-\lambda)=T(\lambda)^{-T},\qquad \det T(\lambda)=1.
\end{equation}

The endpoint normalizations are
\begin{equation}\label{sec3:T-endpoints}
 T(\lambda)=I+O(\lambda^{-1})\quad(\lambda\to\infty),\qquad
 T(\lambda)=I+O(\lambda)\quad(\lambda\to0).
\end{equation}

Define the critical set
\[
\Gamma_{\rm crit}(\xi)=
\begin{cases}
\Sigma,&|\xi|\ge1,\\
\Sigma\cup\{|\lambda|=\lambda_0\},&|\xi|<1.
\end{cases}
\]
If $\mathcal P_N$ is nonempty, omit minima over empty subsets
and choose
\[
 \varrho=\frac14\min\left\{
 \min_{p\in\mathcal P_N}\operatorname{dist}(p,\Sigma),\
 \min_{p\in\mathcal P_{\Lambda^c}}
       \operatorname{dist}(p,\Gamma_{\rm crit}(\xi)),\
 \min_{p\in\mathcal P_N}|p|,\
 \min_{\substack{p,q\in\mathcal P_N\\p\ne q}}|p-q|
 \right\}.
\]
Let $\mathbb D_j=D(\zeta_j,\varrho)$, $1\le j\le12N$.
These disks are pairwise disjoint, avoid zero and $\Sigma$,
and, for removed poles, avoid the critical curves. If $N=0$,
there are no disks and all discrete factors below are identities.

For each actual pole $\zeta_j\in\mathcal P_N$,
write its rank-one residue matrix in the form
\[
B_j=b_jE_{i_jk_j},
\qquad
b_j\ne0,
\qquad
i_j,k_j\in\{1,2,3\},
\quad i_j\ne k_j,
\]
where the amplitude and matrix position are those
specified in
\eqref{firstsixresiduematricesnew}.

The pole deletion matrix is
\begin{equation}\label{Gdef}
 \mathcal G(\lambda)=
 \begin{cases}
 I-\dfrac{B_j}{\lambda-\zeta_j},
     &\lambda\in\mathbb D_j,\quad j\in\widetilde\nabla,\\[3mm]
 I-\dfrac{\lambda-\zeta_j}{b_j}E_{k_ji_j},
     &\lambda\in\mathbb D_j,\quad j\in\widetilde\Delta,\\[3mm]
 I,&\text{elsewhere}.
 \end{cases}
\end{equation}
This formula includes all twelve members of each removed paired
orbit. The coefficients are independent amplitudes from
\eqref{betandefnew}, never substituted by their conjugates.
Nilpotence and the extension rules \eqref{residueZ3new} give
$
 \mathcal G(\omega\lambda)=\mathcal A\mathcal G(\lambda)
                                  \mathcal A^{-1},
 \mathcal G(-\lambda)=\mathcal G(\lambda)^{-T},
 \det\mathcal G=1.
$

Let $\breve\Lambda=\widetilde{\Lambda^c}$ and define
\begin{equation}\label{Sigma1}
 \Sigma^{(1)}=\Sigma\cup\Sigma^{(ci)},\qquad
 \Sigma^{(ci)}=\bigcup_{j\in\breve\Lambda}\partial\mathbb D_j.
\end{equation}
Circles in odd sectors $D_{2\nu-1}$ are counterclockwise and those
in even sectors $D_{2\nu}$ are clockwise, $\nu=1,2,3$.
Set
\[
 M^{(1)}(\lambda)=M(\lambda)\mathcal G(\lambda)T(\lambda).
\]
The preceding calculation yields the following equivalent
meromorphic RH problem.

\begin{RHP}\label{RHP:3.1}
Find a $3\times3$ matrix $M^{(1)}(x,t,\lambda)$ with the following
properties.
\begin{itemize}
\item Analyticity: $M^{(1)}$ is meromorphic in
$\mathbb C\setminus\Sigma^{(1)}$, with simple poles only at
$\mathcal P_\Lambda$.
\item Jump condition: $M^{(1)}_+=M^{(1)}_-V^{(1)}$, where
\begin{equation}\label{V1compact}
 V^{(1)}(\lambda)=
 \begin{cases}
 V_{\mathbb R}^{(1)}(\lambda),&\lambda\in\mathbb R,\\
 \mathcal A V_{\mathbb R}^{(1)}(\omega^2\lambda)\mathcal A^{-1},
     &\lambda\in\omega\mathbb R,\\
 \mathcal A^2 V_{\mathbb R}^{(1)}(\omega\lambda)\mathcal A^{-2},
     &\lambda\in\omega^2\mathbb R,\\
 T^{-1}\mathcal G T,
     &\lambda\in\partial\mathbb D_j\subset D_{2\nu-1},\\
 T^{-1}\mathcal G^{-1}T,
     &\lambda\in\partial\mathbb D_j\subset D_{2\nu}.
 \end{cases}
\end{equation}
For $\lambda\in\mathbb R\setminus\{0\}$ set $s=|\lambda|$.
The real-axis jump can be written
\begin{equation}\label{VR1}
 V_{\mathbb R}^{(1)}(\lambda)=
 \begin{cases}
 (I+\rho_\beta(s)T_{12}e^{-it\theta_{12}}E_{21})
 (I-\rho_\alpha(s)T_{21}e^{it\theta_{12}}E_{12}),
       &\lambda\in\mathbb R_+\setminus I(\xi),\\[2mm]
 (I+\rho_\alpha(s)(T_{21})_-e^{it\theta_{12}}E_{12})
 (I-\rho_\beta(s)(T_{12})_+e^{-it\theta_{12}}E_{21}),
       &\lambda\in I(\xi)\cap\mathbb R_+,\\[2mm]
 (I-\rho_\beta(s)T_{21}e^{it\theta_{12}}E_{12})
 (I+\rho_\alpha(s)T_{12}e^{-it\theta_{12}}E_{21}),
       &\lambda\in\mathbb R_-\setminus I(\xi),\\[2mm]
 (I-\rho_\alpha(s)(T_{12})_-e^{-it\theta_{12}}E_{21})
 (I+\rho_\beta(s)(T_{21})_+e^{it\theta_{12}}E_{12}),
       &\lambda\in I(\xi)\cap\mathbb R_-.
 \end{cases}
\end{equation}
Here the phases and all $T$ factors are evaluated at $\lambda$,
and the two reflection amplitudes are
\begin{equation}\label{rhodef}
 (\rho_\alpha(s),\rho_\beta(s))=
 \begin{cases}
 (\alpha(s),\beta(s)),&s\notin E(\xi),\\
 (-\alpha(s)/\delta(s),-\beta(s)/\delta(s)),&s\in E(\xi).
 \end{cases}
\end{equation}

\item Symmetry:$M^{(1)}(\omega\lambda)=\mathcal A M^{(1)}(\lambda)\mathcal A^{-1}, M^{(1)}(-\lambda)=M^{(1)}(\lambda)^{-T}$.
\item Asymptotic behavior:
\[
 M^{(1)}(\lambda)=I+O(\lambda^{-1})\quad(\lambda\to\infty),
 \qquad M^{(1)}(\lambda)=G(x,t)+O(\lambda)\quad(\lambda\to0).
\]
The matrix $G(x,t)$ is the origin matrix in \eqref{G}.

\item Residue conditions: for $p=\zeta_j\in\mathcal P_\Lambda$,
\begin{equation}\label{sec3:retained-residues}
 \operatorname*{Res}_{\lambda=p}M^{(1)}(\lambda)
  =\lim_{\lambda\to p}M^{(1)}(\lambda)\widetilde B_j^\Lambda,
 \qquad
 \widetilde B_j^\Lambda=T(p)^{-1}B_jT(p).
\end{equation}
The matrix $T$ is analytic and invertible at a retained pole.
The two independent modified constants are precisely
\begin{equation}\label{sec3:modified-two-constants}
 \widetilde c_n^\Lambda=c_n\frac{T_2(\zeta_n)}{T_1(\zeta_n)},
 \qquad
 \widetilde d_n^\Lambda=d_n\frac{T_3(\eta_n)}{T_2(\eta_n)},
 \qquad n\in\Lambda.
\end{equation}
\end{itemize}
\end{RHP}
\section{Hybrid $\bar\partial$-RH problem}
\label{s:4}

We open lenses around the six outward spectral rays, retaining the
independent reflection functions $\alpha,\beta$, the gap
$\delta=1-\alpha\beta$, and the factorization $T=T_dT_c$ of
Section~\ref{s:3}. Throughout $t>0$, $\xi=x/t$, $|\xi|\ne1$, and
$\bar\partial=\tfrac12(\partial_s+i\partial_v)$ for $\lambda=s+iv$.
We use the actual flat scattering class of
Proposition~\ref{prop:reflection-positivity} and assume
\eqref{sec2:continuous-nondegeneracy}. These hypotheses give
$\inf_{s>0}|\delta(s)|>0$, since $\delta$ tends to one at both
ends; they give no positivity or reality of $\delta$.

The two real-line channels needed for the lens factors are
\begin{equation}\label{sec4:two-real-channels}
 \rho_{0,1}(s)=
 \begin{cases}\rho_\beta(s),&s>0,\\-\rho_\alpha(-s),&s<0,\\0,&s=0,\end{cases}
 \rho_{0,2}(s)=
 \begin{cases}\rho_\alpha(s),&s>0,\\-\rho_\beta(-s),&s<0,\\0,&s=0.\end{cases}
\end{equation}
Here $\rho_\alpha,\rho_\beta$ are exactly \eqref{rhodef}.
Thus $\rho_{0,2}(-s)=-\rho_{0,1}(s)$. This relation swaps the
channels without conjugation. In the interior region their traces
at $\pm\lambda_0$ are taken from the adjoining real interval.

\subsection{Opening $\bar\partial$-lenses in
$\xi\in(-\infty,-1]\cup[1,+\infty)$}
\label{s:4.1}

Choose $\phi_0\in(0,\pi/6)$ sufficiently small that the closed
angular neighborhoods of the six rays avoid all pole disks.
Put $q_0=\tan\phi_0$ and
\[
 \Omega_{\rm up}=\{s+iv:s\ne0,\ 0<v<q_0|s|\},\qquad
 \Omega_{\rm down}=\{s+iv:s\ne0,\ -q_0|s|<v<0\}.
\]
For $\xi\geq 1$ take $\Omega_1=\Omega_{\rm down}$ and
$\Omega_2=\Omega_{\rm up}$; for $\xi\leq -1$ take $\Omega_1=\Omega_{\rm up}$ and
$\Omega_2=\Omega_{\rm down}$.
Set
\begin{equation}\label{sec4:sixdomains}
 \Omega_3=\omega\Omega_1,\quad \Omega_4=\omega\Omega_2,\quad
 \Omega_5=\omega^2\Omega_1,\quad \Omega_6=\omega^2\Omega_2,\qquad
 \Omega=\bigcup_{j=1}^6\Omega_j.
\end{equation}

\begin{figure}[htbp]
\centering
\begin{tikzpicture}[
>=stealth,
scale=0.94,
secfourarrow/.style={
    black,
    line width=0.85pt,
    postaction={decorate},
    decoration={
        markings,
        mark=at position 0.58 with {\arrow{stealth}}
    }
},
every node/.style={font=\scriptsize}
]

\begin{scope}[xshift=-3.8cm,scale=0.84]
\foreach \ang/\up/\down in {0/2/1,60/5/6,120/4/3,180/1/2,240/6/5,300/3/4}{
    \begin{scope}[rotate=\ang]

        \fill[gray!10] (0,0)--(12:2.7)--(0:2.7)--(-12:2.7)--cycle;

        \draw[
            densely dashed,
            black!55,
            line width=0.6pt
        ] (0,0)--(0:3.05);

        \draw[secfourarrow] (12:2.7)--(0,0);
        \draw[secfourarrow] (0,0)--(-12:2.7);

        \node[fill=white,inner sep=0.4pt] at (6:1.87) {$\Omega_{\up}$};
        \node[fill=white,inner sep=0.4pt] at (-6:1.87) {$\Omega_{\down}$};

        \node at (14:3.03) {$\Sigma_{\up}$};
        \node at (-14:3.03) {$\Sigma_{\down}$};

    \end{scope}
}

\fill[black] (0,0) circle (1.3pt);
\node[fill=white,inner sep=0.8pt] at (0,-0.22) {$0$};
\node[font=\small] at (0,-3.38) {(a) $\xi\geq1$};
\end{scope}

\begin{scope}[xshift=3.8cm,scale=0.84]
\foreach \ang/\up/\down in {0/1/2,60/6/5,120/3/4,180/2/1,240/5/6,300/4/3}{
    \begin{scope}[rotate=\ang]

        \fill[gray!10] (0,0)--(12:2.7)--(0:2.7)--(-12:2.7)--cycle;

        \draw[
            densely dashed,
            black!55,
            line width=0.6pt
        ] (0,0)--(0:3.05);

        \draw[secfourarrow] (12:2.7)--(0,0);
        \draw[secfourarrow] (0,0)--(-12:2.7);

        \node[fill=white,inner sep=0.4pt] at (6:1.87) {$\Omega_{\up}$};
        \node[fill=white,inner sep=0.4pt] at (-6:1.87) {$\Omega_{\down}$};

        \node at (14:3.03) {$\Sigma_{\up}$};
        \node at (-14:3.03) {$\Sigma_{\down}$};

    \end{scope}
}

\fill[black] (0,0) circle (1.3pt);
\node[fill=white,inner sep=0.8pt] at (0,-0.22) {$0$};
\node[font=\small] at (0,-3.38) {(b) $\xi\leq-1$};
\end{scope}

\end{tikzpicture}

\caption{Exterior lens domains. Dashed lines are the original spectral
lines; solid lines are the oriented lips. The labels follow the fixed
matrix positions used in \eqref{sec4:outerRbase}.
Each lip has its adjacent lens on its left.}
\label{fig:sec4:outer}
\end{figure}
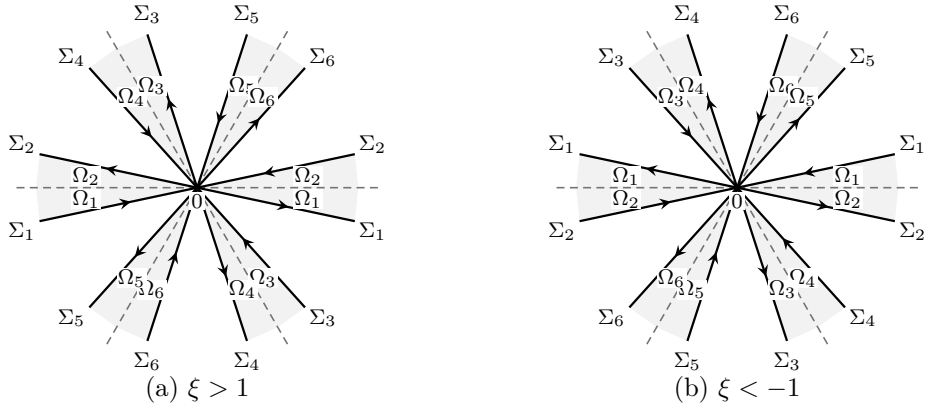

Let $\Sigma_j$ be the non-real boundary of $\Omega_j$.
Each component is oriented with the lens on its left, and
$\Sigma^{(ju)}=\bigcup_{j=1}^6\Sigma_j$.
The matrix positions remain fixed as in the original construction.
For $\lambda=s+iv$ with $v>0$ and $\lambda\ne0$,
\begin{equation}\label{sec4:outerphase}
\Im\theta_{12}(\lambda)
\ge\sqrt3\,v,
 \xi\ge1,
-\Im\theta_{12}(\lambda)
\ge\frac{\sqrt3\,v}{|\lambda|^2},
 \xi\le-1.
\end{equation}

\begin{lemma}\label{lem:sec4:outerextension}
There exist bounded functions
$R_j:\overline{\Omega_j}\to\mathbb C$,
$j=1,\ldots,6$, continuous up to the boundary
and continuously differentiable in each open
component, with the following boundary values:
\begin{align*}
R_1(\lambda)
&=
\begin{cases}
\rho_{0,1}(\lambda)(T_{12})_{\Omega_1}(\lambda),
    & \lambda\in\mathbb R,\\
0, & \lambda\in\Sigma_1,
\end{cases},
R_2(\lambda)
=
\begin{cases}
\rho_{0,2}(\lambda)(T_{21})_{\Omega_2}(\lambda),
    & \lambda\in\mathbb R,\\
0, & \lambda\in\Sigma_2,
\end{cases}\\[2mm]
R_3(\lambda)
&=
\begin{cases}
\rho_{0,1}(\omega^2\lambda)
(T_{31})_{\Omega_3}(\lambda),
    & \lambda\in\omega\mathbb R,\\
0, & \lambda\in\Sigma_3,
\end{cases},
R_4(\lambda)
=
\begin{cases}
\rho_{0,2}(\omega^2\lambda)
(T_{13})_{\Omega_4}(\lambda),
    & \lambda\in\omega\mathbb R,\\
0, & \lambda\in\Sigma_4,
\end{cases}\\[2mm]
R_5(\lambda)
&=
\begin{cases}
\rho_{0,1}(\omega\lambda)
(T_{23})_{\Omega_5}(\lambda),
    & \lambda\in\omega^2\mathbb R,\\
0, & \lambda\in\Sigma_5,
\end{cases},
R_6(\lambda)
=
\begin{cases}
\rho_{0,2}(\omega\lambda)
(T_{32})_{\Omega_6}(\lambda),
    & \lambda\in\omega^2\mathbb R,\\
0, & \lambda\in\Sigma_6.
\end{cases}
\end{align*}

For $j=1,2$, writing $\lambda=s+iv$, we have
\begin{equation}\label{sec4:outerdbarbound}
|R_j(\lambda)|\le C|\rho_{0,j}(s)|,
\qquad
|\bar\partial R_j(\lambda)|
\le C\left(
|\rho_{0,j}'(s)|
+\frac{|\rho_{0,j}(s)|}{|s|}
\right).
\end{equation}
The corresponding estimates hold for the
remaining functions by cyclic rotation.
All six extensions are flat at zero and
decrease rapidly at infinity. Moreover,
\begin{equation}\label{sec4:scalarreductions}
\begin{gathered}
R_2(-\lambda)=-R_1(\lambda),
R_{2m+1}(\omega^m\lambda)=R_1(\lambda),
R_{2m+2}(\omega^m\lambda)=R_2(\lambda),
\quad m=1,2.
\end{gathered}
\end{equation}

\end{lemma}

\begin{proof}
For $j=1,2$ define independently
\[
 R_j(s+iv)=\rho_{0,j}(s)T_{j,3-j}(s+iv)
       \cos^2\!\left(\frac{\pi v}{2q_0|s|}\right),
\]
where $T_{1,2}=T_{12}$ and $T_{2,1}=T_{21}$.
The cutoff is one on the real line and zero on the lip, with
$|\bar\partial\chi|\le C/|s|$. The ratios of $T$ are holomorphic
and bounded, together with their inverses, on these exterior
lenses: the discrete divisors are excluded, there are no finite
continuous endpoints, and \eqref{sec3:T-endpoints} controls zero
and infinity. Thus
$\bar\partial R_j=T_{j,3-j}
(\tfrac12\rho_{0,j}'\chi+\rho_{0,j}\bar\partial\chi)$.
This proves \eqref{sec4:outerdbarbound}.
The cyclic ratio identities are
\begin{equation}\label{sec4:ratiotransport}
 T_{31}(\omega\lambda)=T_{12}(\lambda),\quad
 T_{13}(\omega\lambda)=T_{21}(\lambda),\quad
 T_{23}(\omega^2\lambda)=T_{12}(\lambda),\quad
 T_{32}(\omega^2\lambda)=T_{21}(\lambda).
\end{equation}
They give the four rotated boundary prescriptions and estimates.
The rotation chain rule multiplies $\bar\partial$ by a scalar
of modulus one. Finally
$T_{21}(-\lambda)=T_{12}(\lambda)$ and the channel swap prove
$R_2(-\lambda)=-R_1(\lambda)$. Flatness and decay follow from
the two actual reflection functions and the endpoint bounds of $T$.
\end{proof}

For use in both regions, record all six matrix entries and phases:

\begin{equation}\label{sec4:outerRbase}
\mathcal R^{(2)}(\lambda)=
\begin{cases}
\begin{pmatrix}
1 & 0 & 0\\
R_1(\lambda)e^{-it\theta_{12}(\lambda)} & 1 & 0\\
0 & 0 & 1
\end{pmatrix},
& \lambda\in\Omega_1,\\[4mm]
\begin{pmatrix}
1 & R_2(\lambda)e^{it\theta_{12}(\lambda)} & 0\\
0 & 1 & 0\\
0 & 0 & 1
\end{pmatrix},
& \lambda\in\Omega_2,\\[4mm]
\begin{pmatrix}
1 & 0 & R_3(\lambda)e^{it\theta_{13}(\lambda)}\\
0 & 1 & 0\\
0 & 0 & 1
\end{pmatrix},
& \lambda\in\Omega_3,\\[4mm]
\begin{pmatrix}
1 & 0 & 0\\
0 & 1 & 0\\
R_4(\lambda)e^{-it\theta_{13}(\lambda)} & 0 & 1
\end{pmatrix},
& \lambda\in\Omega_4,\\[4mm]
\begin{pmatrix}
1 & 0 & 0\\
0 & 1 & 0\\
0 & R_5(\lambda)e^{-it\theta_{23}(\lambda)} & 1
\end{pmatrix},
& \lambda\in\Omega_5,\\[4mm]
\begin{pmatrix}
1 & 0 & 0\\
0 & 1 & R_6(\lambda)e^{it\theta_{23}(\lambda)}\\
0 & 0 & 1
\end{pmatrix},
& \lambda\in\Omega_6,\\[3mm]
I, & \lambda\notin\Omega.
\end{cases}
\end{equation}

The phase in each entry is evaluated at $\lambda$.
The surviving matrix reductions are
\begin{equation}\label{sec4:Rreductions}
 \mathcal R^{(2)}(\omega\lambda)=
     \mathcal A\mathcal R^{(2)}(\lambda)\mathcal A^{-1},\qquad
 \mathcal R^{(2)}(-\lambda)=\mathcal R^{(2)}(\lambda)^{-T},
 \qquad\det\mathcal R^{(2)}=1.
\end{equation}

\subsection{Opening $\bar\partial$-lenses in $-1<\xi<1$}
\label{s:4.2}

Put $a=\lambda_0=\sqrt{(1-\xi)/(1+\xi)}$, $\lambda_1=a$,
$\lambda_2=-a$, and
$\mathcal S(\xi)=\{\pm a,\pm\omega a,\pm\omega^2a\}$.
For $s\ne0,\pm a$ define
\begin{equation}\label{sec4:innerdomains}
 h(s)=
 \begin{cases}q_0(|s|-a),&|s|>a,\\
 q_0|s|(a-|s|)/a,&0<|s|<a.\end{cases}
\end{equation}
Choose $q_0$ smaller if necessary. The two real-line lenses are
\begin{equation}\label{sec4:innerfixedpositions}
\begin{aligned}
 \Omega_1={}&\{|s|>a,\ -h(s)<v<0\}
            \cup\{0<|s|<a,\ 0<v<h(s)\},\\
 \Omega_2={}&\{|s|>a,\ 0<v<h(s)\}
            \cup\{0<|s|<a,\ -h(s)<v<0\},
 \qquad\lambda=s+iv.
\end{aligned}
\end{equation}
Define the other four lenses by \eqref{sec4:sixdomains}.
They avoid the pole disks. There is no jump on the phase zero
circle $|\lambda|=a$. Directly from \eqref{Imtheta12},
\begin{equation}\label{sec4:innerphasesigns}
 \Im\theta_{12}<0\ \hbox{in }\Omega_1,\qquad
 \Im\theta_{12}>0\ \hbox{in }\Omega_2.
\end{equation}
Near each real stationary point,
\begin{equation}\label{sec4:quadraticdecay}
 |\Im\theta_{12}(s+iv)|\ge c|v|\,|s-\lambda_i|.
\end{equation}
The same statements hold after rotation. Use the continuous logarithm fixed in \eqref{sec3:continuous-log},
and set
\begin{equation}\label{sec4:complex-nu}
 \nu_a=-\frac{f(a)}{2\pi}\in\mathbb C,\qquad
 b=|\Im\nu_a|,\qquad \eta_i=(-1)^{i+1},\qquad
 w_i(z)=\eta_i(z-\lambda_i).
\end{equation}
The sign $\eta_i$ is distinct from the seed pole $\eta_n$. The local powers are
\begin{equation}\label{sec4:localpowers}
 w_1^{2i\eta_1\nu_a}=(z-a)^{2i\nu_a},\qquad
 w_2^{2i\eta_2\nu_a}=(-z-a)^{-2i\nu_a}.
\end{equation}
Use the principal logarithm of the displayed variable, with
negative-axis traces of arguments $\pm\pi$. On the rotated lines
replace $z$ by $\omega^2\lambda$ or $\omega\lambda$.

\begin{figure}[htbp]
\centering
\begin{tikzpicture}[
>=stealth,
scale=1.2,
secfourarrow/.style={
    black,
    line width=0.9pt,
    postaction={decorate},
    decoration={
        markings,
        mark=at position 0.55 with {\arrow{stealth}}
    }
},
every node/.style={font=\scriptsize}
]

\begin{scope}[xshift=-3.9cm,scale=0.89]

\def\aa{1.6}
\def\qq{0.32}
\def\LL{3.15}

\draw[
    densely dotted,
    black!40,
    line width=0.55pt
] (0,0) circle (\aa);

\foreach \ang/\up/\down in {
0/2/1,
60/5/6,
120/4/3,
180/1/2,
240/6/5,
300/3/4
}{
\begin{scope}[rotate=\ang]

    \draw[
        densely dashed,
        black!55,
        line width=0.6pt
    ] (0,0)--(\LL,0);

    \draw[
        secfourarrow,
        domain=1.6:0,
        samples=36
    ]
    plot (\x,{\qq*\x*(\aa-\x)/\aa});

    \draw[
        secfourarrow,
        domain=0:1.6,
        samples=36
    ]
    plot (\x,{-\qq*\x*(\aa-\x)/\aa});

    \draw[secfourarrow]
    (\LL,{\qq*(\LL-\aa)})--(\aa,0);

    \draw[secfourarrow]
    (\aa,0)--(\LL,{-\qq*(\LL-\aa)});

    \fill[black] (\aa,0) circle (1.7pt);

    \node[
        fill=white,
        inner sep=0.3pt
    ] at (2.71,0.13)
    {$\Omega_{\up}$};

    \node[
        fill=white,
        inner sep=0.3pt
    ] at (2.71,-0.13)
    {$\Omega_{\down}$};

\end{scope}
}

\fill[black] (0,0) circle (1.35pt);
\node[
fill=white,
inner sep=0.6pt
] at (0,-0.20)
{$0$};

\foreach \ang/\lab in {
0/{a},
60/{-\omega^2a},
120/{\omega a},
180/{-a},
240/{\omega^2a},
300/{-\omega a}
}{
\node[
fill=white,
inner sep=0.4pt
] at (\ang:1.99)
{$\lab$};
}

\node[
font=\small
] at (0,-3.68)
{(a) Six stationary points};

\end{scope}

\begin{scope}[xshift=0.3cm,yshift=0cm,scale=1.15]

\def\aa{2.0}
\def\qq{0.40}
\def\LL{4.0}

\draw[
    densely dashed,
    black!55,
    line width=0.6pt
]
(-0.15,0)--(4.2,0);

\draw[
secfourarrow,
domain=2.0:0,
samples=40
]
plot (\x,{\qq*\x*(\aa-\x)/\aa});

\draw[
secfourarrow,
domain=0:2.0,
samples=40
]
plot (\x,{-\qq*\x*(\aa-\x)/\aa});

\draw[secfourarrow]
(\LL,{\qq*(\LL-\aa)})--(\aa,0);

\draw[secfourarrow]
(\aa,0)--(\LL,{-\qq*(\LL-\aa)});

\fill[black] (0,0) circle (1.5pt);
\fill[black] (\aa,0) circle (1.8pt);

\node[
below=3pt
] at (0,0)
{$0$};

\node[
fill=white,
below=3pt
] at (\aa,-0.1)
{$a=\lambda_1$};

\node[
fill=white,
inner sep=0.4pt
] at (0.88,0.16)
{$\Omega_1$};

\node[
fill=white,
inner sep=0.4pt
] at (0.88,-0.16)
{$\Omega_2$};

\node at (3.25,0.21)
{$\Omega_2$};

\node at (3.25,-0.21)
{$\Omega_1$};

\node at (0.80,0.72)
{$e^{-it\theta_{12}}E_{21}$};

\node at (0.80,-0.72)
{$e^{it\theta_{12}}E_{12}$};

\node at (3.20,1.03)
{$e^{it\theta_{12}}E_{12}$};

\node at (3.20,-1.03)
{$e^{-it\theta_{12}}E_{21}$};

\node[
font=\small
] at (1.97,-2.85)
{(b) Positive real-ray detail};

\end{scope}

\end{tikzpicture}

\caption{
Interior lenses. The dotted circle is a phase zero set, not a jump contour.
The domain numbers refer to fixed matrix positions: $\Omega_1$ carries
the $(2,1)$ entry and $\Omega_2$ the $(1,2)$ entry on both sides of each
real stationary point. The other domains are obtained by the stated rotations.
}
\label{fig:sec4:inner}
\end{figure}
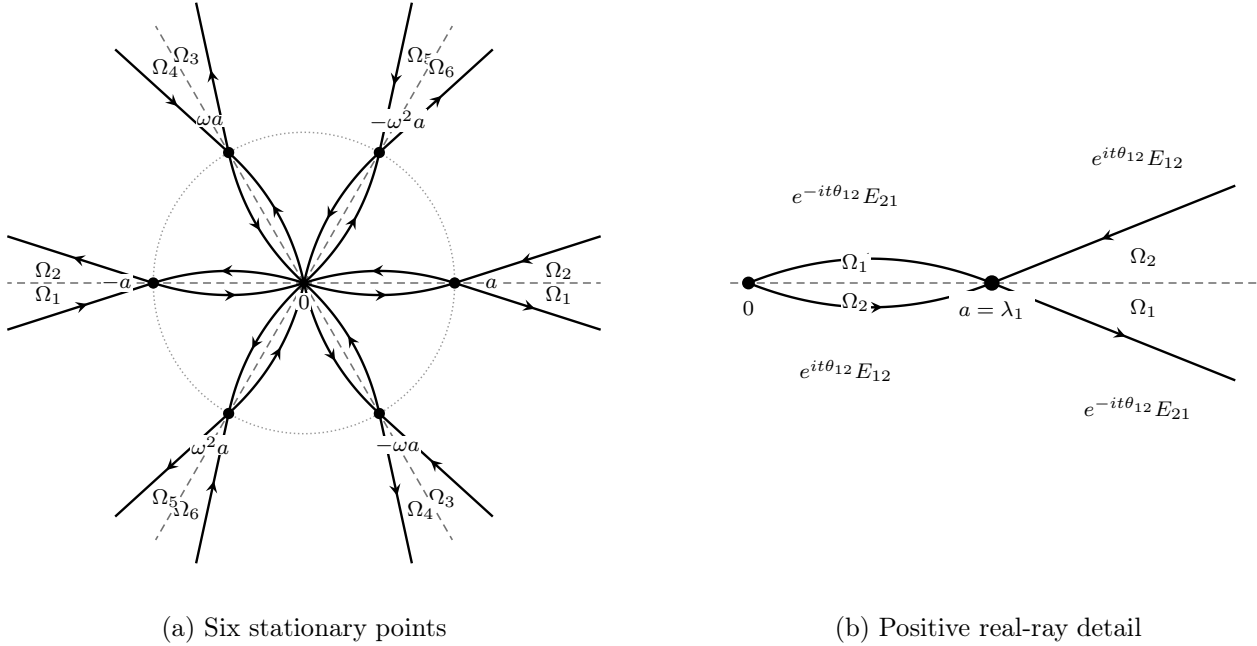

We make the common endpoint coefficient explicit. Near $a$, put
\[
 I_a(z)=\int_0^a\frac{f(s)-f(a)}{s-z}\,ds,\qquad
 J_a(z)=\int_0^a\frac{f(s)}{s+z}\,ds.
\]
Integration of the constant density gives
\[
 F_{(0,a)}(z)=i\nu_a\log(z-a)+\mathfrak r_a(z),\qquad
 \mathfrak r_a(z)=\frac{-f(a)\log z+I_a(z)-J_a(z)}{2\pi i}.
\]
Here $\log z$ is analytic near $a$ and $\log(z-a)$ has the
branch just specified. The regular part has the common endpoint
value
\[
 \mathfrak r_a(a)=\frac1{2\pi i}
 \left[-f(a)\log a+
 \int_0^a\frac{f(s)-f(a)}{s-a}\,ds
 -\int_0^a\frac{f(s)}{s+a}\,ds\right].
\]
The integrals converge; the actual scattering density is smooth
near $a$. Define the nonzero complex number
\begin{equation}\label{aaaa}
 \tau_a=(T_d)_{21}(a)
  \exp\{-2\mathfrak r_a(a)+F_{(0,a)}(\omega a)
                              +F_{(0,a)}(\omega^2a)\}.
\end{equation}
In general $|\tau_a|$ is not one. Subtracting the endpoint value
in $I_a$, or integrating its first-order Taylor term, gives
\begin{equation}\label{sec4:endpointfactor}
 (z-a)^{2i\nu_a}T_{21}(z)
  =\tau_a+O\!\left(|z-a|(1+|\log|z-a||)\right).
\end{equation}
The derivative of the regular factor is
$O(1+|\log|z-a||)$, uniformly up to either adjoining real trace.
Indeed the first Taylor term in $f(s)-f(a)$ has an explicit
logarithmic derivative, and the quadratic remainder has a
bounded derivative Cauchy transform after localization.
The other scalar arguments and $T_d$ are analytic near $a$.

Regularize the ratios by
\[
\begin{aligned}
 T_{12}^{(i)}(z)&=T_{12}(z)w_i(z)^{-2i\eta_i\nu_a},
 &T_{21}^{(i)}(z)&=(T_{12}^{(i)}(z))^{-1},\\
 T_{31}^{(i)}(\lambda)&=T_{12}^{(i)}(\omega^2\lambda),
 &T_{13}^{(i)}&=(T_{31}^{(i)})^{-1},\\
 T_{23}^{(i)}(\lambda)&=T_{12}^{(i)}(\omega\lambda),
 &T_{32}^{(i)}&=(T_{23}^{(i)})^{-1}.
\end{aligned}
\]
The negative-transpose and cyclic rules imply
\begin{equation}\label{sec4:frozenTvalues}
 T_{12}^{(1)}(a)=\tau_a^{-1},
 T_{12}^{(2)}(-a)=\tau_a,
 T_{31}^{(i)}(\omega\lambda_i)=T_{12}^{(i)}(\lambda_i),
 T_{23}^{(i)}(\omega^2\lambda_i)=T_{12}^{(i)}(\lambda_i).
\end{equation}
Reciprocal ratios have reciprocal endpoint values.
The channel traces, in the order outside/inside the stationary
circle, are
\[
\begin{aligned}
(\rho_{0,1},\rho_{0,2})(a+0)
&=(\beta(a),\alpha(a)),
(\rho_{0,1},\rho_{0,2})(a-0)
=\left(-\frac{\beta(a)}{\delta(a)},
        -\frac{\alpha(a)}{\delta(a)}\right),\\
(\rho_{0,1},\rho_{0,2})(-a-0)
&=(-\alpha(a),-\beta(a)),
(\rho_{0,1},\rho_{0,2})(-a+0)
=\left(\frac{\alpha(a)}{\delta(a)},
        \frac{\beta(a)}{\delta(a)}\right).
\end{aligned}
\]

Choose a smooth cutoff $\chi_a$ satisfying
$\chi_a(s)=1$ for $|s-a|\le a/8$ and
$\chi_a(s)=0$ for $|s-a|\ge a/3$, and set
$\chi_i(z)=\chi_a(\eta_i\Re z)$.
The orbit-invariant pole cutoff
$\mathcal X_{\mathcal P}$ is supported
inside the pole disks and vanishes
throughout the lenses.

\begin{lemma}\label{lem:sec4:innerextension}
There exist functions
$R_j:\overline{\Omega_j}\setminus\mathcal S(\xi)
\to\mathbb C$, $j=1,\ldots,6$,
continuous on each closed lens component away from
the stationary vertices and continuously
differentiable in its interior, with the following
boundary values:
\begin{align*}
R_1(\lambda)
&=
\begin{cases}
\rho_{0,1}(\lambda)(T_{12})_{\Omega_1}(\lambda),
  & \lambda\in\mathbb R,\\
\rho_{0,1}(\lambda_i\pm0)
T_{12}^{(i)}(\lambda_i)
w_i(\lambda)^{2i\eta_i\nu_a}
\chi_i(\lambda),
  & \lambda\in\Sigma_1,
\end{cases}
\\[2mm]
R_2(\lambda)
&=
\begin{cases}
\rho_{0,2}(\lambda)(T_{21})_{\Omega_2}(\lambda),
  & \lambda\in\mathbb R,\\
\rho_{0,2}(\lambda_i\pm0)
T_{21}^{(i)}(\lambda_i)
w_i(\lambda)^{-2i\eta_i\nu_a}
\chi_i(\lambda),
  & \lambda\in\Sigma_2,
\end{cases}
\\[2mm]
R_3(\lambda)
&=
\begin{cases}
\rho_{0,1}(\omega^2\lambda)
(T_{31})_{\Omega_3}(\lambda),
  & \lambda\in\omega\mathbb R,\\
\rho_{0,1}(\lambda_i\pm0)
T_{31}^{(i)}(\omega\lambda_i)
w_i(\omega^2\lambda)^{2i\eta_i\nu_a}
\chi_i(\omega^2\lambda),
  & \lambda\in\Sigma_3,
\end{cases}
\\[2mm]
R_4(\lambda)
&=
\begin{cases}
\rho_{0,2}(\omega^2\lambda)
(T_{13})_{\Omega_4}(\lambda),
  & \lambda\in\omega\mathbb R,\\
\rho_{0,2}(\lambda_i\pm0)
T_{13}^{(i)}(\omega\lambda_i)
w_i(\omega^2\lambda)^{-2i\eta_i\nu_a}
\chi_i(\omega^2\lambda),
  & \lambda\in\Sigma_4,
\end{cases}
\\[2mm]
R_5(\lambda)
&=
\begin{cases}
\rho_{0,1}(\omega\lambda)
(T_{23})_{\Omega_5}(\lambda),
  & \lambda\in\omega^2\mathbb R,\\
\rho_{0,1}(\lambda_i\pm0)
T_{23}^{(i)}(\omega^2\lambda_i)
w_i(\omega\lambda)^{2i\eta_i\nu_a}
\chi_i(\omega\lambda),
  & \lambda\in\Sigma_5,
\end{cases}
\\[2mm]
R_6(\lambda)
&=
\begin{cases}
\rho_{0,2}(\omega\lambda)
(T_{32})_{\Omega_6}(\lambda),
  & \lambda\in\omega^2\mathbb R,\\
\rho_{0,2}(\lambda_i\pm0)
T_{32}^{(i)}(\omega^2\lambda_i)
w_i(\omega\lambda)^{-2i\eta_i\nu_a}
\chi_i(\omega\lambda),
  & \lambda\in\Sigma_6.
\end{cases}
\end{align*}
Here $i=1,2$ labels the adjacent stationary vertex. All six functions are flat at zero,
decrease rapidly at infinity, and satisfy
\eqref{sec4:scalarreductions}.

For $j=2m+k$, $m=0,1,2$ and $k=1,2$,
write $s=\Re(\omega^{-m}\lambda)$.
At an adjacent stationary point $p$, with
$r=|\lambda-p|<1$ and $b=|\Im\nu_a|$,
\begin{equation}\label{sec4:innerDPbound}
|R_j(\lambda)|\le Cr^{-2b},\qquad
|\bar\partial R_j(\lambda)|
\le Cr^{-2b}
\left(|\rho_{0,k}'(s)|+r^{-1/2}\right).
\end{equation}
For smooth scattering data, the sharper local bound
$|\bar\partial R_j(\lambda)|
\le Cr^{-2b}(1+|\log r|)$ also holds.
\end{lemma}

\begin{proof}
Construct $R_1,R_2$ independently on the two
positive-real components adjoining $a$.
For $k=1,2$, put $\epsilon_1=1$,
$\epsilon_2=-1$, and define
\[
g_k(s)=\rho_{0,k}(s)
(T_{k,3-k})_{\Omega_k}(s)
(s-a)^{-2i\epsilon_k\nu_a}.
\]
By \eqref{sec4:endpointfactor}, each one-sided
$g_k$ belongs locally to $H^1$ up to $a$, with
\[
|g_k'(s)|
\le C\bigl(
|\rho_{0,k}'(s)|+1+|\log|s-a||
\bigr).
\]
Its endpoint value is determined by the
corresponding one-sided trace and regularized
$T$-ratio. Define
\[
R_k(z)=(z-a)^{2i\epsilon_k\nu_a}
\bigg[
g_k(s)\cos^2\left(\frac{\pi v}{2h(s)}\right)
+\chi_a(s)g_k(a)
\sin^2\left(\frac{\pi v}{2h(s)}\right)
\bigg],
\qquad z=s+iv,\quad s>0.
\]
Extend to the negative-real components by
$R_2(-z)=-R_1(z)$ and
$R_1(-z)=-R_2(z)$, and then rotate.
The one-sided traces and
\eqref{sec4:frozenTvalues} verify the
prescribed boundary values and reductions.

Near $a$, we have
$h(s)\asymp|s-a|\asymp|z-a|$ and
$\left|
\bar\partial
\cos^2\left(\frac{\pi v}{2h(s)}\right)
\right|
\le \frac{C}{h(s)}.
$
The analytic power has zero $\bar\partial$
derivative and modulus bounded by $Cr^{-2b}$.
The $H^1$ estimate
$|g_k(s)-g_k(a)|\le C|s-a|^{1/2}$
yields \eqref{sec4:innerDPbound}, with
the logarithmic term absorbed by $r^{-1/2}$.
For smooth scattering data,
$g_k(s)-g_k(a)=O(r(1+|\log r|))$,
giving the sharper estimate. Away from the stationary vertices,
the exterior-type estimates apply.
Flatness at zero and rapid decay at infinity
follow from the scattering data.
Since the pole cutoff vanishes throughout
the lenses, it produces no additional
$\bar\partial$ contribution.
\end{proof}

With these six functions, the interior matrix extension is
\begin{equation}\nonumber
\mathcal R^{(2)}(\lambda)=
\begin{cases}
\begin{pmatrix}
1 & 0 & 0\\
R_1(\lambda)e^{-it\theta_{12}(\lambda)} & 1 & 0\\
0 & 0 & 1
\end{pmatrix},
& \lambda\in\Omega_1,\\[4mm]
\begin{pmatrix}
1 & R_2(\lambda)e^{it\theta_{12}(\lambda)} & 0\\
0 & 1 & 0\\
0 & 0 & 1
\end{pmatrix},
& \lambda\in\Omega_2,\\[4mm]
\begin{pmatrix}
1 & 0 & R_3(\lambda)e^{it\theta_{13}(\lambda)}\\
0 & 1 & 0\\
0 & 0 & 1
\end{pmatrix},
& \lambda\in\Omega_3,\\[4mm]
\begin{pmatrix}
1 & 0 & 0\\
0 & 1 & 0\\
R_4(\lambda)e^{-it\theta_{13}(\lambda)} & 0 & 1
\end{pmatrix},
& \lambda\in\Omega_4,\\[4mm]
\begin{pmatrix}
1 & 0 & 0\\
0 & 1 & 0\\
0 & R_5(\lambda)e^{-it\theta_{23}(\lambda)} & 1
\end{pmatrix},
& \lambda\in\Omega_5,\\[4mm]
\begin{pmatrix}
1 & 0 & 0\\
0 & 1 & R_6(\lambda)e^{it\theta_{23}(\lambda)}\\
0 & 0 & 1
\end{pmatrix},
& \lambda\in\Omega_6,\\[3mm]
I, & \lambda\notin\Omega.
\end{cases}
\end{equation}
For every positive integer $L$, at fixed $x, t$,

\begin{equation}\label{sec4:Rnormalizations}
 \mathcal R^{(2)}=I+O(|\lambda|^L)\quad(\lambda\to0),\qquad
 \mathcal R^{(2)}=I+O(|\lambda|^{-L})\quad(\lambda\to\infty).
\end{equation}

\subsection{A hybrid $\bar\partial$-RH problem and its decomposition}
\label{s:4.3}

Set $\Sigma^{(2)}=\Sigma^{(ju)}\cup\Sigma^{(ci)}$ and
\begin{equation}\label{sec4:M2transformation}
 M^{(2)}(x,t,\lambda)=M^{(1)}(x,t,\lambda)\mathcal R^{(2)}(x,t,\lambda).
\end{equation}
At a real vertex $\lambda_i$ introduce the analytic sectorial gauge
\begin{equation}\label{sec4:vertex-gauge}
 D_i(z)=\operatorname{diag}
 \bigl(w_i(z)^{i\eta_i\nu_a},w_i(z)^{-i\eta_i\nu_a},1\bigr).
\end{equation}
At $\omega^j\lambda_i$ use
$D_{\omega^j\lambda_i}(\lambda)=
\mathcal A^jD_i(\omega^{-j}\lambda)\mathcal A^{-j}$.
All gauge powers have the branches specified in
\eqref{sec4:localpowers}.
The matrices $TD_p^{-1}$ and
$D_p\mathcal R^{(2)}D_p^{-1}$ are bounded with bounded inverses
in every adjoining sector. For the second assertion the powers
in each off-diagonal entry cancel exactly; for the first it
follows from the scalar endpoint expansion and the regular
rotated factors. Thus $T\mathcal R^{(2)}D_p^{-1}$ has the same
boundedness and invertibility.

\begin{RHP}\label{RHP:4.1}
Find a $3\times3$ matrix-valued function
$M^{(2)}(x,t,\lambda)$ satisfying the
following properties.

\begin{itemize}

\item
$M^{(2)}$ has sectionally continuous first
partial derivatives in $\mathbb C\setminus
\bigl(\Sigma^{(2)}\cup P_\Lambda
\cup\mathcal S(\xi)\bigr)$,
and is meromorphic in
$\mathbb C\setminus
\bigl(\overline{\Omega}\cup
\Sigma^{(\mathrm{ci})}\bigr)$, with only the retained simple poles.

\item
The boundary values satisfy
\[
M_+^{(2)}(x,t,\lambda)
=
M_-^{(2)}(x,t,\lambda)V^{(2)}(\lambda),
\qquad \lambda\in\Sigma^{(2)}.
\]

For $-1<\xi<1$, taking the functions $R_j$
from the adjacent lens components, we have
\begin{equation}\label{sec4:M2jumpexplicit}
V^{(2)}(\lambda)=
\begin{cases}
\begin{pmatrix}
1&0&0\\
R_1(\lambda)e^{-it\theta_{12}(\lambda)}&1&0\\
0&0&1
\end{pmatrix},
&\lambda\in\Sigma_1,
\\[4mm]
\begin{pmatrix}
1&R_2(\lambda)e^{it\theta_{12}(\lambda)}&0\\
0&1&0\\
0&0&1
\end{pmatrix},
&\lambda\in\Sigma_2,
\\[4mm]
\begin{pmatrix}
1&0&R_3(\lambda)e^{it\theta_{13}(\lambda)}\\
0&1&0\\
0&0&1
\end{pmatrix},
&\lambda\in\Sigma_3,
\\[4mm]
\begin{pmatrix}
1&0&0\\
0&1&0\\
R_4(\lambda)e^{-it\theta_{13}(\lambda)}&0&1
\end{pmatrix},
&\lambda\in\Sigma_4,
\\[4mm]
\begin{pmatrix}
1&0&0\\
0&1&0\\
0&R_5(\lambda)e^{-it\theta_{23}(\lambda)}&1
\end{pmatrix},
&\lambda\in\Sigma_5,
\\[4mm]
\begin{pmatrix}
1&0&0\\
0&1&R_6(\lambda)e^{it\theta_{23}(\lambda)}\\
0&0&1
\end{pmatrix},
&\lambda\in\Sigma_6,
\\[4mm]
T^{-1}\mathcal GT,
&\lambda\in\partial\mathbb D_j
\subset D_1\cup D_3\cup D_5,
\\[3mm]
T^{-1}\mathcal G^{-1}T,
&\lambda\in\partial\mathbb D_j
\subset D_2\cup D_4\cup D_6.
\end{cases}
\end{equation}

For $\xi\in(-\infty,-1]\cup[1,+\infty)$,
the jump matrix is
\begin{equation}\label{sec4:V2exterior}
V^{(2)}(\lambda)=
\begin{cases}
T^{-1}\mathcal GT,
&\lambda\in\partial\mathbb D_j
\subset D_1\cup D_3\cup D_5,
\\[2mm]
T^{-1}\mathcal G^{-1}T,
&\lambda\in\partial\mathbb D_j
\subset D_2\cup D_4\cup D_6.
\end{cases}
\end{equation}

\item
The asymptotic conditions are
\[
M^{(2)}(x,t,\lambda)
=I+O(\lambda^{-1}),
\qquad \lambda\to\infty,
\]
\[
M^{(2)}(x,t,\lambda)
=G(x,t)+O(\lambda),
\qquad \lambda\to0.
\]

At each stationary vertex
$p\in\mathcal S(\xi)$, the local behavior
is prescribed by
\[
M^{(2)}(x,t,\lambda)D_p(\lambda)^{-1}
=O(1),
\qquad \lambda\to p,
\]
in every adjacent sector, where
$D_p$ is defined in \eqref{sec4:vertex-gauge}.

\item
In the open components, $M^{(2)}$
satisfies the $\bar\partial$ equation
\[
\bar\partial M^{(2)}(x,t,\lambda)
=
M^{(2)}(x,t,\lambda)W^{(2)}(\lambda), W^{(2)}
=
(\mathcal R^{(2)})^{-1}
\bar\partial\mathcal R^{(2)}.
\]
where
\begin{equation}\label{sec4:W2explicit}
W^{(2)}(\lambda)
=
\begin{cases}
\begin{pmatrix}
0&0&0\\
\bar\partial R_1(\lambda)
 e^{-it\theta_{12}(\lambda)}&0&0\\
0&0&0
\end{pmatrix},
&\lambda\in\Omega_1,
\\[4mm]
\begin{pmatrix}
0&\bar\partial R_2(\lambda)
 e^{it\theta_{12}(\lambda)}&0\\
0&0&0\\
0&0&0
\end{pmatrix},
&\lambda\in\Omega_2,
\\[4mm]
\begin{pmatrix}
0&0&\bar\partial R_3(\lambda)
 e^{it\theta_{13}(\lambda)}\\
0&0&0\\
0&0&0
\end{pmatrix},
&\lambda\in\Omega_3,
\\[4mm]
\begin{pmatrix}
0&0&0\\
0&0&0\\
\bar\partial R_4(\lambda)
 e^{-it\theta_{13}(\lambda)}&0&0
\end{pmatrix},
&\lambda\in\Omega_4,
\\[4mm]
\begin{pmatrix}
0&0&0\\
0&0&0\\
0&\bar\partial R_5(\lambda)
 e^{-it\theta_{23}(\lambda)}&0
\end{pmatrix},
&\lambda\in\Omega_5,
\\[4mm]
\begin{pmatrix}
0&0&0\\
0&0&\bar\partial R_6(\lambda)
 e^{it\theta_{23}(\lambda)}\\
0&0&0
\end{pmatrix},
&\lambda\in\Omega_6,
\\[3mm]
0,
&\lambda\notin\Omega.
\end{cases}
\end{equation}
\item
The only poles are the retained simple
poles. At every $p=\zeta_j\in P_\Lambda$,
\begin{equation}\label{sec4:M2residue}
\operatorname*{Res}_{\lambda=p}
M^{(2)}(x,t,\lambda)
=
\lim_{\lambda\to p}
M^{(2)}(x,t,\lambda)
T(p)^{-1}B_jT(p),
\end{equation}
\end{itemize}
\end{RHP}

Nilpotence gives the identity
\begin{equation}\label{sec4:nilpotentW}
 (\mathcal R^{(2)})^{-1}\bar\partial\mathcal R^{(2)}
       =\bar\partial\mathcal R^{(2)}.
\end{equation}
The two-channel boundary prescriptions yield, in all four cases
of \eqref{VR1},
\begin{equation}\label{sec4:raycancellation}
 (\mathcal R^{(2)}_-)^{-1}V^{(1)}\mathcal R^{(2)}_+=I.
\end{equation}
With $z=\lambda,\omega^2\lambda,\omega\lambda$ for the respective
pairs of lenses, phase decay gives
\begin{equation}\label{sec4:W2decay}
 \|W^{(2)}(\lambda)\|
 \le C|\bar\partial R_j(\lambda)|
          e^{-t|\Im\theta_{12}(z)|},\qquad\lambda\in\Omega_j.
\end{equation}

We separate the analytic and $\bar{\partial}$ parts by
\begin{equation}\label{sec4:M3definition}
 M^{(2)}(\lambda)=M^{(3)}(\lambda)M^R(\lambda),
\end{equation}
where $M^R(\lambda)$ solves the following RH problem and $M^{(3)}(\lambda)$ is treated in Section \ref{s:dbar-contribution}.

\begin{RHP}\label{RHP:4.2}
Find a $3\times3$ matrix-valued function
$M^R(x,t,\lambda)$ satisfying the following properties.
\begin{itemize}

\item
$M^R(x,t,\lambda)$ is analytic in
$\mathbb{C}\setminus(\Sigma^{(2)}\cup\mathcal{P}_{\Lambda})$, and is bounded at zero.
\item
The boundary values satisfy
\[
M^R_+(x,t,\lambda)
=
M^R_-(x,t,\lambda)V^{(2)}(\lambda),
\qquad \lambda\in\Sigma^{(2)},
\]
where the jump matrices $V^{(2)}$ are given by
\eqref{sec4:M2jumpexplicit} for $-1<\xi<1$
and by \eqref{sec4:V2exterior} for
$\xi\in(-\infty,-1]\cup[1,+\infty)$.

\item Asymptotic behavior:
\[
M^R(x,t,\lambda)
=
I+O(\lambda^{-1}).
\]
\item
For $-1<\xi<1$, at every stationary vertex
$p\in\mathcal S(\xi)$, the local behavior
is prescribed by
\[
M^R(x,t,\lambda)D_p(\lambda)^{-1}
=
O(1),
\qquad \lambda\to p,
\]
in each adjacent sector, where $D_p$
is the analytic diagonal gauge defined
in Section~4.3.

\item
The only poles of $M^R$ are the retained
simple poles, with residue conditions
\begin{equation}\label{sec4:MRresidue}
\operatorname*{Res}_{\lambda=p}
M^R(x,t,\lambda)
=
\lim_{\lambda\to p}
M^R(x,t,\lambda)T(p)^{-1}B_jT(p),
\qquad
p=\zeta_j\in\mathcal P_\Lambda.
\end{equation}

\end{itemize}
\end{RHP}

If a solution exists, $\det M^R=1$. Transforming by the
two surviving reductions preserves this class, so
\begin{equation}\label{sec4:MRreductions}
 M^R(\omega\lambda)=\mathcal A M^R(\lambda)\mathcal A^{-1},
 \qquad M^R(-\lambda)=M^R(\lambda)^{-T}.
\end{equation}

The discrete comparison problem $M^r(x,t,\lambda)$ is introduced below. Its
relation to a full reflectionless problem uses the already defined
$T_c$, not a newly defined real logarithmic factor.
\begin{lemma}\label{lem:sec4:reflectionless}
With the factors from \eqref{sec3:scalar-splitting} and
\eqref{THdef}, the transformation
\begin{equation}\label{sec4:restorepoles}
 \widehat M=M^rT^{-1}\mathcal G^{-1}TT_d^{-1}
\end{equation}
converts RH problem~\ref{RHP:4.3} into the no-jump
reflectionless problem on the complete pole set $\mathcal P_N$
with elementary residues
\begin{equation}\label{sec4:continuousresidues}
 \widehat B_p=T_c(p)^{-1}B_pT_c(p).
\end{equation}
Its two independent modified constants are
\begin{equation}\label{sec4:modifiednorming}
 \widetilde c_n=c_n\frac{(T_c)_2(\zeta_n)}{(T_c)_1(\zeta_n)},
 \qquad
 \widetilde d_n=d_n\frac{(T_c)_3(\eta_n)}{(T_c)_2(\eta_n)},
 \qquad 1\le n\le N.
\end{equation}
\end{lemma}

\begin{proof}
Since $T=T_dT_c$ and the diagonal factors
commute, \eqref{sec4:restorepoles} becomes
\[
\widehat M
=M^rT_d^{-1}T_c^{-1}\mathcal G^{-1}T_c.
\]
Outside the pole disks, $\mathcal G=I$,
so $\widehat M=M^rT_d^{-1}$.
The transformation cancels the circle jumps
and changes each retained residue into
$T_c(p)^{-1}B_pT_c(p)$.

To restore the removed poles, set
$H=M^rT_c^{-1}$.
At a decaying pole, the elementary residue
is recovered directly from
$\mathcal G^{-1}=I+B_p/(\lambda-p)$.
At a growing pole, write
$z=\lambda-p$ and $B_p=b_pE_{ij}$.
Then
\[
K_p=T_d^{-1}
\left(I+\frac{z}{b_p}E_{ji}\right)
\left(I-\frac{b_p}{z}E_{ij}\right)
\]
is holomorphic and invertible at $p$.
Indeed, using
$(T_d)_i=a_i/z+O(1)$ and
$(T_d)_j=a_jz+O(z^2)$, its $(i,j)$ block
at $z=0$ is
\[
\begin{pmatrix}
0&-b_p/a_i\\
1/(a_jb_p)&0
\end{pmatrix},
\]
which has nonzero determinant.
Thus
$
HT_d^{-1}\mathcal G^{-1}
=HK_p\left(I+\frac{B_p}{z}\right)
$
has the prescribed elementary residue.
Conjugation by the analytic factor $T_c$
gives \eqref{sec4:continuousresidues}
without derivative corrections.

The remaining residues follow from the
cyclic and negative-transpose symmetries.
The inverse transformation is
\[
M^r=\widehat M T_dT^{-1}\mathcal GT.
\]
Both transformations preserve the endpoint
normalizations, proving equivalence of
the two RH problems.
\end{proof}

In the interior region choose
\begin{equation}\label{sec4:localneighborhoods}
 U=\bigcup_{p\in\mathcal S(\xi)}D(p,\varrho_0),
\end{equation}
with disjoint disks avoiding zero and all pole disks, small enough
that the endpoint cutoffs equal one on the adjacent lips.
The lip jumps are compactly supported by these cutoffs. Outside
$U$ their powers are bounded and the phase has a strict decay
sign, so
\begin{equation}\label{sec4:lipdecay}
 \|V^{(2)}-I\|_{L^1\cap L^2\cap L^\infty
                   (\Sigma^{(ju)}\setminus U)}
      \le Ce^{-ct}.
\end{equation}
Provided the complex local parametrices are constructed with the
same vertex gauges, the analytic comparison has the order
\begin{equation}\label{sec4:purefactorization}
 M^R(\lambda)=
 \begin{cases}
 E(\lambda) M^r(\lambda),&\lambda\notin U,\\
 E(\lambda) M^r(\lambda) M^{lo}(\lambda),&\lambda\in U.
 \end{cases}
\end{equation}
where the local factors $M^{lo}(\lambda)$ and error $E(\lambda)$ are constructed in Section \ref{s:jump-contribution}, while $M^r(\lambda)$ solves the following
discrete problem

\begin{RHP}\label{RHP:4.3}
Find a $3\times3$ matrix-valued function
$M^r(\lambda):=M^r(x,t,\lambda)$ such that
\begin{itemize}

\item
$M^r$ is meromorphic in
$\mathbb C\setminus\Sigma^{(ci)}$,
with only the retained simple poles
at $\mathcal P_\Lambda$.
\item $M^r(\lambda)=I+O(\lambda^{-1}),
\qquad \lambda\to\infty$ as $\lambda\to\infty$ and is bounded at zero.
\item
The boundary values satisfy the jump condition
\[
M^r_+(\lambda)=M^r_-(\lambda)V^r(\lambda),
\qquad \lambda\in\Sigma^{(ci)},
\]
where
\begin{equation}\label{sec4:Vr}
V^r(\lambda)=
\begin{cases}
T(\lambda)^{-1}\mathcal G(\lambda)T(\lambda),
& \lambda\in\partial\mathbb D_j
  \subset D_1\cup D_3\cup D_5,
\\[3mm]
T(\lambda)^{-1}\mathcal G(\lambda)^{-1}T(\lambda),
& \lambda\in\partial\mathbb D_j
  \subset D_2\cup D_4\cup D_6.
\end{cases}
\end{equation}

\item
At every retained pole
$p=\zeta_j\in\mathcal P_\Lambda$,
the residue condition is
\[
\operatorname*{Res}_{\lambda=p}M^r(\lambda)
=
\lim_{\lambda\to p}
M^r(\lambda)T(p)^{-1}B_jT(p).
\]

If the solution exists, it satisfies
$\det M^r(\lambda)=1$ and the two
reduction symmetries
\[
M^r(\omega\lambda)
=
\mathcal A M^r(\lambda)\mathcal A^{-1},
\qquad
M^r(-\lambda)
=
M^r(\lambda)^{-T}.
\]
\end{itemize}
\end{RHP}
\begin{remark}\label{rem:sec4:exterior-model}
For $|\xi|\ge1$ there are no nonzero finite stationary points
on the spectral rays, and all lip jumps are the identity.
Hence $U=\varnothing$ and the two specified analytic problems
coincide: whenever they are solvable, $M^R=M^r$ and the
analytic comparison error is $E=I$.
\end{remark}

\section{Contribution from the discrete spectrum}
\label{s:discrete-contribution}

We compare $M^r$ with the reflectionless model supported on
$\mathcal P_\Lambda$. The pole-circle jumps are exponentially small;
the retained residues determine the leading discrete contribution.

\subsection{$M^r(\lambda)$ and the reflectionless soliton model}
\label{sec5:retainedmodel}

We retain the general paired twelve-point spectrum of
Assumption~\ref{ass:discrete}. Fix $|\xi|\ne1$ and use \eqref{discretesplitting}, the scalar
factor $T=T_dT_c$, and the deletion matrix $\mathcal G$ from
Section~\ref{s:3}. The boundaries of the deleted-pole disks are
separated from zero, the spectral rays, and every retained pole.
Then the circle jumps \eqref{sec4:Vr} satisfy
\begin{equation}\label{polecirclesmallnorm}
 \|V^r-I\|_{L^q(\Sigma^{(ci)})}
       \le C_\xi e^{-\delta_0t},\qquad 1\le q\le\infty.
\end{equation}

Let $N(\Lambda)=|\Lambda|$ count complete paired orbits and set
\[
 \widetilde{\mathcal D}^{\Lambda}(\xi)
 :=\widetilde{\mathcal D}_\Lambda(\xi)
 =\{(\zeta_n,\eta_n,\widetilde c_n^\Lambda,
                         \widetilde d_n^\Lambda):n\in\Lambda\}.
\]
The two effective constants are
\begin{equation}\label{modifiednormingconstant}
\begin{aligned}
 \widetilde c_n^\Lambda
 &=c_n\frac{T_2(\zeta_n)}{T_1(\zeta_n)}
   =\widetilde c_n\frac{(T_d)_2(\zeta_n)}{(T_d)_1(\zeta_n)},\\
 \widetilde d_n^\Lambda
 &=d_n\frac{T_3(\eta_n)}{T_2(\eta_n)}
   =\widetilde d_n\frac{(T_d)_3(\eta_n)}{(T_d)_2(\eta_n)} .
\end{aligned}
\end{equation}
Here $\widetilde c_n,\widetilde d_n$ use only $T_c$, as in
\eqref{sec4:modifiednorming}. For fixed $\xi$ these modifications
are independent of $x,t$. Right multiplication by $T$ changes
$B_p=b_pE_{ij}$ into $T(p)^{-1}B_pT(p)$, with scalar multiplier
$T_j(p)/T_i(p)$. No derivative correction occurs, since
$E_{ij}D E_{ij}=0$ for every diagonal $D$.

Put $\gamma=e^{\pi i/3}=-\omega^2$ and define
$
\kappa_n=\frac{d_n}{\gamma\overline{c_n}},
\qquad
\widetilde\kappa_n^\Lambda
=\frac{\widetilde d_n^\Lambda}
{\gamma\overline{\widetilde c_n^\Lambda}}.
$
By \eqref{sec3:modified-two-constants},
\[
\widetilde\kappa_n^\Lambda
=\kappa_n
\frac{T_3(\eta_n)/T_2(\eta_n)}
{\overline{T_2(\zeta_n)/T_1(\zeta_n)}}.
\]
Since $\eta_n=-\omega^2\overline{\zeta_n}$
and $H_d(-\bar z)=\overline{H_d(z)}$,
the discrete factors cancel:
\[
\frac{(T_d)_3(\eta_n)}{(T_d)_2(\eta_n)}
=\overline{
\frac{(T_d)_2(\zeta_n)}{(T_d)_1(\zeta_n)}}.
\]
This identity concerns only the discrete divisor
and imposes no conjugation relation on $c_n,d_n$.

Writing $F=F_{E(\xi)}$, set
\[
\begin{aligned}
a_{n,c}
&=-2F(\zeta_n)+F(\omega\zeta_n)
+F(\omega^2\zeta_n),\\
b_{n,c}
&=F(\eta_n)-2F(\omega\eta_n)
+F(\omega^2\eta_n).
\end{aligned}
\]
Then
\[
\widetilde\kappa_n^\Lambda
=\kappa_n e^{b_{n,c}-\overline{a_{n,c}}}.
\]
If $d_\Sigma>0$ denotes the distance
of the full pole set from the six rays, then
\[
|a_{n,c}|,\ |b_{n,c}|
\le\frac{4\|f\|_{L^1(\mathbb R_+)}}
{\pi d_\Sigma}.
\]
The estimate is uniform for
$E(\xi)=(0,\lambda_0)$ or $\mathbb R_+$,
and will be used to control the regularity
of the modified discrete data.

\begin{RHP}\label{RHP:modified-reflectionless}
Find a $3\times3$ matrix-valued function
$M^\Lambda(x,t,\lambda)$ satisfying
the following properties.
\begin{itemize}

\item
$M^\Lambda$ is meromorphic in $\mathbb C$,
with at most simple poles at the points of
$
\mathcal P_\Lambda
=\bigcup_{n\in\Lambda}\mathcal O_n,
$
and no other singularities.

\item
$M^\Lambda$ has no jumps and satisfies
the normalization condition
\[
M^\Lambda(x,t,\lambda)
=I+O(\lambda^{-1}),
\qquad \lambda\to\infty.
\]

\item
At each $\zeta_j\in\mathcal P_\Lambda$, the residue condition is \eqref{sec4:MRresidue}, with $M^\Lambda$ in place of $M^R$.

\end{itemize}
\end{RHP}

Whenever it exists, this problem has determinant one and satisfies
only the two inherited reductions
\begin{equation}\label{sec5:modelreductions}
 M^\Lambda(\omega\lambda)
       =\mathcal A M^\Lambda(\lambda)\mathcal A^{-1},\qquad
 M^\Lambda(-\lambda)=M^\Lambda(\lambda)^{-T}.
\end{equation}

We consider the reflectionless model with
$\Lambda=\varnothing$ or $\Lambda=\{n_0\}$.
For $\Lambda=\varnothing$, the solution is
$M^\Lambda=I$ and $q^\Lambda=1$.
For $\Lambda=\{n_0\}$, the discrete spectrum
consists of the complete twelve-point orbit
\[
\mathcal P_\Lambda
=\{\pm\omega^j\zeta,\,
\pm\omega^j\eta:j=0,1,2\},
\qquad
\eta=\gamma\bar\zeta,\quad
\gamma=e^{\pi i/3}.
\]
The reconstructed field is
\begin{equation}\label{sec5:retainedreconstruction}
q^\Lambda(x,t)
=\boldsymbol\ell M^\Lambda(x,t,0)e_3,
\qquad
\boldsymbol\ell=(\omega,\omega^2,1).
\end{equation}

For the retained reflectionless model, we consider
$\Lambda=\varnothing$ or $\Lambda=\{n_0\}$.
If $\Lambda=\varnothing$, then $M^\Lambda=I$.
For $\Lambda=\{n_0\}$, let
$
\zeta=\rho e^{i\theta},
\eta=\gamma\bar\zeta,
\gamma=e^{\pi i/3},
$
where $\pi/6<\theta<\pi/3$.
The complete discrete spectrum is
$
\mathcal P_\Lambda
=\{\pm\omega^\ell\zeta,\,
\pm\omega^\ell\eta:\ell=0,1,2\}.
$
Set
\[
s=-\frac{\widetilde c_{n_0}^\Lambda
e^{\vartheta_{12}(\zeta)}}{\zeta},
r=-\frac{\widetilde d_{n_0}^\Lambda
e^{\vartheta_{23}(\eta)}}{\eta}
=\kappa\bar s,
\kappa=\widetilde\kappa_{n_0}^\Lambda.
\]
Define
\[
h=\frac{\cot^2\theta(1+2\cos2\theta)}
{12(1-2\cos2\theta)},
\qquad
e_\theta=
\frac{1-4\cos4\theta}
{6\sin^2\theta(1-2\cos2\theta)},
\]
and
\[
\mathfrak a
=1-\frac{i}{2\sqrt3}(s+r)-hsr,
\mathfrak b
=1+\frac{2i}{\sqrt3}(s+r)
-\frac{s^2+r^2}{12}
+e_\theta sr
\quad-\frac{2ih}{\sqrt3}sr(s+r)
+h^2s^2r^2.
\]
We impose the regularity conditions
\begin{equation}\label{sec5:singleorbitregularity}
|\kappa|=1,\qquad
\frac{\pi}{6}<\theta<\frac{\pi}{4},
\qquad
6h(1-\Re\kappa)>1,\qquad
\Im\kappa\ne0.
\end{equation}

\begin{lemma}\label{lem:sec5:finitecriterion}
Under \eqref{sec5:singleorbitregularity},
RH problem~\ref{RHP:modified-reflectionless}
admits a unique solution. Moreover,
\begin{equation}\label{sec5:singleorbitRH}
M^\Lambda(\lambda)
=I+\sum_{\ell=0}^{2}
\omega^\ell\mathcal A^\ell
\bigg(
\frac{\mathbf u e_2^T}
{\lambda-\omega^\ell\zeta}
+\frac{\mathbf v e_3^T}
{\lambda-\omega^\ell\eta}
+\frac{\mathbf w e_1^T}
{\lambda+\omega^\ell\zeta}
+\frac{\mathbf z e_2^T}
{\lambda+\omega^\ell\eta}
\bigg)\mathcal A^{-\ell},
\end{equation}
where $\mathbf u,\mathbf v,\mathbf w,\mathbf z
\in\mathbb C^3$ are uniquely determined
by the following system. Set
\[
\begin{aligned}
(p_1,p_2,p_3,p_4)
&=(\zeta,\eta,-\zeta,-\eta),
(b_1,b_2,b_3,b_4)
=(\zeta s,\eta r,\zeta s,\eta r),\\
(i_1,i_2,i_3,i_4)&=(1,2,2,3),
(j_1,j_2,j_3,j_4)=(2,3,1,2),\\
(\mathbf U_1,\mathbf U_2,\mathbf U_3,\mathbf U_4)
&=(\mathbf u,\mathbf v,\mathbf w,\mathbf z).
\end{aligned}
\]
Then
\begin{equation}\label{sec5:singleorbitvectors}
\mathbf U_a
={}b_a e_{i_a}
+b_a
\sum_{\substack{1\le b\le4,\ 0\le\ell\le2\\
(b,\ell)\ne(a,0)}}
\frac{\omega^\ell
(e_{j_b}^T\mathcal A^{-\ell}e_{i_a})
\mathcal A^\ell\mathbf U_b}
{p_a-\omega^\ell p_b},
\quad a=1,\ldots,4.
\end{equation}

Furthermore, the reconstructed field satisfies
\begin{equation}\label{sec5:singleorbitformula}
q^\Lambda(x,t)
=\boldsymbol\ell M^\Lambda(x,t,0)e_3
=\frac{\mathfrak b}{\mathfrak a^2},
\qquad
\boldsymbol\ell=(\omega,\omega^2,1).
\end{equation}
Both $q^\Lambda$ and $(q^\Lambda)^{-1}$
are bounded, and $M^\Lambda$ and
$(M^\Lambda)^{-1}$ are uniformly bounded.

The continuous logarithm
$u^\Lambda=\log q^\Lambda$, normalized by
$u^\Lambda(+\infty,t)=0$, defines a regular
reflectionless Tzitz\'eica solution
with differentiated exponential spatial tails.
Its topological charge is $1$ on the upper
regular branch and $-1$ on the lower branch.
\end{lemma}

\begin{proof}
Substituting \eqref{sec5:singleorbitRH}
into the residue conditions yields a
finite linear system for the residue vectors.
Using the cyclic symmetry and the discrete
Fourier transform, the system reduces to
three $4\times4$ blocks whose determinants
are $\mathfrak a^2$, $\mathfrak a^2$, and
$\mathfrak b$, respectively. Consequently,
\[
\det\mathsf C=\mathfrak a^4\mathfrak b.
\]
The reconstruction at $\lambda=0$ gives
$
q^\Lambda=\frac{\mathfrak b}{\mathfrak a^2}.
$

Under \eqref{sec5:singleorbitregularity},
writing
$\kappa=e^{2i\phi}$ and
$s=e^{i\phi}(X+iY)$,
one verifies that
$\mathfrak a$ and $\mathfrak b$
have no zeros for $s\in\mathbb C$.
Thus the RH problem is uniquely solvable.
Radial compactification provides positive
lower bounds for
$
\frac{|\mathfrak a|}{1+|s|^2},
\frac{|\mathfrak b|}{(1+|s|^2)^2},
$
which imply the stated uniform bounds.

The residue evolution and reconstruction
recover the Tzitz\'eica Lax pair.
Since $q^\Lambda$ is nonvanishing,
its continuous logarithm defines a
regular reflectionless solution.

Finally, the winding number is constant
on each connected regular branch.
Deformation to the reference parameters
$\kappa=i$ and $\cot^2\theta=2$ gives
$
\Delta\arg\mathfrak a=-\frac{\pi}{2},
\Delta\arg\mathfrak b=\pi,
$
and hence
$
\Delta\arg q^\Lambda=2\pi.
$
The conjugate parameter branch has the
opposite winding. The spatial asymptotics
of the rational reconstruction yield
the stated exponential tails.
\end{proof}

\begin{proposition}[An explicit globally regular
one-orbit reflectionless solution]
\label{prop:sec5:regular-one-orbit-example}
Consider RH problem~\ref{RHP:modified-reflectionless}
with $\Lambda=\{1\}$ and one complete paired orbit
\[
\mathcal P_\Lambda
=
\{\pm\omega^\ell\zeta,\,
  \pm\omega^\ell\eta:\ell=0,1,2\},
\]
where
$
\zeta=\frac{\sqrt2+i}{\sqrt3},
\eta=\gamma\bar\zeta,
\gamma=e^{\pi i/3}.
$
Choose the effective norming constants
$
\widetilde c_1^\Lambda=-\zeta e^{i\pi/4},
\widetilde d_1^\Lambda=-\eta e^{i\pi/4}.
$
Then the RH problem has a unique solution
for every $(x,t)\in\mathbb R^2$.

Put
$
R=e^{-x},
X=R\cos(\sqrt2\,t),
$
and define
\begin{equation}\label{sec5:regular-example-polynomials}
\begin{aligned}
\mathfrak a_*
&=
1+\frac{1-i}{\sqrt6}X-\frac{5i}{6}R^2,
\\
\mathfrak b_*
&=
\left(1-\frac56R^2\right)
\left(1+\frac56R^2-\frac4{\sqrt6}X\right)
+i\left[
\frac4{\sqrt6}X
\left(1+\frac56R^2\right)
+\frac{19}{3}R^2-\frac13X^2
\right].
\end{aligned}
\end{equation}
The reconstructed field is
\begin{equation}\label{sec5:regular-example-field}
q_*(x,t)
=
1-
\frac{
6\left[
4+\sqrt{10}\cos(\sqrt2\,t)
\cosh\!\left(
x+\frac12\log\frac65+\frac{i\pi}{4}
\right)
\right]
}{
\left[
\sqrt{10}\cosh\!\left(
x+\frac12\log\frac65+\frac{i\pi}{4}
\right)
+\cos(\sqrt2\,t)
\right]^2
}.
\end{equation}
There are constants $m,M>0$ such that
\[
m\le |q_*(x,t)|\le M,
\qquad (x,t)\in\mathbb R^2.
\]

The continuous logarithm
\[
u_*(x,t)=\log_{\mathrm{cont}}q_*(x,t),
\qquad
u_*(+\infty,t)=0,
\]
is a global smooth solution of
\[
(u_*)_{tt}-(u_*)_{xx}
=e^{-2u_*}-e^{u_*}.
\]
Its spatial limits are
\[
u_*(+\infty,t)=0,
\qquad
u_*(-\infty,t)=2\pi i,
\]
with differentiated exponential tails.
In particular, its Cauchy data belong to
$\mathcal X_1^{\mathcal S}$.
\end{proposition}

\begin{proof}
The chosen spectral point satisfies
$
\cot^2\theta=2,
h=\frac56,
e_\theta=\frac{37}{6}.
$
Moreover,
$
\kappa
=
\frac{\widetilde d_1^\Lambda}
{\gamma\overline{\widetilde c_1^\Lambda}}
=i.
$
The phase functions give
$
\vartheta_{12}(\zeta)
=-x+i\sqrt2\,t,
\vartheta_{23}(\eta)
=-x-i\sqrt2\,t.
$
Consequently,
\[
s=e^{i\pi/4}e^{-x+i\sqrt2\,t},
\qquad
r=e^{i\pi/4}e^{-x-i\sqrt2\,t}
=i\bar s.
\]
Substitution into the finite reconstruction
\eqref{sec5:singleorbitformula} gives
\eqref{sec5:regular-example-polynomials}.

We first prove that $\mathfrak a_*$ never vanishes.
Since $|X|\le R$, a zero would imply
$
1+\frac{X}{\sqrt6}=0,
-\frac{X}{\sqrt6}-\frac56R^2=0.
$
Thus $X=-\sqrt6$ and $R^2=6/5$,
contradicting $X^2\le R^2$.

Next put
$
Y=R\sin(\sqrt2\,t),
 R^2=X^2+Y^2.
$
The second factor in $\Re\mathfrak b_*$ satisfies
\[
1+\frac56R^2-\frac4{\sqrt6}X
=
\frac56
\left[
\left(X-\frac{2\sqrt6}{5}\right)^2+Y^2
\right]
+\frac15
>0.
\]
Hence $\Re\mathfrak b_*=0$ can occur only when
$R^2=6/5$.
On that circle,
\[
\Im\mathfrak b_*
=
\frac{38}{5}
-\frac13X^2
+\frac8{\sqrt6}X.
\]
Since $|X|\le\sqrt{6/5}$,
$
\Im\mathfrak b_*
\ge
\frac{36}{5}-\frac8{\sqrt5}
>0.
$
Therefore $\mathfrak b_*$ has no zeros either.

The finite residue reduction gives
$
\det\mathsf C
=\mathfrak a_*^4\mathfrak b_*.
$
It is nonzero at every real $(x,t)$,
so the residue system is uniquely solvable
and defines the required RH solution.

As $R\to0$,
$
\mathfrak a_*=1+O(R),
\mathfrak b_*=1+O(R).
$
As $R\to\infty$,
$
\mathfrak a_*=-\frac{5i}{6}R^2+O(R),
\mathfrak b_*=-\frac{25}{36}R^4+O(R^3).
$
These estimates are uniform in the temporal phase.
Thus $q_*\to1$ uniformly at both spatial ends.
Compactification in $R$, together with
$|X|\le R$ and the proved nonvanishing,
gives uniform upper and lower bounds for $|q_*|$.

A global smooth logarithm can be defined by
\[
u_*(x,t)
=
-\int_x^\infty
\frac{\partial_y q_*(y,t)}{q_*(y,t)}\,dy.
\]
The differentiated exponential tails justify
differentiation of this integral, and
$e^{u_*}=q_*$.

For completeness, the scalar equation can also be
checked directly.
For a smooth scalar function $P$, write
\[
\mathscr L(P)
=
P(P_{tt}-P_{xx})-P_t^2+P_x^2.
\]
Substitution of
\eqref{sec5:regular-example-polynomials} gives
\[
\mathfrak a_*^2\mathscr L(\mathfrak b_*)
-2\mathfrak b_*^2\mathscr L(\mathfrak a_*)
=
\mathfrak a_*^6-\mathfrak b_*^3.
\]
Dividing by
$\mathfrak a_*^2\mathfrak b_*^2$
yields
\[
(u_*)_{tt}-(u_*)_{xx}
=q_*^{-2}-q_*
=e^{-2u_*}-e^{u_*}.
\]

Finally, at $t=0$ one has $X=R$.
For $R>0$,
\[
\Re\mathfrak a_*>0,
\qquad
\Im\mathfrak a_*<0,
\]
while
$
\Im\mathfrak b_*
=
\frac4{\sqrt6}R
\left(1+\frac56R^2\right)+6R^2>0.
$
As $R$ increases from zero to infinity,
the continuous arguments of
$\mathfrak a_*$ and $\mathfrak b_*$
change by $-\pi/2$ and $\pi$, respectively.
Therefore the continuous argument of
$q_*=\mathfrak b_*/\mathfrak a_*^2$
changes by $2\pi$.
This gives $u_*(-\infty,0)=2\pi i$. The left endpoint remains an integer multiple
of $2\pi i$ and depends continuously on $t$,
so it is $2\pi i$ for every $t$.
The rational formulas and their differentiated
endpoint expansions give the stated exponential tails
and the membership in $\mathcal X_1^{\mathcal S}$.
\end{proof}

\begin{corollary}\label{corollary5.2}
For fixed $\xi$, the unique solution
$M^\Lambda(\lambda)$ of RH problem~
\ref{RHP:modified-reflectionless}
defines the reflectionless soliton model
\begin{equation}\label{sec5:solitonmodel}
M_{\mathrm{sol}}^\Lambda(x,t,\lambda;\xi)
:=M^\Lambda(x,t,\lambda)
=M^{\mathrm{sol}}(x,t,\lambda\mid
\widetilde{\mathcal D}^{\Lambda}(\xi)).
\end{equation}
Its reconstruction is
\begin{equation}\label{sec5:solitonreconstruction}
\begin{aligned}
q_{\mathrm{sol}}^\Lambda(x,t;\xi)
&:=\boldsymbol\ell
M_{\mathrm{sol}}^\Lambda(x,t,0;\xi)e_3
=q^\Lambda(x,t;\xi),\\
u_{\mathrm{sol}}^\Lambda(x,t;\xi)
&:=\log_{\mathrm{cont}}
q_{\mathrm{sol}}^\Lambda(x,t;\xi),
\end{aligned}
\end{equation}
where the continuous logarithm is normalized
by $u_{\mathrm{sol}}^\Lambda(+\infty,t;\xi)=0$. For $\Lambda=\{n_0\}$, this gives the regular
single-soliton solution with topological
charge $\pm1$. For $\Lambda=\varnothing$,
$M_{\mathrm{sol}}^\varnothing=I$,
$q_{\mathrm{sol}}^\varnothing=1$, and
$u_{\mathrm{sol}}^\varnothing=0$.
\end{corollary}

\subsection{Residual error between $M^r(\lambda)$ and
$M^\Lambda(\lambda)$}
\label{sec5:discreteerror}

Where both discrete models exist, set
\begin{equation}
M^{\mathrm{err}}(\lambda)
=M^r(\lambda)(M^\Lambda(\lambda))^{-1}.
\nonumber
\end{equation}
Their identical local residue factors cancel, giving the following
normalized pure-jump problem. Conversely, its solution constructs
$M^r=M^{\mathrm{err}}M^\Lambda$.

\begin{RHP}\label{RHP:sec5:error}
Find a $3\times3$ matrix-valued function
$M^{\mathrm{err}}(x,t,\lambda)$ such that
\begin{itemize}

\item
$M^{\mathrm{err}}$ is analytic in
$\mathbb C\setminus\Sigma^{(ci)}$,
has continuous boundary values on each circle, and has no poles.

\item
$
M^{\mathrm{err}}(x,t,\lambda)
=I+O(\lambda^{-1}),
\qquad \lambda\to\infty.
$

\item
$M^{\mathrm{err}}$ satisfies the jump condition
\[
M^{\mathrm{err}}_+(x,t,\lambda)
=
M^{\mathrm{err}}_-(x,t,\lambda)
V_{\mathrm{err}}(x,t,\lambda),
\qquad \lambda\in\Sigma^{(ci)},
\]
where
$
V_{\mathrm{err}}(x,t,\lambda)
=
M^\Lambda(x,t,\lambda)
V^r(x,t,\lambda)
\bigl(M^\Lambda(x,t,\lambda)\bigr)^{-1},
$
 $V^r$ is given by \eqref{sec4:Vr}.

\end{itemize}
\end{RHP}

By the model bounds and \eqref{polecirclesmallnorm},
\[
 \|w_{\mathrm{err}}\|_{L^1\cap L^2\cap L^\infty(\Sigma^{(ci)})}
       \le C_\xi e^{-\delta_0t},\qquad
 w_{\mathrm{err}}=V_{\mathrm{err}}-I .
\]
Use the actual clockwise or counterclockwise circle orientations:
\[
 C f(\lambda)=\frac1{2\pi i}
       \int_{\Sigma^{(ci)}}\frac{f(s)}{s-\lambda}\,ds,\qquad
 C_+-C_-=I .
\]
The Beals--Coifman representation is
\begin{equation}\label{sec5:errorrepresentation}
 M^{\mathrm{err}}(\lambda)
       =I+C(\mu_{\mathrm{err}}w_{\mathrm{err}})(\lambda),
\end{equation}
where $\mu_{\mathrm{err}}-I\in L^2$ solves
$
 (I-C_{\mathrm{err}})(\mu_{\mathrm{err}}-I)=C_{\mathrm{err}}I,
 C_{\mathrm{err}}f=C_-(fw_{\mathrm{err}}).
$ The Cauchy projection is bounded on this fixed finite family
of separated smooth circles. Hence
\begin{equation}\label{sec5:erroroperatorbound}
 \|C_{\mathrm{err}}\|_{L^2\to L^2}
       \le\|C_-\|\,\|w_{\mathrm{err}}\|_\infty
       \le C_\xi e^{-\delta_0t}<\tfrac12
\end{equation}
for sufficiently large $t$. The Neumann series gives
\[
 \|(I-C_{\mathrm{err}})^{-1}\|\le2,\qquad
 \|\mu_{\mathrm{err}}-I\|_2\le C_\xi e^{-\delta_0t}.
\]
The Plemelj formulas verify the jump and show that
$\mu_{\mathrm{err}}$ is its minus trace. The jump extends
analytically and invertibly to an annulus around each circle;
gluing with that extension gives continuous traces. Its
determinant is one. Removability and Liouville's theorem give
$\det M^{\mathrm{err}}=1$ and uniqueness. Consequently the
error and the constructed $M^r$ obey the cyclic and negative
inverse-transpose reductions; no reality reduction is used.

At zero write
$M^{\mathrm{err}}(\lambda)=M^{\mathrm{err}}(0)
+\lambda M^{\mathrm{err}}_1+O(\lambda^2)$. Explicitly,
\begin{equation}\label{sec5:errorzerointegrals}
 M^{\mathrm{err}}(0)=I+\frac1{2\pi i}
       \int_{\Sigma^{(ci)}}\frac{\mu_{\mathrm{err}}w_{\mathrm{err}}}{s}\,ds,
 M^{\mathrm{err}}_1=\frac1{2\pi i}
       \int_{\Sigma^{(ci)}}\frac{\mu_{\mathrm{err}}w_{\mathrm{err}}}{s^2}\,ds .
\end{equation}

\begin{lemma}\label{lem:sec5:discreteerror}
RH problem~\ref{RHP:sec5:error}, and hence RH
problem~\ref{RHP:4.3}, is uniquely solvable for sufficiently
large time. Uniformly a fixed positive distance from the circles,
\begin{equation}\label{discreteerrorestimate}
 M^{\mathrm{err}}(\lambda)
       =I+O\!\left(\frac{e^{-\delta_0t}}{1+|\lambda|}\right).
\end{equation}
At zero,
\begin{equation}\label{discreteerrorzero}
 \|M^{\mathrm{err}}(0)-I\|+\|M^{\mathrm{err}}_1\|
       \le C_\xi e^{-\delta_0t}.
\end{equation}
\end{lemma}
\begin{proof}
The operator estimate proves existence and uniqueness. Moreover
\[
 \|\mu_{\mathrm{err}}w_{\mathrm{err}}\|_1
 \le\|w_{\mathrm{err}}\|_1+
       \|\mu_{\mathrm{err}}-I\|_2\|w_{\mathrm{err}}\|_2
 \le C_\xi e^{-\delta_0t}.
\]
Insert this bound into \eqref{sec5:errorrepresentation}.
The kernel is bounded off the circles and is $O(|\lambda|^{-1})$
at infinity. The kernels $s^{-1},s^{-2}$ in
\eqref{sec5:errorzerointegrals} are bounded because the contour
stays a positive distance from zero. This proves both estimates.
\end{proof}

\begin{proposition}\label{prop:sec5:discretecomparison}

\[
 M^r(\lambda)=
 \left[I+O\!\left(\frac{e^{-\delta_0t}}{1+|\lambda|}\right)\right]
 M^\Lambda(\lambda)
\]
off the circles and away from the retained poles. Put
$q^r=\boldsymbol\ell M^r(0)e_3$. Then
\begin{equation}\label{sec5:reconstructedcomparison}
 q^r-q^\Lambda=O(e^{-\delta_0t}),\qquad
 \frac{q^r}{q^\Lambda}=1+O(e^{-\delta_0t}).
\end{equation}
The restored full model in \eqref{sec4:restorepoles} has the
same scalar $q^r$. For matched continuous logarithms,
$u^r-u^\Lambda=O(e^{-\delta_0t})$.
For independently normalized lifts on a connected comparison
domain the precise statement instead is
$u^r-u^\Lambda=2\pi i k+O(e^{-\delta_0t})$, $k\in\mathbb Z$.
\end{proposition}
\begin{proof}
The matrix relation follows with its multiplier on the left.
The endpoint identities $T_d(0)=T_c(0)=I$ and $\mathcal G=I$
near zero give $\widehat M(0)=M^r(0)$. The uniform bound for
$M^\Lambda(0)$ and \eqref{discreteerrorzero} give the absolute
scalar estimate. There is also an exact relative identity.
Since $M^{\mathrm{err}}(0)$ commutes with $\mathcal A$,
\[
 \boldsymbol\ell M^{\mathrm{err}}(0)
       =q_{\mathrm{err}}\boldsymbol\ell,\qquad
 q_{\mathrm{err}}=\boldsymbol\ell M^{\mathrm{err}}(0)e_3
       =1+O(e^{-\delta_0t}),
 \qquad q^r=q_{\mathrm{err}}q^\Lambda .
\]
For large $t$, $|q_{\mathrm{err}}-1|<1/2$ and its disk logarithm
satisfies
$|\operatorname{Log}q_{\mathrm{err}}|\le2|q_{\mathrm{err}}-1|$.
Define the matched lift by
$u^r=u^\Lambda+\operatorname{Log}q_{\mathrm{err}}$.
Any other continuous lift differs by a constant $2\pi i k$
on each connected comparison domain. Identifying that integer
with the globally normalized Cauchy solution requires a
topological normalization; it does not follow from the
exponentially small matrix error.
\end{proof}

\section{Contribution from the jump contours}
\label{s:jump-contribution}

We construct the local models at the six stationary points and
compute their contribution to \eqref{sec4:purefactorization}.

\subsection{Local model near phase points}

Fix $K\Subset(-1,1)$ on which the retained and deleted sets
are constant and all the spectral separations and background
bounds in Section~\ref{sec5:retainedmodel} are uniform.
Use the flat scattering data of Section~\ref{s:4}.
For the explicit Weber construction we assume, in addition,
\[
 \|\alpha\beta\|_{L^\infty(\mathbb R_+)}\le\rho_*<1.
\]
Thus the continuous logarithm $f$ in
\eqref{sec3:continuous-log} is the principal logarithm of
$\Delta=\delta=1-\alpha\beta$. Define
\[
 \nu(s):=-\frac{f(s)}{2\pi},\qquad s>0.
\]
For $\xi\in K$, write
\[
 a=a(\xi)=\sqrt{\frac{1-\xi}{1+\xi}},\qquad
 \nu=\nu(a(\xi)),\qquad
 b_K=\sup_{\xi\in K}|\Im\nu(a(\xi))|.
\]
Then $b_K<1/4$. Use precisely the lenses, orientations, and disks $U_p$ of
\eqref{sec4:innerfixedpositions} and
\eqref{sec4:localneighborhoods}. Choose their common radius
small enough that the endpoint cutoff $\chi_i$ equals one on
a fixed neighborhood of every closed disk. Set
\[
 \mathcal S(\xi)=\{\pm a,\pm\omega a,\pm\omega^2a\},\qquad
 \Sigma^{lo}=\Sigma^{(ju)}\cap U,\qquad
 V^{lo}=V^{(2)}|_{\Sigma^{lo}}.
\]

\begin{RHP}\label{RHP:sec6:local}
Find a $3\times3$ matrix-valued function
$M^{lo}(x,t,\lambda)$, separately in each disk
$U_p$, $p\in\mathcal S(\xi)$, satisfying
the following properties.
\begin{itemize}

\item
$M^{lo}(\lambda)$ is analytic in
$U_p\setminus\Sigma^{lo}$ and has no poles
in $U_p$.

\item
$M^{lo}(\lambda)$ satisfies the jump condition
\[
M^{lo}_+(\lambda)
=
M^{lo}_-(\lambda)V^{lo}(\lambda),
\qquad
\lambda\in\Sigma^{lo}\cap U_p,
\]
where
\[
V^{lo}(\lambda)
=
V^{(2)}(\lambda)
\big|_{\Sigma^{lo}\cap U_p}.
\]

\item
At the stationary point $p$, the local
behavior is prescribed by
\[
M^{lo}(\lambda)D_p(\lambda)^{-1}
=O_t(1),
\qquad \lambda\to p,
\]
in each adjacent sector, where $D_p$
is defined in \eqref{sec4:vertex-gauge}.

\item Asymptotic behavior:
\[
M^{lo}(x,t,\lambda)
\bigl(M_p^{pc}(x,t,\lambda)\bigr)^{-1}
=
I+O_K\!\left(t^{-1/2+b_K}\right),
\qquad
\lambda\in\partial U_p,\quad t\to\infty.
\]

\end{itemize}
\end{RHP}

Let $\Gamma_p=\Sigma^{lo}\cap U_p$ and
$w_p=(V^{lo}-I){\bf1}_{\Gamma_p}$. The symbol $w_p$ denotes a
jump weight; the Morse variable $w$ below is a spectral coordinate.

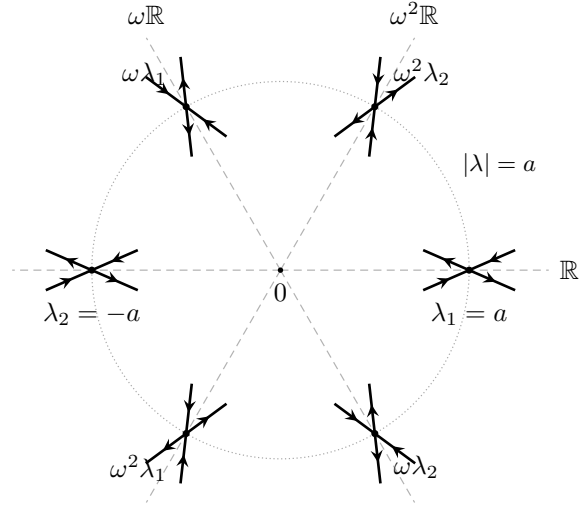
\begin{figure}[H]
\centering

\begin{tikzpicture}[
    >=stealth,
    scale=0.78,
    localcross/.style={
        black,
        line width=1.0pt,
        postaction={decorate},
        decoration={
            markings,
            mark=at position 0.58 with {\arrow{stealth}}
        }
    },
    every node/.style={
        font=\small
    }
]

\def\rr{3.20}      
\def\rin{2.42}     
\def\rout{3.98}    
\def\hh{0.34}      
\def\Rline{4.55}   

\draw[
    densely dotted,
    black!35,
    line width=0.45pt
]
(0,0) circle (\rr);

\foreach \ang in {0,60,120,180,240,300}{
    \begin{scope}[rotate=\ang]

        \draw[
            densely dashed,
            black!30,
            line width=0.45pt
        ]
        (0,0)--(\Rline,0);


        \draw[localcross]
        (\rr,0)--(\rin,\hh);

        \draw[localcross]
        (\rin,-\hh)--(\rr,0);

        \draw[localcross]
        (\rout,\hh)--(\rr,0);

        \draw[localcross]
        (\rr,0)--(\rout,-\hh);

        \fill[black]
        (\rr,0) circle (1.7pt);

    \end{scope}
}

\fill[black] (0,0) circle (1.3pt);

\node[
    anchor=north,
    inner sep=0pt,
    yshift=-5pt
]
at (0,0)
{$0$};


\node[
    anchor=north,
    inner sep=0pt,
    yshift=-12pt
]
at (3.20,0)
{$\lambda_1=a$};

\node[
    anchor=north,
    inner sep=0pt,
    yshift=-12pt
]
at (-3.20,0)
{$\lambda_2=-a$};

\node[
    anchor=south east,
    inner sep=0pt,
    xshift=-7pt,
    yshift=8pt
]
at (120:3.20)
{$\omega\lambda_1$};

\node[
    anchor=south west,
    inner sep=0pt,
    xshift=7pt,
    yshift=8pt
]
at (60:3.20)
{$\omega^2\lambda_2$};

\node[
    anchor=north east,
    inner sep=0pt,
    xshift=-7pt,
    yshift=-8pt
]
at (240:3.20)
{$\omega^2\lambda_1$};

\node[
    anchor=north west,
    inner sep=0pt,
    xshift=7pt,
    yshift=-8pt
]
at (300:3.20)
{$\omega\lambda_2$};


\node[
    anchor=west,
    inner sep=0pt,
    xshift=4pt
]
at (\Rline,0)
{$\mathbb R$};

\node[
    anchor=south,
    inner sep=0pt,
    yshift=5pt
]
at (120:\Rline)
{$\omega\mathbb R$};

\node[
    anchor=south,
    inner sep=0pt,
    yshift=5pt
]
at (60:\Rline)
{$\omega^2\mathbb R$};

\node[
    anchor=west,
    inner sep=0pt,
    font=\footnotesize
]
at (3.08,1.78)
{$|\lambda|=a$};

\end{tikzpicture}

\caption{
The local jump contour $\Sigma^{lo}$ at the six stationary
points $\omega^\ell\lambda_i$, $i=1,2$, $\ell=0,1,2$, for
$-1<\xi<1$, where $\lambda_1=a$, $\lambda_2=-a$, and
$a=\sqrt{(1-\xi)/(1+\xi)}$.
The solid lines form the local crosses and inherit the orientations
of the lens contours in Section~\ref{s:4}.
The dashed spectral lines and the dotted circle $|\lambda|=a$
are shown only for reference and carry no jump of the local problem.
}
\label{fig:sec6:localcontour}

\end{figure}

\begin{lemma}\label{lem:sec6:localweights}
For $1\le q<\infty$ with $2qb_K<1$,
\[
 \|w_p\|_{L^q}\le C_Kt^{-1/(2q)+b_K},\qquad t\ge1.
\]
In particular the $L^1,L^2$ bounds have orders
$t^{-1/2+b_K},t^{-1/4+b_K}$.
\end{lemma}
\begin{proof}
The endpoint traces in Lemma~\ref{lem:sec4:innerextension}
and \eqref{sec4:quadraticdecay} give
$|w_p|\le C_Kr^{-2b_K}e^{-c_Ktr^2}$, $r=|\lambda-p|$.
Arclength is comparable to $dr$. Substitution $v=\sqrt t\,r$
in $\int_0^{\varrho_0}r^{-2qb_K}e^{-qctr^2}dr$ proves the
bound and its integrability condition. The six supports are
disjoint.
\end{proof}

\begin{lemma}\label{lem:sec6:localinteraction}
For $p\ne q$, the separated Cauchy composition, initially defined
by ordinary integrals between the distinct supports, satisfies
\begin{equation}\label{sec6:mixedCauchyoperators}
 \left\|C_+\!\left({\bf1}_{\Gamma_p}
              C(fw_q)w_p\right)\right\|_2
 \le C_K\|w_p\|_2\|w_q\|_2\|f\|_2
 =O(t^{-1/2+2b_K})\|f\|_2.
\end{equation}
If $a_p$ are uniformly bounded on their supports and
$F_p=C(a_pw_p)$, then
\begin{equation}\label{sec6:separatedCauchyfields}
 \|F_p\|_{L^\infty(\Gamma_q)}=O(t^{-1/2+b_K}),\qquad
 C({\bf1}_{\Gamma_p}F_qa_pw_p)
                    =O(t^{-1+2b_K})
\end{equation}
on compact sets separated from $\Sigma^{lo}$.
\end{lemma}
\begin{proof}
Separation gives
$\sup_{\Gamma_p}|C(fw_q)|\le C_K\|f\|_2\|w_q\|_2$.
Multiplication by $w_p$ and the bounded $L^2$ Cauchy projection
give \eqref{sec6:mixedCauchyoperators}. Likewise
$\sup_{\Gamma_q}|F_p|\le C_K\|a_p\|_\infty\|w_p\|_1$;
a second separated integral proves
\eqref{sec6:separatedCauchyfields}. This does not assert that
each unweighted center-containing operator $f\mapsto C_+(fw_p)$
is bounded on $L^2$. We instead solve the matched problem away
from the centers in Subsection~\ref{s:dbar-problem}.
\end{proof}

\begin{RHP}\label{RHP:sec6:positive}
Find a $3\times3$ matrix-valued function
$M^{lo}_{1,0}(\lambda):=
M^{lo}_{1,0}(\lambda;x,t)$ such that
\begin{itemize}

\item
$M^{lo}_{1,0}$ is analytic in
$
U_a\setminus
\bigl((\Sigma_1\cup\Sigma_2)\cap U_a\bigr),
$
has no poles in $U_a$, and satisfies
the local endpoint condition
\[
M^{lo}_{1,0}(\lambda)D_a(\lambda)^{-1}
=O_t(1),
\qquad \lambda\to a,
\]
in each adjacent sector.

\item
Its boundary values satisfy
\begin{equation}\label{sec6:positiveLocalJump}
(M^{lo}_{1,0})_+(\lambda)
=
(M^{lo}_{1,0})_-(\lambda)
V^{lo}_{1,0}(\lambda),
\end{equation}
where
\[
V^{lo}_{1,0}(\lambda)
=
V^{(2)}(\lambda)
\big|_{(\Sigma_1\cup\Sigma_2)\cap U_a}.
\]

\item
The normalization is fixed by the
canonical parabolic-cylinder model
continued to the physical lens contours:
\begin{equation}\label{sec6:positivePCmatching}
M^{lo}_{1,0}(\lambda)
=
M^{pc}_{1,0}(\lambda),
\qquad
\lambda\in U_a\setminus\Sigma^{lo}.
\end{equation}
Consequently, for $\xi$ in a fixed
compact subset $K\Subset(-1,1)$,
the matching condition is
\[
M^{lo}_{1,0}(\lambda)
=
I+O_K(t^{-1/2+b_K}),
\qquad
\lambda\in\partial U_a,
\quad t\to\infty,
\]
where
$
b_K=\sup_{\xi\in K}
|\operatorname{Im}\nu(a(\xi))|<\frac14.
$
\end{itemize}
\end{RHP}

Write $\phi_\xi=\theta_{12}$, where
$
\phi_\xi(\lambda)=\frac{\sqrt3}{2}
\left((1+\xi)\lambda+\frac{1-\xi}{\lambda}\right).
$
Introduce the exact local coordinate
\begin{equation}\label{positiveconformalmap}
w=(\lambda-a)\sqrt{\frac a\lambda},
\qquad z=\sqrt{tF}\,w,
\end{equation}
where
\begin{equation}\label{positiveconformalderivative}
F=\frac{\sqrt3(1+\xi)}a
=\frac{2\sqrt3}{a(1+a^2)}>0,
\qquad
\kappa_0=\sqrt F,\qquad w'(a)=1.
\end{equation}
The square root is analytic near $a$
and equals one at $a$. Since
$1-\xi=(1+\xi)a^2$, we have the exact identity
\begin{equation}\label{positivephasequadratic}
\phi_\xi(\lambda)-\phi_\xi(a)
=\frac F2w^2,
\qquad
\vartheta_{12}(\lambda)
=it\phi_\xi(a)+\frac i2z^2.
\end{equation}

By \eqref{sec4:endpointfactor} and
$T=T_dT_c$, the continuous factor satisfies
\[
(T_c)_{21}(\lambda)
=Q_{\mathrm{reg}}(a)w^{-2i\nu}
\left(1+O(|w|^{1/2})\right),
\]
where
\[
Q_d(a)=(T_d)_{21}(a),\qquad
Q(a)=Q_d(a)Q_{\mathrm{reg}}(a)=\tau_a.
\]
Here $T_d$ is analytic and invertible near
$a$, and $\tau_a$ is the endpoint constant
defined in \eqref{sec4:endpointfactor}.
In general, $|Q(a)|$ need not equal one.

Define the two independent Weber parameters
\[
r_c=-\alpha(a)Q(a)e^{it\phi_\xi(a)}
(tF)^{i\nu},
k_c=\beta(a)Q(a)^{-1}e^{-it\phi_\xi(a)}
(tF)^{-i\nu}.
\]
Then
\[
r_ck_c=-\alpha(a)\beta(a),
\qquad
\Delta_a=1+r_ck_c
=\Delta(a)=e^{-2\pi\nu}.
\]
No conjugation relation between $r_c$
and $k_c$ is imposed.

Finally, let $\Pi$ interchange the first
two coordinates and fix the third.
In the canonical coordinates, define
\[
X_j=e^{(2j-1)\pi i/4}\mathbb R_+,
\qquad j=1,2,3,4,
\]
with all rays oriented outwards and
principal branches for the powers of $z$.

\begin{RHP}\label{RHP:sec6:PC}
Find a $3\times3$ matrix-valued function
$\mathsf P(z)=\operatorname{diag}(P(z),1)$,
where $P(z)$ is a $2\times2$ matrix, satisfying
the following properties.
\begin{itemize}

\item
$\mathsf P(z)$ is analytic in
$\mathbb C\setminus X$, where
$
X=\bigcup_{j=1}^{4}X_j,
X_j=e^{(2j-1)\pi i/4}\mathbb R_+,
$
and each ray is oriented outwards.
It has no poles and satisfies
$
\mathsf P^{\pm1}(z)
=O\bigl(|z|^{-|\operatorname{Im}\nu|}\bigr),
 z\to0,
$
in every adjacent sector.

\item
The boundary values satisfy the jump condition
\[
\mathsf P_+(z)=\mathsf P_-(z)V^{pc}(z),
\qquad z\in X,
\]
where
\begin{equation}\label{sec6:PCjumpexplicit}
V^{pc}(z)=
\begin{cases}
\begin{pmatrix}
1&0&0\\
r_c z^{-2i\nu}e^{iz^2/2}&1&0\\
0&0&1
\end{pmatrix},
& z\in e^{\pi i/4}\mathbb R_+,
\\[4mm]
\begin{pmatrix}
1&-k_c\Delta_a^{-1}
z^{2i\nu}e^{-iz^2/2}&0\\
0&1&0\\
0&0&1
\end{pmatrix},
& z\in e^{3\pi i/4}\mathbb R_+,
\\[4mm]
\begin{pmatrix}
1&0&0\\
-r_c\Delta_a^{-1}
z^{-2i\nu}e^{iz^2/2}&1&0\\
0&0&1
\end{pmatrix},
& z\in e^{5\pi i/4}\mathbb R_+,
\\[4mm]
\begin{pmatrix}
1&k_c z^{2i\nu}e^{-iz^2/2}&0\\
0&1&0\\
0&0&1
\end{pmatrix},
& z\in e^{7\pi i/4}\mathbb R_+.
\end{cases}
\end{equation}
The complex powers are defined using the
principal branch $\arg z\in(-\pi,\pi)$,
with the prescribed lateral values on
the negative real axis. The parameters satisfy
\[
r_ck_c=-\alpha(a)\beta(a),
\qquad
\Delta_a=1+r_ck_c=e^{-2\pi\nu}.
\]

\item
As $z\to\infty$, the $2\times2$ block
$P(z)$ satisfies
\begin{equation}\label{pcexpansion}
P(z)=I_2+\frac{C}{z}+O(z^{-2}),
\qquad
C=
\begin{pmatrix}
0&-ib_{12}\\
ib_{21}&0
\end{pmatrix}.
\end{equation}
Equivalently,
\[
\mathsf P(z)
=
I_3+\frac{1}{z}
\begin{pmatrix}
0&-ib_{12}&0\\
ib_{21}&0&0\\
0&0&0
\end{pmatrix}
+O(z^{-2}).
\]

\end{itemize}
\end{RHP}

The two coefficients are
\[
 b_{12}=\frac{\sqrt{2\pi}e^{-\pi\nu/2}e^{i\pi/4}}
                      {r_c\Gamma(-i\nu)},\qquad
 b_{21}=-\frac{\sqrt{2\pi}e^{-\pi\nu/2}e^{-i\pi/4}}
                      {k_c\Gamma(i\nu)}.
\]
Their apparent singularities on either parameter axis are
removable, as proved next.

\begin{lemma}\label{lem:sec6:PCcomparison}
RH problem~\ref{RHP:sec6:PC} has a unique solution. Its
determinant is one and $b_{12}b_{21}=\nu$.
The sectorwise expansion \eqref{pcexpansion}, and its
derivative, are uniform on compact sets of the parameters
after removal of the explicit time diagonal conjugator.
Continuation to the physical lips gives
\begin{equation}\label{sec6:PCcomparison}
 M^{pc}_{1,0}(\lambda)=\Pi L(\lambda)\mathsf P(\sqrt{tF}\,w)
                         L(\lambda)^{-1}\Pi,\qquad
 L(\lambda)=\operatorname{diag}
       ((\lambda/a)^{i\nu/2},(\lambda/a)^{-i\nu/2},1).
\end{equation}
The powers in $L$ are analytic near $a$ and $L(a)=I$.
This factor solves RH problem~\ref{RHP:sec6:positive}; its
inverse has the same allowed endpoint powers.
\end{lemma}

\begin{proposition}\label{prop:sec6:localexpansions}
The positive and negative center coefficients are
\begin{equation}\label{sec6:localCoefficientMatrices}
 \mathcal A_1=i b_{21}E_{12}-i b_{12}E_{21},\qquad
 \mathcal A_2=\mathcal A_1^T.
\end{equation}
In the original basis, the former entry notation is therefore
\begin{equation}\label{pcbetas}
 \beta_{12}=i b_{21},\qquad \beta_{21}=-i b_{12},\qquad
 \beta_{12}\beta_{21}=\nu.
\end{equation}
No conjugacy or modulus identity for these entries is assumed.
On fixed matching collars the first correction is
$O(t^{-1/2+b_K})$ and the large-variable remainder is
$O(t^{-1+b_K})$, in the piecewise $W^{1,\infty}$ norm.
\end{proposition}
\begin{proof}
Conjugation of $C$ by $\Pi$ gives \eqref{sec6:localCoefficientMatrices}.
The time dependence is exact. With the stripped pair
$r_0=r_ct^{-i\nu}$, $k_0=k_ct^{i\nu}$,
\[
 P(z;r_c,k_c)=d_tP(z;r_0,k_0)d_t^{-1},\qquad
 d_t=\operatorname{diag}(t^{-i\nu/2},t^{i\nu/2}),\qquad
 \|d_t\|\|d_t^{-1}\|=t^{|\Im\nu|}.
\]
The stripped pair ranges in a compact parameter set on $K$;
its real oscillatory phase has modulus one. The Weber remainder
and derivative on a fixed collar hence have orders
$t^{-1+b_K}$. The analytic $L^{\pm1}$ and their derivatives
are uniformly bounded there. At $a$, $w'(a)=1$ and $L(a)=I$,
so the first meromorphic residue is precisely
$\mathcal A_1/\sqrt F$. The negative coefficient follows
from the inverse-transpose identity below.
\end{proof}

Generate the remaining five models without Schwarz reflection:
\begin{equation}\label{negativeparametrixcofactor}
 M^{lo}_{2,0}(\lambda)=M^{lo}_{1,0}(-\lambda)^{-T},\qquad
 w_-(\lambda)=-w(-\lambda),\qquad w_-'(-a)=1,
\end{equation}
\begin{equation}\label{rotatedparametrices}
 M^{lo}_{i,\ell}(\lambda)
 =\mathcal A^\ell M^{lo}_{i,0}(\omega^{-\ell}\lambda)
                                      \mathcal A^{-\ell},
 \qquad i=1,2,\quad\ell=0,1,2.
\end{equation}
Since inversion changes the first correction's sign and
$w(-\lambda)=-w_-(\lambda)$, the residue at $-a$ has the
positive transpose sign in \eqref{sec6:localCoefficientMatrices}.

\begin{proposition}\label{prop:sec6:sixLocalExpansions}
Uniformly on the matching circles
$\partial U$, the local model satisfies
\begin{equation}\label{sec6:combinedLocalNormalization}
M^{lo}(\lambda)
=I+\frac{1}{\sqrt t}
\sum_{i=1}^{2}\mathcal F_i(\lambda)
+O(t^{-1+2b_K}),
\end{equation}
where
\begin{equation}\label{sec6:combinedLocalF}
\mathcal F_i(\lambda)
=\frac{1}{\sqrt F}
\sum_{\ell=0}^{2}
\frac{
\omega^\ell
\mathcal A^\ell\mathcal A_i
\mathcal A^{-\ell}
}{
\lambda-\omega^\ell\lambda_i
},
\qquad i=1,2.
\end{equation}
Here $F$ is defined in
\eqref{positiveconformalderivative},
and
$
\mathcal A_1
=ib_{21}E_{12}-ib_{12}E_{21},
\mathcal A_2=\mathcal A_1^T.
$
\end{proposition}

\begin{proof}
Set $\mathcal F=\mathcal F_1+\mathcal F_2$ and define
\[
\widehat w_1(\lambda)=w(\lambda),\qquad
\widehat w_2(\lambda)=-w(-\lambda).
\]
The analytic numerator matrices of the raw Weber expansions are
\[
B_{1,0}(\lambda)
 =\bigl(\Pi L(\lambda)\Pi\bigr)\mathcal A_1
   \bigl(\Pi L(\lambda)^{-1}\Pi\bigr),
B_{2,0}(\lambda)=B_{1,0}(-\lambda)^T,
B_{i,\ell}(\lambda)
 =\mathcal A^\ell B_{i,0}(\omega^{-\ell}\lambda)
   \mathcal A^{-\ell}.
\]
For $p=\omega^\ell\lambda_i$, put
\[
S_{i,\ell}(\lambda)
=\frac{B_{i,\ell}(\lambda)}
 {\sqrt F\,\widehat w_i(\omega^{-\ell}\lambda)}.
\]
The raw model satisfies, on the fixed matching collar,
\[
M^{lo}_{i,\ell}(\lambda)
=I+t^{-1/2}S_{i,\ell}(\lambda)
 +O_K(t^{-1+b_K}).
\]
Since $\widehat w_i'(\lambda_i)=1$ and
$B_{i,0}(\lambda_i)=\mathcal A_i$, we have
\[
\operatorname*{Res}_{\lambda=p}S_{i,\ell}(\lambda)
=\frac{\omega^\ell}{\sqrt F}
 \mathcal A^\ell\mathcal A_i\mathcal A^{-\ell}
=\operatorname*{Res}_{\lambda=p}\mathcal F(\lambda).
\]
Consequently,
\[
H_{i,\ell}(\lambda)
=\mathcal F(\lambda)-S_{i,\ell}(\lambda)
\]
extends holomorphically through $p$ and is analytic
in a neighborhood of $\overline{U_p}$. The fixed geometry
and the bounds on $\mathcal A_i$ give
\[
\|H_{i,\ell}\|_{W^{1,\infty}(U_p)}
\le C_Kt^{b_K}.
\]
Define
\[
\mathcal Q_{i,\ell}(\lambda)
=\exp\bigl(t^{-1/2}H_{i,\ell}(\lambda)\bigr),
\qquad
P_p(\lambda)
:=\mathcal Q_{i,\ell}(\lambda)
  M^{lo}_{i,\ell}(\lambda),
\]
and use $M^{lo}|_{U_p}=P_p$ throughout the subsequent
factorization. The matrix $H_{i,\ell}$ has trace zero,
so $\det\mathcal Q_{i,\ell}=1$. Moreover,
\[
\mathcal Q_{i,\ell}
=I+t^{-1/2}H_{i,\ell}+O_K(t^{-1+2b_K}),
\qquad
\|\mathcal Q_{i,\ell}^{\pm1}\|_{W^{1,\infty}(U_p)}
\le C_K.
\]
Analytic left multiplication preserves the local jumps
and the prescribed vertex behavior. Multiplying the two
expansions and using $b_K\ge0$ yields
\[
P_p(\lambda)
=I+t^{-1/2}\mathcal F(\lambda)
 +O_K(t^{-1+2b_K})
\]
uniformly on the matching circles. This proves
\eqref{sec6:combinedLocalNormalization}.
\end{proof}

The dressed residue used in reconstruction is
\begin{equation}\label{dressedradiationcoefficient}
 \mathcal K_a=\frac1{\sqrt F}M^\Lambda(a)\mathcal A_1
                                             M^\Lambda(a)^{-1}.
\end{equation}
All six are specified by
\begin{equation}\label{radiationcoefficientreductions}
 \mathcal K_{\omega p}=\omega\mathcal A\mathcal K_p\mathcal A^{-1},
 \qquad \mathcal K_{-p}=\mathcal K_p^T.
\end{equation}
The identical definitions with $M^r$ in place of $M^\Lambda$
differ by $O(e^{-\delta_0t}t^{b_K})$, by Proposition~\ref{prop:sec5:discretecomparison}
on the separated disks.

\subsection{The small-norm RH problem for $E(\lambda)$}
\label{s:dbar-problem}

For $p=\omega^\ell\lambda_i$, use the corrected local factor
$P_p=M^{lo}|_{U_p}=\mathcal Q_{i,\ell}M^{lo}_{i,\ell}$
from Proposition~\ref{prop:sec6:sixLocalExpansions}. Define
\[
M^{\mathrm{app}}(\lambda)=
\begin{cases}
M^r(\lambda)P_p(\lambda),&\lambda\in U_p,\\
M^r(\lambda),&\lambda\notin U.
\end{cases}
\]
The error associated with \eqref{sec4:purefactorization} is
\[
E=M^R(M^{\mathrm{app}})^{-1}.
\]
Its jump contour is
\[
\Sigma_E=(\Sigma^{(ju)}\setminus U)\cup\partial U,
\]
as shown in Figure~\ref{fig:sec6:errorcontour}.
Each circle is oriented clockwise, with its exterior
on the plus side.

\begin{RHP}\label{RHP:sec6:error}
Find a $3\times3$ matrix-valued function
$E(\lambda):=E(\lambda;x,t)$ such that
\begin{itemize}

\item
$E$ is analytic in $\mathbb C\setminus\Sigma_E$, with $E_\pm-I\in L^2(\Sigma_E)$.
It has bounded sectorial values at every
finite contour junction and is analytic at
$\lambda=0$.

\item $
E(\lambda)=I+O(\lambda^{-1}),
\qquad \lambda\to\infty.
$

\item
The boundary values satisfy the jump condition
\[
E_+(\lambda)=E_-(\lambda)V_E(\lambda),
\qquad \lambda\in\Sigma_E,
\]
where
\begin{equation}\label{pureRHerrorcirclejump}
V_E(\lambda)=
\begin{cases}
M^r(\lambda)V^{(2)}(\lambda)
\bigl(M^r(\lambda)\bigr)^{-1},
&\lambda\in\Sigma^{(ju)}\setminus U,
\\[3mm]
M^r(\lambda)P_p(\lambda)
\bigl(M^r(\lambda)\bigr)^{-1},
&\lambda\in\partial U_p,
\quad p=\omega^\ell\lambda_i.
\end{cases}
\end{equation}
Here $i=1,2$, $\ell=0,1,2$, and
$U_{\omega^\ell\lambda_i}$ denotes the
corresponding stationary-point disk.

\end{itemize}
\end{RHP}

The retained poles and pole-circle jumps
cancel in the error problem, since
$M^R$ and $M^{\mathrm{app}}$ have the
same local pole factors and circle jumps.
The local cross jumps cancel by the
construction of $M^{lo}$.

On each clockwise matching circle,
with $P_p=M^{lo}|_{U_p}$, we have
\[
M^R P_p^{-1}(M^r)^{-1}
\bigl(M^rP_p(M^r)^{-1}\bigr)
=M^R(M^r)^{-1},
\]
which verifies the jump relation
\eqref{pureRHerrorcirclejump}.

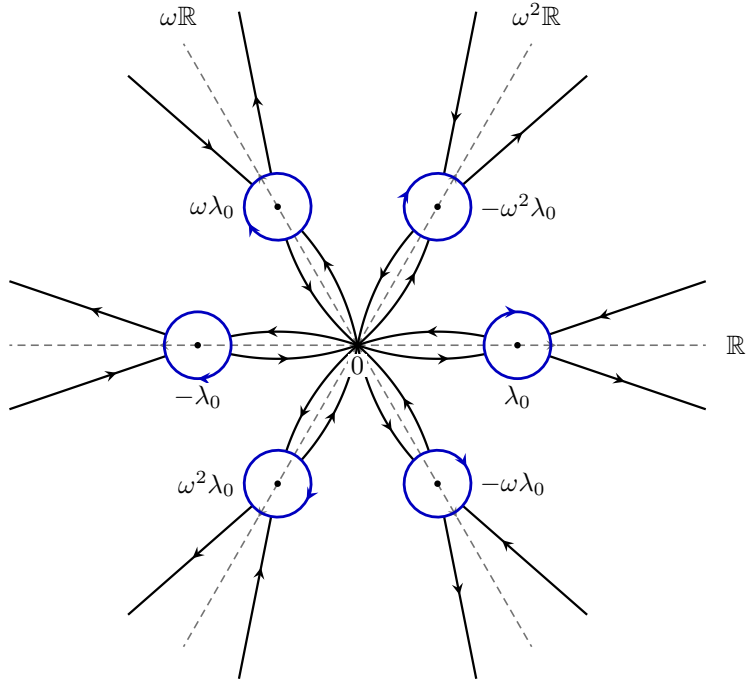
\begin{figure}[htbp]
\centering
\begin{tikzpicture}[>=stealth,scale=0.94,
 secsixarrow/.style={postaction={decorate},decoration={markings,
 mark=at position 0.57 with {\arrow{stealth}}}},
 secsixcircle/.style={blue!75!black,line width=1.05pt,
 postaction={decorate},decoration={markings,
 mark=at position 0.75 with {\arrow{stealth}}}},
 every node/.style={font=\small}]
\def\aa{2.25}\def\qq{0.34}\def\LL{4.90}\def\rr{0.47}
\foreach \ang in {0,60,120,180,240,300}{
\begin{scope}[rotate=\ang]
\draw[densely dashed,black!55,line width=0.6pt,secsixarrow]
 (0,0)--(\LL,0);
\begin{scope}[even odd rule]
\clip (-0.1,-1.8) rectangle (5.1,1.8)
 (\aa,0) circle (\rr);
\draw[black,line width=0.85pt,secsixarrow,domain=2.25:0,samples=50]
 plot (\x,{\qq*\x*(\aa-\x)/\aa});
\draw[black,line width=0.85pt,secsixarrow,domain=0:2.25,samples=50]
 plot (\x,{-\qq*\x*(\aa-\x)/\aa});
\draw[black,line width=0.85pt,secsixarrow]
 (\LL,{\qq*(\LL-\aa)})--(\aa,0);
\draw[black,line width=0.85pt,secsixarrow]
 (\aa,0)--(\LL,{-\qq*(\LL-\aa)});
\end{scope}
\draw[secsixcircle] ({\aa+\rr},0) arc (0:-360:\rr);
\fill[black] (\aa,0) circle (1.3pt);
\end{scope}}
\fill[black] (0,0) circle (1.3pt);
\node[fill=white,inner sep=1pt] at (0,-0.28) {$0$};
\node[fill=white,inner sep=1pt] at (2.25,-0.73) {$\lambda_0$};
\node[fill=white,inner sep=1pt] at (-2.25,-0.73) {$-\lambda_0$};
\node[anchor=east,fill=white,inner sep=1pt] at (-1.68,1.95)
 {$\omega\lambda_0$};
\node[anchor=west,fill=white,inner sep=1pt] at (1.68,1.95)
 {$-\omega^2\lambda_0$};
\node[anchor=east,fill=white,inner sep=1pt] at (-1.68,-1.95)
 {$\omega^2\lambda_0$};
\node[anchor=west,fill=white,inner sep=1pt] at (1.68,-1.95)
 {$-\omega\lambda_0$};
\node[anchor=west] at (5.05,0) {$\mathbb R$};
\node[anchor=south] at (120:5.04) {$\omega\mathbb R$};
\node[anchor=south] at (60:5.04) {$\omega^2\mathbb R$};
\end{tikzpicture}
\caption{The error contour $\Sigma_E$ in $-1<\xi<1$.
The black solid curves are the surviving lips
$\Sigma^{(ju)}\setminus U$; the blue circles are $\partial U_p$,
oriented clockwise. The dashed spectral lines are shown for reference
and carry no jump of $E$.}
\label{fig:sec6:errorcontour}
\end{figure}

Set $\epsilon_t=t^{-1/2+b_K}$.
By the local model expansions and the
exponential decay of the exterior jumps,
we have
\begin{equation}\label{sec6:errorjumpestimate}
\begin{aligned}
\|V_E-I\|_{W^{1,\infty}(\partial U)}
&\le C_K\epsilon_t,\\
\|V_E-I\|_{L^1\cap L^2\cap L^\infty
(\Sigma^{(ju)}\setminus U)}
&\le C_Ke^{-c_Kt}.
\end{aligned}
\end{equation}
Here the derivatives on $\partial U$ are
understood piecewise, with one-sided
traces at contour intersections.

Let $w_E=V_E-I$ and define
$
C_Ef=C_-(fw_E).
$
The corresponding Beals--Coifman equation is
$
(I-C_E)(\mu_E-I)=C_EI.
$
Since
\[
\|C_E\|_{L^2\to L^2}
\le C_K\epsilon_t,
\]
the operator $I-C_E$ is invertible
for sufficiently large $t$.
Consequently,
\begin{equation}\label{sec6:errorrepresentation}
E(\lambda)
=I+\frac{1}{2\pi i}
\int_{\Sigma_E}
\frac{\mu_E(s)w_E(s)}{s-\lambda}\,ds,
\qquad
\|\mu_E-I\|_{L^2}
\le C_K\epsilon_t.
\end{equation}

Moreover, the piecewise $W^{1,\infty}$
bounds and the exact compatibility
of the jumps at contour intersections
yield
\begin{equation}\label{sec6:erroruniform}
\|E-I\|_{L^\infty(\mathbb C\setminus\Sigma_E)}
+
\|E^{-1}-I\|_{L^\infty(\mathbb C\setminus\Sigma_E)}
\le C_K\epsilon_t.
\end{equation}

Finally, defining
\[
M^R=EM^{\mathrm{app}},
\]
we obtain the unique solution of
RH problem~\ref{RHP:4.2}.

\begin{proposition}\label{prop:pureRHerror-expansion}
Under the uniform hypotheses above,
\begin{equation}\label{pureRHerroratZero}
 E(0)=I+t^{-1/2}\mathcal H^{(0)}(0)
                +O_K(t^{-1+2b_K}),
\end{equation}
where
\begin{equation}\label{sec6:errorReconstructionCoefficient}
 \mathcal H^{(0)}(0)=-\sum_{p\in\mathcal S(\xi)}\frac{\mathcal K_p}{p}.
\end{equation}
The coefficients $\mathcal K_p$, which may depend on $t$, are
\eqref{dressedradiationcoefficient}--\eqref{radiationcoefficientreductions}
and satisfy $\|\mathcal K_p\|\le C_Kt^{b_K}$.
\end{proposition}

\begin{proof}
Set $w_E=V_E-I$ and $\epsilon_t=t^{-1/2+b_K}$.
The Cauchy representation gives
\[
E(0)=I+\frac{1}{2\pi i}
 \int_{\Sigma_E}\frac{\mu_E(s)w_E(s)}{s}\,ds.
\]
The endpoint cutoffs imply
\[
\operatorname{dist}
\bigl(0,\operatorname{supp}w_E\bigr)
\ge d_K>0.
\]
Therefore,
\[
E(0)=I+\frac{1}{2\pi i}
 \int_{\Sigma_E}\frac{w_E(s)}s\,ds
 +O_K(\epsilon_t^2),
\]
because
\[
\|(\mu_E-I)w_E\|_1
\le\|\mu_E-I\|_2\|w_E\|_2
=O_K(\epsilon_t^2).
\]

On $\partial U_p$, the corrected matching expansion yields
\[
w_E(s)
 =M^r(s)(P_p(s)-I)(M^r(s))^{-1}
=t^{-1/2}G_p(s)+O_{L^1}(t^{-1+2b_K}),
\]
where
\[
G_p(s)
:=M^r(s)\mathcal F(s)(M^r(s))^{-1},
\qquad
\mathcal F=\mathcal F_1+\mathcal F_2.
\]
The functions $M^r$ and $(M^r)^{-1}$ are analytic in $U_p$.
Within this disk, $\mathcal F$ has just one simple pole,
at $p$, with residue
\[
R_p=\frac{\omega^\ell}{\sqrt F}
 \mathcal A^\ell\mathcal A_i\mathcal A^{-\ell},
\qquad p=\omega^\ell\lambda_i.
\]
Hence
\[
\operatorname*{Res}_{s=p}G_p(s)
=M^r(p)R_p(M^r(p))^{-1}
=\mathcal K_p+O_K(e^{-\delta_0t}t^{b_K}),
\]
by Proposition~\ref{prop:sec5:discretecomparison}.

Since the circle is clockwise and $0\notin\overline{U_p}$,
\[
\frac{1}{2\pi i}
\int_{\partial U_p}\frac{G_p(s)}s\,ds
=-\frac1p\operatorname*{Res}_{s=p}G_p(s).
\]
All other terms are analytic in $U_p$ and have zero
Cauchy residue. The exterior lip contributions are
exponentially small. Summing over the six circles and
absorbing the exponentially small comparison terms gives
\[
E(0)=I-\frac1{\sqrt t}
 \sum_{p\in\mathcal S(\xi)}\frac{\mathcal K_p}{p}
 +O_K(t^{-1+2b_K}),
\]
which proves \eqref{pureRHerroratZero}.
\end{proof}


\section{Contribution from the $\bar\partial$ components}
\label{s:dbar-contribution}

Fix $\xi$ with $|\xi|\neq1$, let $x=\xi t.$ Removing the pure RH factor from $M^{(2)}$ gives
\[
 M^{(3)}(\lambda)=M^{(2)}(\lambda)(M^R(\lambda))^{-1}.
\]
\begin{Dbarproblem}\label{sec7:pureproblem}
Find a $3\times3$ matrix-valued function
$M^{(3)}(x,t,\lambda)$ satisfying the
following properties.
\begin{itemize}

\item
$M^{(3)}$ is bounded and continuous in
$\mathbb C$, has no jumps or poles, and has locally integrable first weak derivatives.
\item
$
M^{(3)}(x,t,\lambda)
=I+O(\lambda^{-1})$,
as $\lambda\to\infty.
$

\item
In the distributional sense,
\begin{equation}\label{sec4:puredbar}
\bar\partial M^{(3)}(x,t,\lambda)
=
M^{(3)}(x,t,\lambda)W^{(3)}(x,t,\lambda),
\end{equation}
where
\begin{equation}\label{sec7:conjugatedSource}
W^{(3)}(\lambda)
=
M^R(\lambda)W^{(2)}(\lambda)M^R(\lambda)^{-1},
\end{equation}
with $W^{(2)}(\lambda)$ given by \eqref{sec4:W2explicit}.
\end{itemize}
\end{Dbarproblem}

Here $\bar\partial=(\partial_s+i\partial_v)/2$ and $dA=ds\,dv$.
The normalized solution satisfies the Cauchy--Green equation
\begin{equation}\label{puredbarintegralequation}
 M^{(3)}(\lambda)=I+\frac1\pi\iint_{\mathbb C}
        \frac{M^{(3)}(z)W^{(3)}(z)}{\lambda-z}\,dA(z).
\end{equation}
Set
\[
 (\mathcal C_WF)(\lambda)=\frac1\pi\iint_{\mathbb C}
                      \frac{F(z)W^{(3)}(z)}{\lambda-z}\,dA(z).
 \qquad (I-\mathcal C_W)M^{(3)}=I.
\]
In particular, the reconstruction value has the negative sign
\begin{equation}\label{sec7:originValueIntegral}
 M^{(3)}(0)=I-\frac1\pi\iint_{\mathbb C}
                  \frac{M^{(3)}(z)W^{(3)}(z)}z\,dA(z).
\end{equation}

\subsection{In the region $\xi\in(-\infty,-1]\cup[1,+\infty)$}
\label{sec7:exterior}

In the exterior regions there are no
nonzero finite stationary points.
By the phase estimates and the bounds
for the extensions established in
Section~\ref{s:4}, the source $W^{(3)}$
decays exponentially in the corresponding
lenses. Moreover, the retained discrete
spectrum is empty, and $M^R=M^r$
together with its inverse is uniformly
bounded on the $\bar\partial$ support.

\begin{lemma}\label{lem:sec7:outerCG}
For the original extension constructed in
Section~\ref{s:4}, uniformly in the exterior
regions $|\xi|\ge1$, we have
\begin{equation}\label{outerCauchyGreenoperator}
\|\mathcal C_W\|_{L^\infty\to L^\infty}
\le Ct^{-1/2}.
\end{equation}
\end{lemma}

\begin{proof}
After rotation, consider $z=s+iv$ with $s>0$ and
$0<v<q_0s$. The pulled-back real-axis coefficient $\rho$
belongs to $\mathcal F_{\mathrm{flat}}$. Put
$
a_\rho(s)=|\rho'(s)|+\frac{|\rho(s)|}{s}.
$
Then $a_\rho\in L^2(\mathbb R_+)$, and
\[
\int_0^\infty\frac{a_\rho(s)}s\,ds<\infty,
\qquad
\int_0^\infty s a_\rho(s)\,ds<\infty.
\]

The extension bounds and the bounded pure RH factors give,
for $\xi\ge1$,
\[
\|W^{(3)}(s+iv)\|
\le C a_\rho(s)e^{-ctv}.
\]
Writing $\lambda=x_0+iy_0$ and applying
Cauchy--Schwarz in $s$, we obtain
\[
\iint
\frac{\|W^{(3)}(z)\|}{|z-\lambda|}\,dA(z)
\le C\|a_\rho\|_2
\int_0^\infty
e^{-ctv}|v-y_0|^{-1/2}\,dv
\le Ct^{-1/2},
\]
uniformly in $\lambda$. Here
\[
\sup_{Y\in\mathbb R}
\int_0^\infty e^{-cV}|V-Y|^{-1/2}\,dV<\infty.
\]

For $\xi\le-1$, the corresponding envelope is
\[
\|W^{(3)}(s+iv)\|
\le C a_\rho(s)e^{-ctv/|s+iv|^2}.
\]
Set
$
\mu=1/z, \eta=1/\lambda,
\widetilde F(\mu)=F(1/\mu),
\widetilde W(\mu)
=-\bar\mu^{-2}W^{(3)}(1/\mu).
$
The transformed wedge satisfies
\[
0<|\Im\mu|<q_0\Re\mu.
\]
Since $|\Im\mu|=v/|z|^2$, its source has an envelope
\[
\|\widetilde W(\mu)\|
\le C\widetilde a(\Re\mu)e^{-ct|\Im\mu|}.
\]
More explicitly, one may take
\[
\widetilde a(u)
=u^{-2}
\sup_{1/((1+q_0^2)u)\le s\le1/u}a_\rho(s),
\qquad u>0.
\]
Indeed, $s=(\Re\mu)/|\mu|^2$ lies in the indicated
interval. Flatness at zero and rapid decay at infinity
give $\widetilde a\in L^2(\mathbb R_+)$.

For the Cauchy--Green convention in
\eqref{puredbarintegralequation}, direct substitution gives
\[
(\mathcal C_WF)(1/\eta)
=(\mathcal C_{\widetilde W}\widetilde F)(\eta)
-(\mathcal C_{\widetilde W}\widetilde F)(0).
\]
Consequently,
\[
\|\mathcal C_W\|_{\infty\to\infty}
\le2\|\mathcal C_{\widetilde W}\|_{\infty\to\infty}
\le Ct^{-1/2}.
\]
Summing the finitely many lens components proves
the assertion.
\end{proof}

\begin{corollary}\label{cor:sec7:exteriorSolvability}
For each fixed $\xi$ with $|\xi|\ge1$,
Problem~\ref{sec7:pureproblem} admits a
unique bounded continuous solution for
all sufficiently large $t$. Moreover,
\[
\|M^{(3)}-I\|_{L^\infty(\mathbb C)}
=O_\xi(t^{-1/2}).
\]
\end{corollary}

\begin{proposition}
\label{prop:sec7:exteriorContribution}
For each fixed $\xi$ with $|\xi|\ge1$
and all sufficiently large $t$,
\begin{equation}\label{outsidelightconedbar}
M^{(3)}(0)=I+O_\xi(t^{-1}).
\end{equation}
\end{proposition}

\begin{proof}
Use $a_\rho$ from the proof of
Lemma~\ref{lem:sec7:outerCG}.
In a right exterior wedge,
\[
\iint
\frac{\|W^{(3)}(z)\|}{|z|}\,dA(z)
\le C\int_0^\infty\frac{a_\rho(s)}s
\int_0^{q_0s}e^{-ctv}\,dv\,ds
\le Ct^{-1},
 \xi\ge1.
\]

For $\xi\le-1$, use
$|s+iv|^2\le(1+q_0^2)s^2$ to obtain
\[
\iint
\frac{\|W^{(3)}(z)\|}{|z|}\,dA(z)
\le C\int_0^\infty\frac{a_\rho(s)}s
\int_0^{q_0s}
e^{-ctv/((1+q_0^2)s^2)}\,dv\,ds
\le Ct^{-1}.
\]
The remaining wedges follow by the reductions.
Corollary~\ref{cor:sec7:exteriorSolvability} and
\eqref{sec7:originValueIntegral} now yield
\[
M^{(3)}(0)=I+O(t^{-1}).
\]
\end{proof}

\subsection{In the region $\xi\in(-1,1)$}
\label{sec7:interior}
Fix $K\Subset(-1,1)$ and let
\[
b_K=\sup_{\xi\in K}|\Im\nu(a(\xi))|<\frac14.
\]
Near each stationary point $p$, the pure
RH solution has the factorization
$
M^R=EM^rP_p,
$
where $P_p$ is the corresponding local
Weber parametrix. Hence
\[
W^{(3)}
=EM^rP_pW^{(2)}P_p^{-1}
(M^r)^{-1}E^{-1}.
\]

Unlike the real-valued case, the
parameter $\nu$ may be complex.
The local factors must therefore be
combined before taking absolute values.
In the canonical Weber coordinate $z=\sqrt{tF}\,w$, put
\[
r_0=r_ct^{-i\nu},\qquad
k_0=k_ct^{i\nu},\qquad
P_0(z)=P(z;r_0,k_0),\qquad
\mathsf P_0=P_0\oplus1,
\]
where $P$ is defined in
\eqref{app:explicit-Weber-solution}.
The stripped pair ranges over a compact parameter set.

With $\Psi_0,C_0$ denoting
\eqref{app:explicit-Weber-fundamental}--
\eqref{app:explicit-Weber-sectors}
at $(r_0,k_0)$,
the exact identities are
\[
\begin{aligned}
\mathsf P_0
\bigl(z^{-2i\nu}e^{iz^2/2}E_{21}\bigr)
\mathsf P_0^{-1}
&=
\bigl(
\Psi_0C_0E_{21}^{(2)}C_0^{-1}\Psi_0^{-1}
\bigr)\oplus0,
\\
\mathsf P_0
\bigl(z^{2i\nu}e^{-iz^2/2}E_{12}\bigr)
\mathsf P_0^{-1}
&=
\bigl(
\Psi_0C_0E_{12}^{(2)}C_0^{-1}\Psi_0^{-1}
\bigr)\oplus0.
\end{aligned}
\]
They remove the singular complex powers
from the triangular source.

Restoring the time dependence gives
\[
\mathsf P(z;r_c,k_c)
=d_t\mathsf P_0(z)d_t^{-1},
\qquad
d_t=\operatorname{diag}
(t^{-i\nu/2},t^{i\nu/2},1),
\]
and
\[
\|d_t\|\|d_t^{-1}\|
=t^{|\Im\nu|}\le t^{b_K}.
\]
The analytic corrections
$\mathcal Q_{i,\ell}^{\pm1}$ in $P_p$
remain uniformly bounded and do not change these powers.

Combining these identities with the
extension bounds of Section~\ref{s:4}
and the uniform RH estimates of Section~6
gives, in the local Morse coordinates,
\begin{equation}\label{sec7:innerSourceEnvelope}
\|W^{(3)}\|
\le C_Kt^{b_K}\rho^{-1/2}
\left\{1+(\sqrt t\,\rho)^{2b_K}\right\}
e^{-c_Ktuv},
\end{equation}
where $\rho=|w|$, $u=|\Re w|$,
$v=|\Im w|$, and the source is contained
in wedges $0<v<q_*u$, with $q_*<1$.

\begin{lemma}
\label{lem:sec7:positiveWedge}
Suppose that $0\le\gamma<1/2$ and
\[
\|B_T\|
\le A_T\rho^{-1/2}
\left\{1+(\sqrt T\,\rho)^\gamma\right\}
e^{-cTuv},
\qquad 0<v<q_*u.
\]
Then
\[
\|\mathcal C_{B_T}\|_{L^\infty\to L^\infty}
\le CA_TT^{-1/4},
\qquad
\|B_T\|_{L^1}\le CA_TT^{-3/4}.
\]
\end{lemma}

\begin{proof}
Since $\rho\asymp u$ in the wedge,
it suffices to estimate terms of the form
$
T^{\gamma/2}u^{-d}e^{-cTuv},
 d=\frac12-\gamma.
$
Choose
\[
\frac1d<p<\frac2d,
\qquad \frac1p+\frac1q=1.
\]
H\"older's inequality in $u$ gives
\[
\|u^{-d}\|_{L^p(v,\infty)}
\le Cv^{1/p-d},
\]
and
\[
\||u+iv-\lambda|^{-1}\|_{L^q(\mathbb R)}
\le C|v-\Im\lambda|^{-1/p}.
\]
Using $u>v$ and splitting the
exponential, the corresponding
Cauchy integral is bounded by
\[
CT^{\gamma/2}
\int_0^\infty
v^{1/p-d}|v-\Im\lambda|^{-1/p}
e^{-c'Tv^2}\,dv
\le CT^{-1/4}.
\]
The bound is uniform in $\lambda$,
since the exponents of the two possible
singular factors sum to $d<1$.

For the area integral, H\"older's
inequality gives
\[
T^{\gamma/2}
\int_0^\infty\int_v^\infty
u^{-d}e^{-cTuv}\,du\,dv
\le
CT^{\gamma/2-1/q}
\int_0^\infty
v^{2/p-d-1}e^{-c'Tv^2}\,dv
\le CT^{-3/4}.
\]
The restrictions on $p$ ensure
integrability. Applying these estimates
with $\gamma=0$ and with the prescribed
positive $\gamma$ proves the lemma.
\end{proof}

\begin{proposition}\label{prop:dbar-contribution}
For $\xi\in K\Subset(-1,1)$ and all
sufficiently large $t$,
Problem~\ref{sec7:pureproblem}
admits a unique bounded solution.
Moreover,
\begin{equation}\label{sec7:innerOperatorEstimate}
\|\mathcal C_W\|_{L^\infty\to L^\infty}
+\|M^{(3)}-I\|_{L^\infty}
\le C_Kt^{-1/4+b_K},
\end{equation}
and
\begin{equation}\label{sec7:innerZeroEstimate}
M^{(3)}(0)
=I+O_K(t^{-3/4+b_K}).
\end{equation}
\end{proposition}

\begin{proof}
Applying Lemma~\ref{lem:sec7:positiveWedge}
to \eqref{sec7:innerSourceEnvelope}
with
\[
\gamma=2b_K<\frac12,
\qquad A_t=C_Kt^{b_K},
\]
and summing over the six stationary
neighborhoods, we obtain
$
\|\mathcal C_W\|_{L^\infty\to L^\infty}
\le C_Kt^{-1/4+b_K},
$
and
$
\|W^{(3)}\|_{L^1(U)}
\le C_Kt^{-3/4+b_K}.
$
The remaining nonstationary parts
contribute $O_K(t^{-1/2})$ to the
operator estimate and $O_K(t^{-1})$
to the weighted area integral.
Since $b_K<1/4$, the operator norm
tends to zero. The Neumann series
therefore gives the unique bounded
solution and
\[
\|M^{(3)}-I\|_{L^\infty}
\le C_Kt^{-1/4+b_K}.
\]

Finally, the stationary disks remain
uniformly separated from $\lambda=0$.
The Cauchy--Green representation yields
\[
\|M^{(3)}(0)-I\|
\le C_K
\iint_{\mathbb C}
\frac{\|W^{(3)}(z)\|}{|z|}\,dA(z)
\le C_Kt^{-3/4+b_K}.
\]
This proves \eqref{sec7:innerZeroEstimate}.
\end{proof}

\subsection{Uniform estimates near the light cone}

\begin{proposition}[Uniform endpoint estimates]
\label{prop:uniform-endpoint-estimates}
Assume the scattering and inverse-reconstruction hypotheses of
the preceding sections, with
$\alpha,\beta\in\mathcal F_{\mathrm{flat}}$ and
$\|\alpha\beta\|_{L^\infty}\le\rho_*<1$.
Choose $\varepsilon>0$ sufficiently small that the two strips
\[
 \mathcal T_\varepsilon
 =\{\xi:1-\varepsilon<|\xi|<1\}
\]
are separated from all discrete velocities and that their
scaled stationary neighborhoods avoid all pole circles.
Choose $\delta_0>0$ below the resulting uniform discrete phase gap.
Then $\Lambda=\varnothing$ throughout $\mathcal T_\varepsilon$.
Set
\begin{equation}\label{endpoint:scales}
 a=\sqrt{\frac{1-\xi}{1+\xi}},\qquad
 h=\frac{a}{1+a^2},\qquad
 \tau=\sqrt3\,th,\qquad
 b_\varepsilon=\sup_{\xi\in\mathcal T_\varepsilon}
 |\Im\nu(a(\xi))|<\frac14.
\end{equation}
Using equivalent endpoint-adapted lens deformations, with
the unchanged radiation coefficient of
\eqref{scalarDressedRadiation}, the factors at zero satisfy
the exact identity
\[
 M(0)=M^{(3)}(0)E(0)S(0),\qquad
 S(0)=I+O_\varepsilon(e^{-ct}),
\]
where $S$ is the pure deleted-pole circle solution.
For $\tau\ge1$,
\begin{align}
 E(0)
 &=I+\frac{\mathcal H^{(0)}(0)}{\sqrt t}
   +O_\varepsilon\!\left(
       h^2\tau^{-1+2b_\varepsilon}+e^{-ct}\right),
 \label{endpoint:RH}\\
 M^{(3)}(0)
 &=I+O_\varepsilon\!\left(
       h^2\tau^{-3/4+b_\varepsilon}+t^{-1}\right),
 \label{endpoint:dbar}\\
 \left\|\frac{\mathcal H^{(0)}(0)}{\sqrt t}\right\|
 &\le C_\varepsilon h^2\tau^{-1/2+b_\varepsilon}.
 \label{endpoint:leading}
\end{align}
Here $\mathcal H^{(0)}(0)$ is the six-residue coefficient
in \eqref{sec6:errorReconstructionCoefficient}, with
$M^\Lambda=I$.
For $0<\tau\le1$, a separate endpoint deformation gives
\begin{equation}\label{endpoint:smalltau}
 M(0)=I+O_\varepsilon(t^{-1}),\qquad
 \frac{\mathcal U_{\rm rad}^{\varnothing}(x,t;\xi)}
 {\sqrt t}=O_\varepsilon(t^{-1}).
\end{equation}
All constants and the lower bound on $t$ are independent
of $\xi\in\mathcal T_\varepsilon$.
\end{proposition}

\section{long-time asymptotic result for the Tzitz\'eica equation}
\label{s:reconstruction-final}

\begin{theorem}
\label{thm:soliton-resolution}
\label{cor:interior-lightcone-transition}
Let $u(x,t)$ solve the Cauchy problem \eqref{TZI}
with complex-valued initial data
$(u_0,u_1)\in\mathcal X_1^{\mathcal S}$
defined in \eqref{sec2:complex-cauchy-class}
and scattering data
$\{\alpha,\beta;
\{(\zeta_n,\eta_n,c_n,d_n)\}_{n=1}^{N}\}$ satisfying Assumption~\ref{ass:discrete}, $\zeta_{n}=\rho_{n}e^{i\theta_{n}}$ satisfying \eqref{sec5:singleorbitregularity}.
Let
$u_{\mathrm{sol}}^\Lambda
(x,t,\widetilde{\mathcal D}^{\Lambda}(\xi))
=\log_{\mathrm{cont}}
q^\Lambda(x,t,\widetilde{\mathcal D}^{\Lambda}(\xi))$
be the reconstruction of the reflectionless
$N(\Lambda)$-soliton model in
Corollary~\ref{corollary5.2}, where
$N(\Lambda)=|\Lambda|$ and
$\widetilde{\mathcal D}^{\Lambda}(\xi)
=\{(\zeta_n,\eta_n,
\widetilde c_n^\Lambda,\widetilde d_n^\Lambda):
n\in\Lambda\}$,
with the effective norming constants given by
\eqref{modifiednormingconstant}.
Here $\log_{\mathrm{cont}}$ denotes the continuous
logarithm with the normalization specified in
\eqref{sec5:solitonreconstruction}.
Then the solution of the Tzitz\'eica equation \eqref{TZI} as $t \to +\infty$ in the different regions can be described as follows:
\begin{enumerate}[(i)]

\item
For the regions $|\xi|\ge1$ ($\xi:=x/t$),
\begin{equation}\label{finalOutsideAsymptotics}
u(x,t)=
\begin{cases}
O(t^{-1}),&\xi\ge1,\\
2\pi i+O(t^{-1}),&\xi\le-1.
\end{cases}
\end{equation}
\equationlabels{\label{finalOutsideDecay}}

\item
For the region $-1<\xi<1$,

\begin{equation}\label{finalInsideAsymptotics}
u(x,t)
={}u_{\mathrm{sol}}^\Lambda(x,t;\xi)
+2\pi i k_\Lambda
+\operatorname{Log}\!\left(
1+\frac{\mathcal U_{\mathrm{rad}}^\Lambda(x,t;\xi)}
{\sqrt t}
\right)
+O_K(t^{-3/4+b_K}),
\end{equation}
where $\mathcal U_{\mathrm{rad}}^\Lambda$ is the
generally complex coefficient defined in
\eqref{scalarDressedRadiation},
and $u_{\mathrm{sol}}^\Lambda(x,t;\xi)$ is given by
\eqref{sec5:solitonreconstruction}.
Here
$
b_K=\sup_{\xi\in K}|\Im\nu(a(\xi))|<\frac14,
a(\xi)=\sqrt{\frac{1-\xi}{1+\xi}}.
$
The logarithm $\operatorname{Log}$ denotes the analytic
branch near one normalized by $\operatorname{Log}1=0$.
The integer $k_\Lambda$ is determined by the physical
and model continuous logarithms as in
\eqref{fieldReconstructionInside}
and is constant on each connected comparison domain.

\item
For the regions $1-\varepsilon<|\xi|<1$,
\begin{equation}\label{interiorLightconeTransition}
u(x,t)=
\begin{cases}
\displaystyle
\operatorname{Log}\!\left(
1+\frac{\mathcal U_{\mathrm{rad}}^{\varnothing}(x,t;\xi)}
{\sqrt t}
\right)
+O_\varepsilon(t^{-3/4+b_\varepsilon}),
&1-\varepsilon<\xi<1,
\\[4mm]
\displaystyle
2\pi i+
\operatorname{Log}\!\left(
1+\frac{\mathcal U_{\mathrm{rad}}^{\varnothing}(x,t;\xi)}
{\sqrt t}
\right)
+O_\varepsilon(t^{-3/4+b_\varepsilon}),
&-1<\xi<-1+\varepsilon.
\end{cases}
\end{equation}
Here
$
b_\varepsilon
=\sup_{1-\varepsilon<|\xi|<1}
|\Im\nu(a(\xi))|<\frac14,
$
and $\mathcal U_{\mathrm{rad}}^{\varnothing}$
is given by \eqref{scalarDressedRadiation}
with $M^\Lambda=I$.
The logarithm $\operatorname{Log}$ denotes the
analytic branch near one normalized by
$\operatorname{Log}1=0$.

\end{enumerate}
\end{theorem}

We recover the solution from the matrix
asymptotics established in Sections~\ref{s:discrete-contribution}--\ref{s:dbar-contribution}.
Throughout, the scattering data and the
retained reflectionless model satisfy the
hypotheses of the preceding sections.
For a $3\times3$ matrix $X$, write
\begin{equation}\label{reconstructionFunctional}
\mathfrak q[X]
:=\boldsymbol\ell X\mathbf e_3,
\qquad
\boldsymbol\ell=(\omega,\omega^2,1),
\qquad
\mathbf e_3=(0,0,1)^T.
\end{equation}

By Corollary~\ref{corollary5.2},
$M^\Lambda$ is the retained reflectionless
soliton model corresponding to the modified
discrete data $\widetilde{\mathcal D}^{\Lambda}(\xi)$.
We set
\begin{equation}\label{modifiedSolitonField}
q^\Lambda
:=\mathfrak q[M^\Lambda(0)],
\qquad
u^\Lambda:=\log_{\rm cont}q^\Lambda,
\end{equation}
where, for each fixed $\xi$, the model logarithm is normalized
by $u^\Lambda(+\infty,t;\xi)=0$, as in
Corollary~\ref{corollary5.2}. The physical field $u$ is the
right-normalized Cauchy lift from Section~\ref{s:2}.
Statements about its spatial endpoint values are understood
on a regular solution interval on which the charge-one
boundary conditions are preserved.
The cyclic reduction implies that, for
every matrix $X$ commuting with $\mathcal A$,
\[
\boldsymbol\ell X=\mathfrak q[X]\boldsymbol\ell.
\]
Consequently,
\[
\mathfrak q[XY]
=\mathfrak q[X]\mathfrak q[Y]
\]
whenever $X$ and $Y$ commute with
$\mathcal A$. Moreover, $\mathfrak q[X]\ne0$
if $X$ is invertible. These properties
apply to the reconstruction factors
evaluated at $\lambda=0$.

\begin{itemize}

\item
\textbf{The exterior regions:
$\xi\in(-\infty,-1]\cup[1,+\infty)$.}

Undoing the transformations in
Sections~\ref{s:3}--\ref{s:4} gives
\[
M(\lambda)
=M^{(3)}(\lambda)M^R(\lambda)
\bigl(\mathcal R^{(2)}(\lambda)\bigr)^{-1}
T(\lambda)^{-1}\mathcal G(\lambda)^{-1}.
\]
At the reconstruction point,
$T(0)=I$, $\mathcal G=I$ near zero,
and $\mathcal R^{(2)}(0)=I$.
Since $0$ lies outside the stationary
disks, we obtain
\begin{equation}\label{matrixfactorizationatZero}
M(0)
=M^{(3)}(0)E(0)
M^{\rm err}(0)M^\Lambda(0).
\end{equation}
The reconstruction formula
\eqref{TzitzeicaReconstructionFormula}
takes the form
\begin{equation}\label{reconstructionAtZero}
q(x,t)=\mathfrak q[M(0)],
\qquad
u(x,t)=\log_{\rm cont}q(x,t).
\end{equation}

In the exterior regions there are no nonzero finite stationary
points on the spectral rays, and $M^R=M^r$, $E=I$.
The uniform exterior estimates in Section~\ref{sec7:exterior}
and \eqref{discreteerrorzero} give
\[
M^{(3)}(0)=I+O(t^{-1}),\qquad
M^{\rm err}(0)=I+O(e^{-ct}),
\]
uniformly for $|\xi|\ge1$. Therefore
\begin{equation}\label{sec8:outsidematrix}
M(0)=\bigl(I+O(t^{-1})\bigr)M^\Lambda(0).
\end{equation}
All discrete velocities lie strictly inside $(-1,1)$, so
$\Lambda=\varnothing$ and $M^\Lambda=I$ for $|\xi|\ge1$.
Consequently,
\[
\sup_{|x|\ge t}|q(x,t)-1|\le Ct^{-1}.
\]
For sufficiently large $t$, $q$ takes values in $D(1,1/2)$
on each of the connected spatial half-lines $[t,\infty)$
and $(-\infty,-t]$. Let $\operatorname{Log}$ be the analytic
logarithm on this disk satisfying $\operatorname{Log}1=0$.
On each half-line, $u-\operatorname{Log}q$ is a continuous
$2\pi i\mathbb Z$-valued function and is therefore constant.
For the normalized charge-one Cauchy solution,
$u(+\infty,t)=0$ and $u(-\infty,t)=2\pi i$.
These endpoint values determine the two constants and give
\begin{equation}\label{sec8:exteriorLifts}
u(\xi t,t)=
\begin{cases}
O(t^{-1}),&\xi\ge1,\\
2\pi i+O(t^{-1}),&\xi\le-1,
\end{cases}
\end{equation}
with uniform remainders on the respective exterior regions.

\item
\textbf{The interior region:
$\xi\in K\Subset(-1,1)$.}

We assume that the retained discrete
spectrum is unchanged on $K$ and
\[
b_K:=\sup_{\xi\in K}|\Im\nu(a(\xi))|<\frac14.
\]
Outside the stationary disks, the
inverse transformations give
\[
M(\lambda)
=M^{(3)}(\lambda)E(\lambda)M^r(\lambda)
\bigl(\mathcal R^{(2)}(\lambda)\bigr)^{-1}
T(\lambda)^{-1}\mathcal G(\lambda)^{-1}.
\]
At $\lambda=0$,
Proposition~\ref{prop:pureRHerror-expansion},
Proposition~\ref{prop:dbar-contribution}
and the discrete comparison yield
\begin{equation}\label{combinedmatrixatZero}
M(0)=
\left[
I-\frac1{\sqrt t}
\sum_{p\in\mathcal S(\xi)}
\frac{\mathcal K_p}{p}
+O_K(t^{-3/4+b_K})
\right]M^\Lambda(0).
\end{equation}
Indeed, the pure RH remainder
$O_K(t^{-1+2b_K})$ and the exponentially
small discrete error are absorbed into
the $\bar\partial$ remainder, since
$b_K<1/4$.

Applying $\mathfrak q$ to
\eqref{combinedmatrixatZero} gives
\[
\frac{q(x,t)}{q^\Lambda}
=
1+\frac{\mathcal U_{\rm rad}^\Lambda}{\sqrt t}
+\rho_K,
\qquad
\rho_K=O_K(t^{-3/4+b_K}),
\]
where
\begin{equation}\label{scalarDressedRadiation}
\mathcal U_{\rm rad}^\Lambda
:=
-\frac{
\boldsymbol\ell\mathfrak K
M^\Lambda(0)\mathbf e_3
}{q^\Lambda},
\end{equation}
with $\mathfrak K
:=\sum_{p\in\mathcal S(\xi)}
\frac{\mathcal K_p}{p}.$
The six matrices $\mathcal K_p$ are
specified by
\eqref{dressedradiationcoefficient}--
\eqref{radiationcoefficientreductions}.
By the cyclic symmetry,
$\mathfrak K$ commutes with $\mathcal A$.
Therefore
$
\mathcal U_{\rm rad}^\Lambda
=-\boldsymbol\ell\mathfrak K\mathbf e_3
=O_K(t^{b_K}).
$
In particular,
$t^{-1/2}\mathcal U_{\rm rad}^\Lambda
\to0$ uniformly on $K$.
Unlike the real-valued case, the
quantities
$t^{-1/2}\mathcal U_{\rm rad}^\Lambda$
and $\rho_K$ are generally complex.
For sufficiently large $t$, the
analytic logarithm near one satisfies
\[
\operatorname{Log}\left(
1+\frac{\mathcal U_{\rm rad}^\Lambda}{\sqrt t}
+\rho_K
\right)
=
\operatorname{Log}\left(
1+\frac{\mathcal U_{\rm rad}^\Lambda}{\sqrt t}
\right)
+O_K(t^{-3/4+b_K}).
\]
Let $\mathcal D$ be a connected comparison domain at sufficiently
large times, with $\xi=x/t\in K$, on which the retained model
is regular and depends continuously on $(x,t)$.
The relative estimate above ensures that
$q/q^\Lambda\in D(1,1/2)$ on $\mathcal D$.
Fix $(x_*,t_*)\in\mathcal D$ and put $\xi_*=x_*/t_*$.
Define the integer associated with the prescribed physical
and model normalizations by
\[
k_\Lambda:=\frac{1}{2\pi i}
\left[
u(x_*,t_*)-u^\Lambda(x_*,t_*;\xi_*)
-\operatorname{Log}\frac{q(x_*,t_*)}{q^\Lambda(x_*,t_*;\xi_*)}
\right]\in\mathbb Z.
\]
Indeed, $u-u^\Lambda-\operatorname{Log}(q/q^\Lambda)$
is continuous and takes values in $2\pi i\mathbb Z$ on
$\mathcal D$, so it is constant there.
Consequently,
\begin{equation}\label{fieldReconstructionInside}
u(x,t)
={}u^\Lambda(x,t;\xi)+2\pi i k_\Lambda
+\operatorname{Log}\left(
1+\frac{\mathcal U_{\rm rad}^\Lambda(x,t;\xi)}{\sqrt t}
\right)
+O_K(t^{-3/4+b_K}).
\end{equation}
When the differentiated spatial tails hold, the same integer
can be expressed without choosing any additional logarithmic lift:
\[
\begin{aligned}
2\pi i k_\Lambda
={}&-\int_{x_*}^{\infty}
\left(
\frac{\partial_y q(y,t_*)}{q(y,t_*)}
-\frac{\partial_y q^\Lambda(y,t_*;\xi_*)}
{q^\Lambda(y,t_*;\xi_*)}
\right)\,dy\\
&-\operatorname{Log}
\frac{q(x_*,t_*)}{q^\Lambda(x_*,t_*;\xi_*)}.
\end{aligned}
\]
Here $\xi_*$ and all modified model data are held fixed
throughout the $y$-integration. In particular, $k_\Lambda=0$
requires the matching identity
\[
u(x_*,t_*)-u^\Lambda(x_*,t_*;\xi_*)
=\operatorname{Log}
\frac{q(x_*,t_*)}{q^\Lambda(x_*,t_*;\xi_*)}
\]
at one point of $\mathcal D$. The small relative error alone
does not establish this identity or determine the numerical
value of $k_\Lambda$.

For $\Lambda=\varnothing$, one has
$M^\Lambda=I$, $q^\Lambda=1$, and
\[
\mathcal U_{\rm rad}^{\varnothing}
=-\boldsymbol\ell\mathfrak K\mathbf e_3.
\]
Thus \eqref{fieldReconstructionInside}
reduces to the pure-radiation formula
\[
u(x,t)
=
2\pi i k_0+
\operatorname{Log}\left(
1+\frac{\mathcal U_{\rm rad}^{\varnothing}}
{\sqrt t}
\right)
+O_K(t^{-3/4+b_K}).
\]

\end{itemize}

Proposition~\ref{prop:sec5:discretecomparison}
provides the exponentially small
comparison with the retained discrete
spectrum and the modified norming data.
This completes the reconstruction
in the fixed exterior and interior regions.

\begin{proof}[Proof of Theorem~\ref{thm:soliton-resolution}\textup{(iii)}]
Let $r_\varepsilon=3/4-b_\varepsilon$ and
$z_{\rm rad}=t^{-1/2}\mathcal U_{\rm rad}^{\varnothing}$.
If $\tau\ge1$, Proposition~\ref{prop:uniform-endpoint-estimates}
and $b_\varepsilon<1/4$ give
\[
 M(0)
 =I+\frac{\mathcal H^{(0)}(0)}{\sqrt t}
  +O_\varepsilon\!\left(h^2\tau^{-r_\varepsilon}+t^{-1}\right).
\]
Indeed, $\tau^{-1+2b_\varepsilon}
\le\tau^{-3/4+b_\varepsilon}$, and products of the
corrections satisfy the same bound.
Moreover,
$
 h^2\tau^{-r_\varepsilon}
 =3^{-r_\varepsilon/2}t^{-r_\varepsilon}
   h^{\,2-r_\varepsilon}
 \le C t^{-r_\varepsilon}.
$
Applying the reconstruction functional
\eqref{reconstructionFunctional}, and then using
\eqref{endpoint:smalltau} when $0<\tau\le1$, yields
\[
 q(x,t)=1+z_{\rm rad}
        +O_\varepsilon(t^{-3/4+b_\varepsilon})
\]
uniformly throughout $\mathcal T_\varepsilon$.
Also $z_{\rm rad}=O_\varepsilon(t^{-1/2+b_\varepsilon})$.
For all sufficiently large $t$, both $q$ and
$1+z_{\rm rad}$ therefore lie in $D(1,1/2)$, where
\[
 |\operatorname{Log}q-\operatorname{Log}(1+z_{\rm rad})|
 \le2|q-1-z_{\rm rad}|.
\]
Here $\operatorname{Log}$ is the principal logarithm on that disk.
The uniform exterior estimates extend to $\xi=\pm1$
by \eqref{sec4:outerphase}. Together with the charge-one
spatial boundary conditions, they identify the physical
lift as $\operatorname{Log}q$ on the right strip and
$2\pi i+\operatorname{Log}q$ on the left strip.
Continuity across $x=\pm t$ and the uniform nonvanishing
disk prevent a change of this integer.
This proves \eqref{interiorLightconeTransition}
with the logarithmic convention stated there.
\end{proof}

\section*{Acknowledgments}\label{s:8} 	

    The authors thank Professor Dengshan Wang from Beijing Normal University, Professor Lin Huang from Hangzhou Dianzi University, and Dr. Xiaodong Zhu from Beijing Normal University and SISSA for their insightful remarks on the first version of this work. This work was supported by the National Natural Science Foundation of China under Grant No. 12371255, the Fundamental Research Funds for the Central Universities of CUMT under Grant No. 2024ZDPYJQ1003, and the Postgraduate Research \& Practice Program of Education \& Teaching Reform of CUMT under Grant No. 2025YJSJG031.\\

  \textbf{Data availibility}: The data which supports the findings of this study is available within the article.\\

\textbf{Conflict of interest}: The authors declare no conflict of interest.

\appendix
\section{The complex parabolic-cylinder model}
\label{app:model-rhp}

We record the parabolic-cylinder model used in
Sections~\ref{s:jump-contribution}--\ref{s:reconstruction-final}.
Unlike the real-valued case, the parameters
$r,k\in\mathbb C$ are independent. Let
\[
\mathcal U=\{(r,k)\in\mathbb C^2:|rk|<1\},
\qquad
\Delta_X=1+rk,
\qquad
\nu=-\frac1{2\pi}\operatorname{Log}\Delta_X,
\]
where $\operatorname{Log}$ denotes the principal
logarithm. In particular,
$\Re\Delta_X>0$ and
$b:=|\Im\nu|<1/4$.

Let
\[
X_j=e^{(2j-1)\pi i/4}\mathbb R_+,
\qquad j=1,2,3,4,
\]
and $X=\bigcup_{j=1}^4X_j$, with all rays
oriented outwards. The powers $z^{\pm2i\nu}$
are defined with $-\pi<\arg z<\pi$.
Put
\[
\mathcal D_X(z)=\operatorname{diag}
\left(
z^{i\nu}e^{-iz^2/4},
z^{-i\nu}e^{iz^2/4},1
\right).
\]

\begin{RHP}\label{RHP:app:model-positive}
Find a $3\times3$ matrix-valued function
$M^X(z;r,k)$ satisfying the following
properties.
\begin{itemize}

\item
$M^X$ is analytic in $\mathbb C\setminus X$,
has no poles, and has continuous boundary
values on $X\setminus\{0\}$, where
\[
X=\bigcup_{j=1}^{4}X_j.
\]

\item
The boundary values satisfy the jump condition
\[
M^X_+(z)=M^X_-(z)V^X(z;r,k),
\qquad z\in X,
\]
where
\begin{equation}\label{app:model-jump}
V^X(z;r,k)=
\begin{cases}
\begin{pmatrix}
1&r z^{-2i\nu}e^{iz^2/2}&0\\
0&1&0\\
0&0&1
\end{pmatrix},
&z\in X_1,
\\[4mm]
\begin{pmatrix}
1&0&0\\
-k\Delta_X^{-1}z^{2i\nu}e^{-iz^2/2}
&1&0\\
0&0&1
\end{pmatrix},
&z\in X_2,
\\[4mm]
\begin{pmatrix}
1&-r\Delta_X^{-1}
z^{-2i\nu}e^{iz^2/2}&0\\
0&1&0\\
0&0&1
\end{pmatrix},
&z\in X_3,
\\[4mm]
\begin{pmatrix}
1&0&0\\
k z^{2i\nu}e^{-iz^2/2}&1&0\\
0&0&1
\end{pmatrix},
&z\in X_4.
\end{cases}
\end{equation}

Here
\[
\Delta_X=1+rk=e^{-2\pi\nu},
\qquad
\nu=-\frac1{2\pi}\operatorname{Log}\Delta_X.
\]
The complex powers are defined using
the prescribed branches of the logarithm.

\item
The asymptotic behavior at infinity is
\[
M^X(z;r,k)=I+O(z^{-1}),
\qquad z\to\infty.
\]
At the origin, the sectorial prefactors
\[
M^X(z;r,k)\mathcal D_X(z)^{-1}
\]
and their inverses are bounded.
In particular,
\[
\bigl(M^X(z;r,k)\bigr)^{\pm1}
=O(|z|^{-b}),
\qquad z\to0,
\]
in each sector, where
$b=|\Im\nu|$.

\end{itemize}
\end{RHP}

\begin{figure}[htbp]
\centering
\begin{tikzpicture}[>=stealth,scale=0.88,
modelarrow/.style={line width=0.8pt,
postaction={decorate},
decoration={markings,
mark=at position 0.60
with {\arrow{stealth}}}},
every node/.style={font=\small}]
\foreach \ang/\lab in
{45/1,135/2,225/3,315/4}{
\draw[modelarrow] (0,0)--(\ang:2.05);
\node at (\ang:2.32) {$X_{\lab}$};}
\fill (0,0) circle (1.1pt);
\node[below] at (0,-0.06) {$0$};
\end{tikzpicture}
\caption{The outward-oriented cross for
the complex parabolic-cylinder model.}
\label{fig:app:model-cross}
\end{figure}

\begin{lemma}\label{lem:app:model-positive}
For every $(r,k)\in\mathcal U$,
Problem~\ref{RHP:app:model-positive}
has a unique solution with determinant one
and identity third row and column.
Its sectorwise expansion is
\begin{equation}\label{app:model-expansion}
M^X(z;r,k)
=I+\frac{M_1^X(r,k)}z
+\mathcal E_X(z;r,k),
\qquad
M_1^X=
\begin{pmatrix}
0&\beta_{12}&0\\
\beta_{21}&0&0\\
0&0&0
\end{pmatrix},
\end{equation}
where
\begin{equation}\label{app:model-betas}
\beta_{12}
=-\frac{\sqrt{2\pi}e^{-\pi\nu/2}e^{i\pi/4}}
{k\Gamma(i\nu)},
\qquad
\beta_{21}
=\frac{\sqrt{2\pi}e^{-\pi\nu/2}e^{-i\pi/4}}
{r\Gamma(-i\nu)}.
\end{equation}
Both coefficients extend holomorphically
to the parameter axes and satisfy
$\beta_{12}\beta_{21}=\nu$.

For every compact $\mathcal K\Subset\mathcal U$,
\[
\|\partial_z^j\mathcal E_X(z;r,k)\|
\le C_{\mathcal K}|z|^{-2-j},
\qquad |z|\ge1,\quad j=0,1,
\]
uniformly up to the contour traces.
Moreover,
\begin{equation}\label{app:model-bound}
\sup_{(r,k)\in\mathcal K}
\sup_{z\notin X}
\min(1,|z|)^{|\Im\nu|}
\left(
\|M^X(z;r,k)\|
+\|M^X(z;r,k)^{-1}\|
\right)
\le C_{\mathcal K}.
\end{equation}
\end{lemma}

\begin{proof}
Let $\Pi$ interchange the first two
coordinates. In the canonical basis,
write $\mathsf P=P\oplus1$ and
\begin{equation}\label{app:canonicalIdentification}
M^X=\Pi\mathsf P\Pi,
\qquad
P(z)=I+\frac Cz+O(z^{-2}),
\end{equation}
where
\[
C=
\begin{pmatrix}
0&-ib_{12}\\
ib_{21}&0
\end{pmatrix},
\]
and
\[
\begin{aligned}
b_{12}
&=\frac{\sqrt{2\pi}e^{-\pi\nu/2}e^{i\pi/4}}
{r\Gamma(-i\nu)},\\
b_{21}
&=-\frac{\sqrt{2\pi}e^{-\pi\nu/2}e^{-i\pi/4}}
{k\Gamma(i\nu)}.
\end{aligned}
\]

Let $D_\alpha$ denote the parabolic-cylinder function,
$E_{ij}^{(2)}$ the $2\times2$ matrix units, and
$\sigma_3=\operatorname{diag}(1,-1)$.
Define the entire matrix
\begin{equation}\label{app:explicit-Weber-fundamental}
\Psi(z)=e^{\pi\nu/4}
\begin{pmatrix}
D_{i\nu}(e^{i\pi/4}z)
&
-ie^{-i\pi/4}b_{12}
D_{-i\nu-1}(e^{-i\pi/4}z)
\\
ie^{i\pi/4}b_{21}
D_{i\nu-1}(e^{i\pi/4}z)
&
D_{-i\nu}(e^{-i\pi/4}z)
\end{pmatrix}.
\end{equation}

Put
\[
L_2(v)=I_2+vE_{21}^{(2)},\qquad
U_2(v)=I_2+vE_{12}^{(2)},\qquad
D_0=z^{-i\nu\sigma_3}e^{iz^2\sigma_3/4}.
\]
For $\varphi=\arg z\in(-\pi,\pi)$, define
\begin{equation}\label{app:explicit-Weber-sectors}
C(\varphi)=
\begin{cases}
I_2,
&-\pi/4<\varphi<\pi/4,
\\
L_2(r),
&\pi/4<\varphi<3\pi/4,
\\
L_2(r)U_2(-k/\Delta_X),
&3\pi/4<\varphi<\pi,
\\
L_2(r)U_2(-k/\Delta_X)
\operatorname{diag}(\Delta_X^{-1},\Delta_X),
&-\pi<\varphi<-3\pi/4,
\\
U_2(-k),
&-3\pi/4<\varphi<-\pi/4.
\end{cases}
\end{equation}
Then
\begin{equation}\label{app:explicit-Weber-solution}
P(z;r,k)=\Psi(z)C(\arg z)D_0(z),
\qquad
M^X=\Pi(P\oplus1)\Pi.
\end{equation}

The sector matrices on the two sides of the negative
real axis compensate for the principal-power branch jump,
so that axis does not carry an additional jump.
In particular,
\[
L_2(r)U_2(-k/\Delta_X)
\operatorname{diag}(\Delta_X^{-1},\Delta_X)
L_2(-r/\Delta_X)U_2(k)=I_2.
\]
The remaining four jumps follow by taking the ratios of
the adjacent sector matrices and conjugating by $D_0$.

We have
\[
\Psi'(z)=
\begin{pmatrix}
-iz/2&b_{12}\\
b_{21}&iz/2
\end{pmatrix}\Psi(z),
\qquad
\det\Psi=1,
\]
and
\[
P(z)=I_2+\frac1z
\begin{pmatrix}
0&-ib_{12}\\
ib_{21}&0
\end{pmatrix}
+O(z^{-2}).
\]

For each closed sector and its contour traces,
apply the connection formulas first,
so that the special-function arguments lie in closed
subsectors of their large-argument validity sectors.
Those expansions and their differentiated forms give
the stated remainder bounds uniformly on compact
parameter sets.
Near zero, $\Psi C(\arg z)$ and its inverse
are bounded on those parameter sets.

The coefficient conversion is
\begin{equation}\label{app:coefficientConversion}
\beta_{12}=ib_{21},
\qquad
\beta_{21}=-ib_{12},
\qquad
M_1^X=\mathcal A_1.
\end{equation}
This gives \eqref{app:model-betas}.
Using
\[
\Gamma(i\nu)\Gamma(-i\nu)
=\frac{\pi}{\nu\sinh(\pi\nu)},
\qquad
rk=e^{-2\pi\nu}-1,
\]
we obtain
\[
\beta_{12}\beta_{21}
=-\frac{2\pi e^{-\pi\nu}}
{rk\Gamma(i\nu)\Gamma(-i\nu)}
=\nu.
\]

The apparent singularities at $r=0$
and $k=0$ are removable. Indeed,
\[
\nu=-\frac{rk}{2\pi}+O((rk)^2),
\qquad
\frac1{\Gamma(z)}=z+O(z^2).
\]
Consequently,
\begin{equation}\label{app:axisValues}
\begin{aligned}
r=0:\quad&
\beta_{12}=0,\qquad
\beta_{21}=\frac{e^{i\pi/4}k}{\sqrt{2\pi}},
\\
k=0:\quad&
\beta_{12}=-\frac{e^{-i\pi/4}r}{\sqrt{2\pi}},
\qquad
\beta_{21}=0.
\end{aligned}
\end{equation}
The Gamma quotients have removable singularities
on the parameter axes, and $D_\alpha(z)$ is entire
in its order and argument.
Formula~\eqref{app:explicit-Weber-fundamental}
therefore extends holomorphically to those axes as well.

The uniform large-argument expansions
of the parabolic-cylinder functions
give \eqref{app:model-expansion}
and the stated remainder estimates.
Near zero, $\Psi C(\arg z)$ and its inverse
are uniformly bounded on compact
parameter sets. The complex powers
in $D_0$ therefore imply
\[
\|M^X(z)^{\pm1}\|
\le C_{\mathcal K}|z|^{-b},
\qquad b=|\Im\nu|.
\]
Together with the large-argument
expansions, this proves
\eqref{app:model-bound}.

The constructed solution has determinant
one. For any admissible solution,
the determinant has no jump, is bounded
at zero and tends to one at infinity;
thus it equals one by Liouville's theorem.
The quotient of two solutions has no
jump and is $O(|z|^{-2b})$ near zero.
Since $2b<1$, this isolated singularity
is removable, and uniqueness follows.
\end{proof}

\paragraph{Time-dependent normalization.}
For $T\ge1$, let
$
r_T=r_0T^{i\nu},
k_T=k_0T^{-i\nu},
$
where $(r_0,k_0)$ ranges over a compact
subset of $\mathcal U$.
Since $r_Tk_T=r_0k_0$, the parameter
$\nu$ is unchanged. Uniqueness gives
\begin{equation}\label{app:timeCovariance}
M^X(z;r_T,k_T)
=D_T^XM^X(z;r_0,k_0)(D_T^X)^{-1},
\end{equation}
where
\[
D_T^X=
\operatorname{diag}
(T^{i\nu/2},T^{-i\nu/2},1).
\]
In particular,
\[
\|D_T^X\|\|(D_T^X)^{-1}\|
=T^{|\Im\nu|}.
\]
Thus the model and its remainder estimates
acquire at most this single factor after
restoring the time-dependent diagonal.
This is the additional loss appearing
in the complex asymptotic analysis.

For the negative stationary point,
we use the inverse-transpose model.

\begin{RHP}\label{RHP:app:model-cofactor}
Find a $3\times3$ matrix-valued function
$M^{X,c}(z;r,k)$ satisfying the following
properties.
\begin{itemize}

\item
$M^{X,c}$ is analytic in
$\mathbb C\setminus X$, has no poles,
and has continuous boundary values
on $X\setminus\{0\}$, where
\[
X=\bigcup_{j=1}^{4}X_j.
\]

\item
The boundary values satisfy the
jump condition
\[
M^{X,c}_+(z)
=M^{X,c}_-(z)V^{X,c}(z;r,k),
\qquad z\in X,
\]
where
\begin{equation}\label{app:cofactor-jump}
V^{X,c}(z;r,k)=
\begin{cases}
\begin{pmatrix}
1&0&0\\
-rz^{-2i\nu}e^{iz^2/2}&1&0\\
0&0&1
\end{pmatrix},
&z\in X_1,
\\[4mm]
\begin{pmatrix}
1&k\Delta_X^{-1}
z^{2i\nu}e^{-iz^2/2}&0\\
0&1&0\\
0&0&1
\end{pmatrix},
&z\in X_2,
\\[4mm]
\begin{pmatrix}
1&0&0\\
r\Delta_X^{-1}
z^{-2i\nu}e^{iz^2/2}&1&0\\
0&0&1
\end{pmatrix},
&z\in X_3,
\\[4mm]
\begin{pmatrix}
1&-kz^{2i\nu}e^{-iz^2/2}&0\\
0&1&0\\
0&0&1
\end{pmatrix},
&z\in X_4.
\end{cases}
\end{equation}
Here
\[
\Delta_X=1+rk=e^{-2\pi\nu},
\]
and the complex powers use the
prescribed branches of the logarithm.

\item
The asymptotic behavior at infinity is
\[
M^{X,c}(z;r,k)=I+O(z^{-1}),
\qquad z\to\infty.
\]

\item
At the origin, the sectorial
right prefactors
\[
M^{X,c}(z;r,k)\mathcal D_X(z)^T
\]
and their inverses are bounded.
In particular,
\[
\bigl(M^{X,c}(z;r,k)\bigr)^{\pm1}
=O(|z|^{-b}),
\qquad z\to0,
\]
in each sector, where
$b=|\Im\nu|$.

\end{itemize}
\end{RHP}

Its unique solution is
\begin{equation}\label{app:model-cofactor}
M^{X,c}(z;r,k)
=M^X(z;r,k)^{-T}
=I-\frac{M_1^X(r,k)^T}{z}
+O(z^{-2}).
\end{equation}
Indeed, inverse transposition preserves
the order of the right jump relation,
and
$
V^{X,c}=(V^X)^{-T}.
$
The determinant, weighted bounds and
uniform remainder estimates follow
from Lemma~\ref{lem:app:model-positive}.

\paragraph{Identification with the local parametrices.}
Take the actual parameters $r=r_c$,
$k=k_c$ of Section~6, including the
full transmission factor $Q(a)$.
Let
\[
w_+(\lambda)
=(\lambda-a)\sqrt{\frac a\lambda},
\qquad
w_-(\lambda)=-w_+(-\lambda),
\]
and
$
H(\lambda)=\Pi L(\lambda)\Pi,
$
where $L$ is defined in
\eqref{sec6:PCcomparison}.
Then
\begin{equation}\label{app:model-identification}
\begin{aligned}
M^{pc}_{1,0}(\lambda)
&=H(\lambda)
M^X(\sqrt{tF}\,w_+(\lambda);r_c,k_c)
H(\lambda)^{-1},\\
M^{pc}_{2,0}(\lambda)
&=H(-\lambda)^{-T}
M^{X,c}(-\sqrt{tF}\,w_-(\lambda);r_c,k_c)
H(-\lambda)^T.
\end{aligned}
\end{equation}
The sectorial continuations to the
actual curved lips are those of
Section~\ref{s:jump-contribution}.
Since
\[
w_+'(a)=w_-'(-a)=1,
\qquad H(a)=I,
\]
the leading residue matrices are
$
\mathcal A_1=M_1^X,
\mathcal A_2=(M_1^X)^T,
$
in agreement with
\eqref{sec6:localCoefficientMatrices}.
The remaining four stationary points
follow from the cyclic reduction.
The complete first correction under
the analytic conjugation is retained
in the local matching estimates;
only its residue contributes to the
linear Cauchy integral at zero.

	\bibliographystyle{plain}

\end{document}